\documentclass[11pt,a4paper]{article}
\usepackage{amsmath,amssymb,amsthm,mathtools}
\usepackage{graphicx}
\usepackage{xcolor}
\usepackage{enumitem}
\usepackage{booktabs}
\usepackage{caption}
\usepackage{subcaption}
\usepackage[margin=2.4cm]{geometry}
\usepackage{hyperref}
\hypersetup{colorlinks=true,linkcolor=blue!50!black,citecolor=blue!50!black,urlcolor=blue!50!black}
\graphicspath{{figures/}}

\newtheorem{theorem}{Theorem}[section]
\newtheorem{lemma}[theorem]{Lemma}
\newtheorem{proposition}[theorem]{Proposition}

\newtheorem{introthm}{Theorem}

\newtheorem{introcor}[introthm]{Corollary}
\theoremstyle{definition}
\newtheorem{definition}[theorem]{Definition}
\newtheorem{remark}[theorem]{Remark}
\newtheorem{question}[theorem]{Question}

\newcommand{\R}{\mathbb{R}}
\newcommand{\Z}{\mathbb{Z}}
\newcommand{\Ball}{\mathbb{B}}
\newcommand{\Sph}{\mathbb{S}}
\newcommand{\Ext}{\mathbb{E}}
\newcommand{\n}{\mathbf{n}}
\newcommand{\hc}{\hat c}
\newcommand{\Ostar}{O_\star}
\newcommand{\lst}{\lesssim_\star}
\newcommand{\cL}{\mathcal{L}}
\newcommand{\cB}{\mathcal{B}}
\newcommand{\cA}{\mathcal{A}}
\newcommand{\cQ}{\mathcal{Q}}
\newcommand{\cK}{\mathcal{K}}
\newcommand{\cH}{\mathcal{H}}
\newcommand{\cF}{\mathcal{F}}
\newcommand{\cX}{\mathcal{X}}
\newcommand{\cY}{\mathcal{Y}}
\newcommand{\sK}{\mathsf{K}}
\newcommand{\sG}{\mathsf{G}}
\newcommand{\sM}{\mathsf{M}}
\newcommand{\sU}{\mathsf{U}}
\newcommand{\sV}{\mathsf{V}}
\newcommand{\Id}{\mathrm{Id}}
\DeclareMathOperator{\sech}{sech}
\DeclareMathOperator{\csch}{csch}
\DeclareMathOperator{\artanh}{artanh}
\DeclareMathOperator{\arcosh}{arcosh}
\DeclareMathOperator{\reg}{reg}
\DeclareMathOperator{\av}{av}
\DeclareMathOperator{\diag}{diag}
\DeclareMathOperator{\Ree}{Re}

\title{Exterior free-boundary minimal surfaces of every topological type with at least two ends}
\author{Igor Kozyrev\\[4pt]
{\normalsize Independent University of Moscow}\\
{\normalsize Bolshoy Vlasyevskiy per.~11, Moscow 119002, Russia}\\
{\normalsize\texttt{ikozyrev@gmail.com}}}
\date{October 2026}

\begin{document}
\maketitle

\begin{abstract}
An exterior free-boundary minimal surface is a properly immersed minimal surface in the complement of the open unit
ball of $\R^3$ whose boundary lies on the unit sphere and meets it orthogonally. We prove that for all integers
$\ell\geq2$, $b\geq1$ and $\gamma\geq0$ there is a connected, orientable, properly embedded exterior free-boundary
minimal surface with $\ell$ ends, $b$ boundary components and genus $\gamma$. Its total curvature is finite, and its ends
are parallel catenoidal ends of multiplicity one. By the theorem of Mazet and Mendes, a surface of this kind with one end
is a rotational annulus. Hence the topological types of such embedded surfaces are exactly the annulus with one end and
all types with at least two ends.

There are two gluing constructions. For $\ell\geq3$ we stack $\ell$ sheets over a stable catenoidal annulus that meets the
sphere along a circle of radius $4/5$, and we join adjacent sheets by half-catenoidal bridges sitting on the sphere. The
sheets are small normal graphs over this annulus whose displacements contain different multiples of the Jacobi field of
the family of such annuli, so their boundary circles and the logarithmic growth of their ends differ slightly. For
$\ell=2$ we double a catenoidal annulus near the critical angle $\pi/4$, where it acquires a bounded rotationally
invariant Jacobi field; the distance to the critical annulus is coupled to the size of the bridges. In gluing
constructions inside the unit ball that double or stack a compact surface, the bridges themselves must balance the sheets,
which forces many bridges. Here the logarithmic growth of the ends absorbs this balancing. A fixed number of bridges
therefore suffices, and every topological type is reached, including the small ones.
\end{abstract}

\medskip\noindent{\small\textit{2020 Mathematics Subject Classification.} 53A10; 49Q05, 53C42.}

\tableofcontents

\section{Introduction}\label{sec:intro}

Let $\Ball$ be the closed unit ball of $\R^3$, $\Sph=\partial\Ball$ the unit sphere and $\Ext=\R^3\setminus\mathrm{int}\,\Ball$.

\paragraph{The objects.} An \emph{exterior free-boundary minimal surface} is a properly immersed minimal surface in
$\Ext$ whose boundary lies on $\Sph$ and meets $\Sph$ orthogonally (Definition~\ref{def:efbms}). Such surfaces are
noncompact: a compact minimal surface with boundary on $\Sph$ lies in $\Ball$ by the convex hull property. Their
Steklov-type spectral theory, which motivated this paper, is developed in~\cite{Kozyrev2026}. The simplest examples are
rotational. For every $c\in(0,1)$ the part of a catenoid that lies outside the ball and meets $\Sph$ orthogonally along
the circle of radius $c$ is an embedded annulus $\mathcal C_c$ with one catenoidal end (Section~\ref{sec:annuli}). The
exterior of the equatorial disk is the limit $c=1$ with a planar end. Mazet and Mendes~\cite[Theorem~1.2]{MazetMendes2022}
proved that these are the only examples with one regular end of multiplicity one. Embedded surfaces had been treated
earlier by Park and Pyo~\cite{ParkPyo2017}; Mazet and Mendes complete a boundary case omitted there and do not assume
embeddedness. So the first question is which other topological types occur.

\paragraph{The analogue in the ball.} For free-boundary minimal surfaces \emph{inside} the ball the corresponding
question was raised by Fraser and Li~\cite{FraserLi2014}: which compact orientable surfaces with boundary embed in $\Ball$
as free-boundary minimal surfaces? It was answered only recently. Partial answers came in several forms.
\begin{itemize}[nosep]
\item Eigenvalue optimization: genus zero with any number of boundary components~\cite{FraserSchoen2016}.
\item Gluing: two parallel disks joined by many boundary bridges~\cite{FolhaPacardZolotareva2017}; desingularization of
the critical catenoid and the equatorial disk~\cite{KapouleasLi2021}; tripling of the disk~\cite{KapouleasWiygul2023};
doubling~\cite{KapouleasMcGrath2023}.
\item Equivariant min-max: \cite{Ketover2016} and connected boundary with arbitrary genus~\cite{CarlottoFranzSchulz2022}.
\end{itemize}
The full answer is due to Karpukhin, Kusner, McGrath and Stern~\cite{KKMS2024}: every compact orientable surface with
boundary embeds in $\Ball$ as a free-boundary minimal surface. Their method is equivariant optimization of Steklov
eigenvalues. Each gluing construction above reaches only an asymptotic range of types, such as large genus or many
boundary components, because the sheets must be balanced by the bridges themselves.

\paragraph{Main result.} We write $(\ell,b,\gamma)$ for the \emph{type} of a surface: number of ends, number of boundary
components, genus (Definition~\ref{def:type}).

\begin{introthm}\label{thm:main}
For all integers $\ell\geq2$, $b\geq1$ and $\gamma\geq0$ there is a connected, orientable, properly embedded exterior
free-boundary minimal surface of type $(\ell,b,\gamma)$. Its total curvature is finite, and its ends are parallel
catenoidal ends of multiplicity one.
\end{introthm}

\begin{introcor}\label{cor:classification}
Let $\Sigma$ be a connected, properly embedded exterior free-boundary minimal surface with finite total curvature and
type $(\ell,b,\gamma)$. Then either $\ell=1$ and $(b,\gamma)=(1,0)$, or $\ell\geq2$. Every type with $\ell\geq2$, $b\geq1$,
$\gamma\geq0$, together with $(1,1,0)$, is realized.
\end{introcor}

\begin{proof}
A compact surface of this kind would lie in $\Ball$, so $\ell\geq1$. The boundary is nonempty by definition, so
$b\geq1$. Embedded ends of finite total curvature are regular and of multiplicity one~\cite{JorgeMeeks1983},
\cite[Proposition~1]{Schoen1983}, in the sense of Schoen used in~\cite[Section~2.2]{MazetMendes2022}. If $\ell=1$,
the theorem of Mazet and Mendes~\cite[Theorem~1.2]{MazetMendes2022} (see also~\cite{ParkPyo2017}) shows that $\Sigma$ is
rotationally symmetric.
A rotationally symmetric minimal surface is a piece of a catenoid or of a plane. Since $\Sigma$ meets $\Sph$ orthogonally along
circles, its axis passes through the origin, and the only exterior pieces with one end that meet $\Sph$ orthogonally are the
annuli $\mathcal C_c$, up to a reflection, and the exterior of an equatorial disk~\cite[Section~4]{MazetMendes2022}. All of
them have type $(1,1,0)$. The remaining types are
realized by Theorem~\ref{thm:main}.
\end{proof}

\paragraph{The two constructions.} Theorem~\ref{thm:main} follows from two gluing constructions, which are illustrated in
Figures~\ref{fig:doubling} and~\ref{fig:stacking}. Interactive three-dimensional versions of these pictures for all types
with $\ell\leq6$, $b\leq6$, $\gamma\leq4$, together with Surface Evolver approximations of surfaces of types $(3,1,0)$ and
$(3,1,1)$ at bridge scale $\tau=0.01$ with truncated ends, are available at \url{https://surfaces.igorkozyrev.com}. These
computations are not used in the proofs.

\begin{introthm}[Stacking, Section~\ref{sec:proof}]\label{thm:stacking}
Let $\ell\geq3$, $b\geq1$, $\gamma\geq0$. For every $\tau\in(0,\tau_0)$ there is a surface $\Sigma_\tau$ of type
$(\ell,b,\gamma)$ as in Theorem~\ref{thm:main} with the following properties; we do not claim that $\Sigma_\tau$ depends
continuously on $\tau$.
\begin{itemize}[nosep]
\item As $\tau\to0$, $\Sigma_\tau$ converges smoothly away from finitely many points of the circle
$\Sph\cap\mathcal C_{4/5}$ to $\mathcal C_{4/5}$ with multiplicity $\ell$.
\item Near those points the $\ell$ sheets are joined by $E=\ell+b+2\gamma-2$ half-catenoidal bridges of waist comparable
to $\tau$. Each bridge sits on $\Sph$, is cut in half by it, and joins two consecutive sheets.
\item There are $\gamma+b$ bridges between the first two sheets, $\gamma+1$ between the second and third, and one between
any two later consecutive sheets.
\end{itemize}
\end{introthm}

\begin{introthm}[Doubling, Section~\ref{sec:proof}]\label{thm:doubling}
Let $b\geq1$, $\gamma\geq0$. For every $d\in(0,d_0)$ there are a bridge scale $\tau=\tau(d)>0$ and a surface $\Sigma_d$ of
type $(2,b,\gamma)$ as in Theorem~\ref{thm:main} with the following properties.
\begin{itemize}[nosep]
\item Away from the bridges the surface consists of two sheets close to $\mathcal C_c$, $c=\sqrt{(1-d)/2}$.
\item The sheets are joined by $b$ half-catenoidal bridges on $\Sph$ and by $\gamma$ full bridges, that is,
catenoidal necks, far out on the ends. All bridges have waist comparable to $\tau$.
\item The parameters are coupled by $d\,|\log\tau|\to\Gamma$, where $\Gamma=2^{-1/2}(\gamma-1+b/2)$ if $\gamma\geq1$, and
$\Gamma=\Gamma_0=2^{-1/2}b/4$ if $\gamma=0$.
\item As $d\to0$ the base annulus tends to the \emph{critical annulus} $\mathcal C_{2^{-1/2}}$. It meets $\Sph$ along a circle
at polar angle $\pi/4$, and among the annuli $\mathcal C_c$ its end has the largest logarithmic growth
(Remark~\ref{rem:upsilon}).
\end{itemize}
\end{introthm}

\begin{figure}[t]
\centering
\begin{subfigure}{0.32\textwidth}\includegraphics[width=\linewidth]{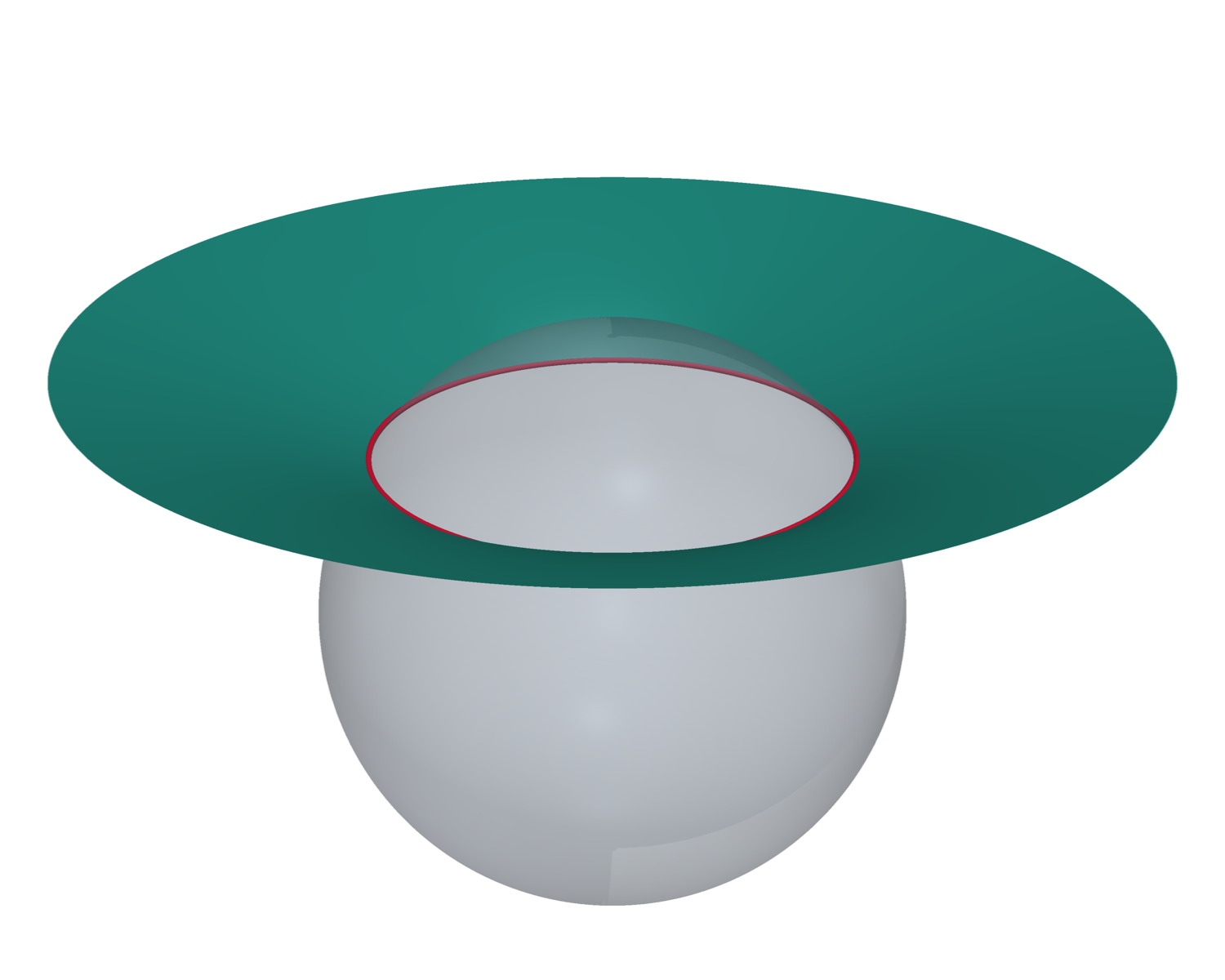}\caption*{$(1,1,0)$}\end{subfigure}
\begin{subfigure}{0.32\textwidth}\includegraphics[width=\linewidth]{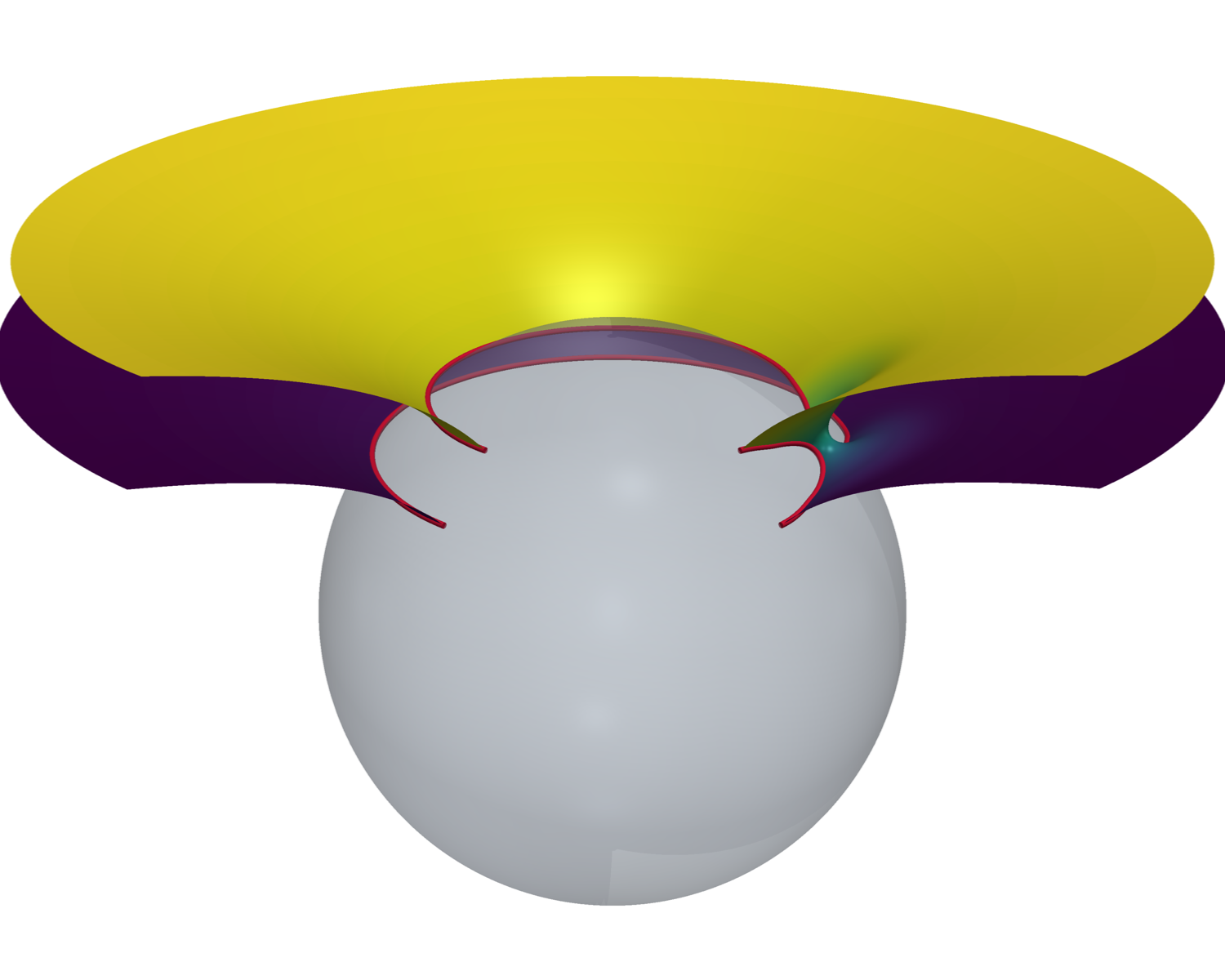}\caption*{$(2,1,0)$}\end{subfigure}
\begin{subfigure}{0.32\textwidth}\includegraphics[width=\linewidth]{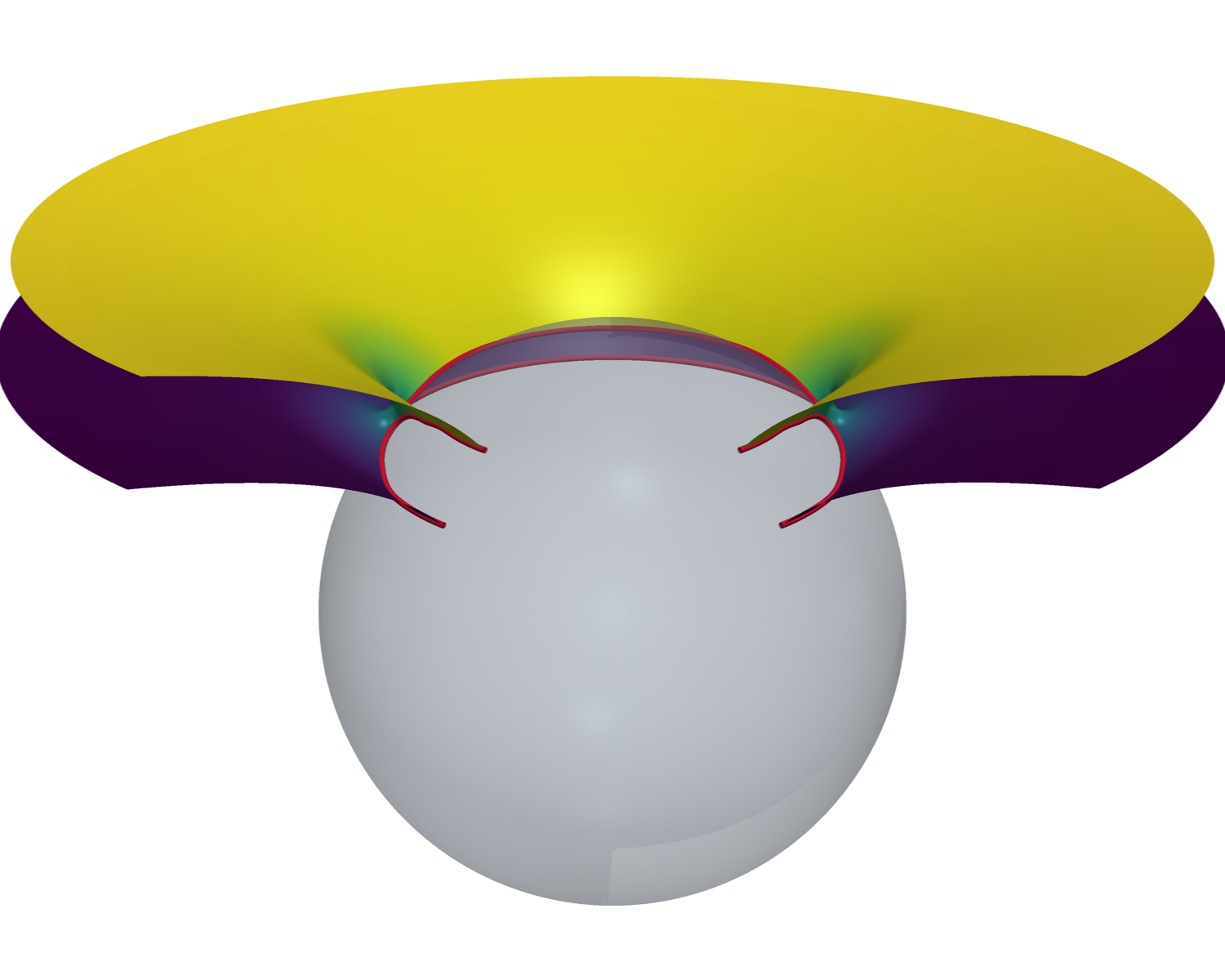}\caption*{$(2,2,0)$}\end{subfigure}\\[2pt]
\begin{subfigure}{0.32\textwidth}\includegraphics[width=\linewidth]{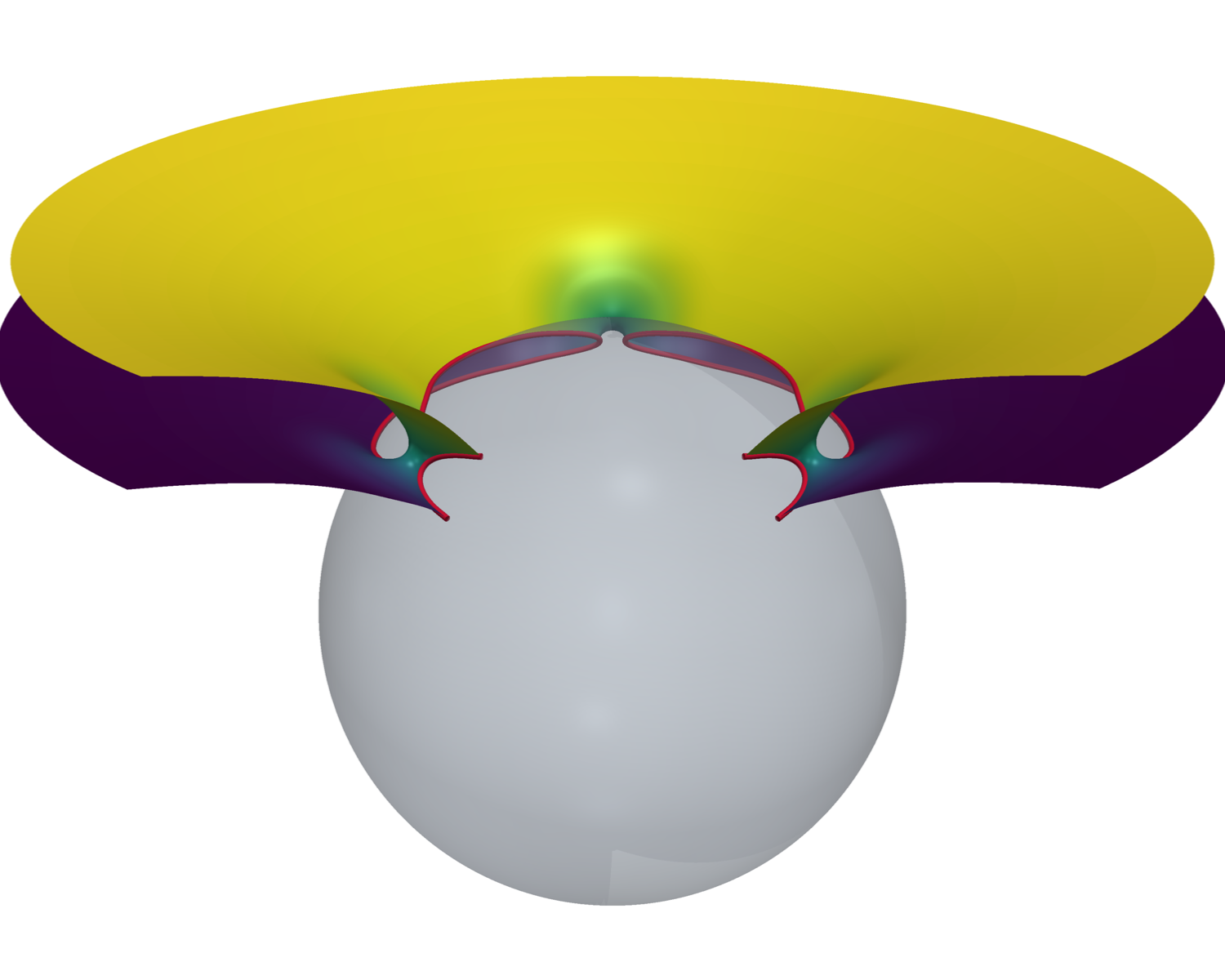}\caption*{$(2,3,0)$}\end{subfigure}
\begin{subfigure}{0.32\textwidth}\includegraphics[width=\linewidth]{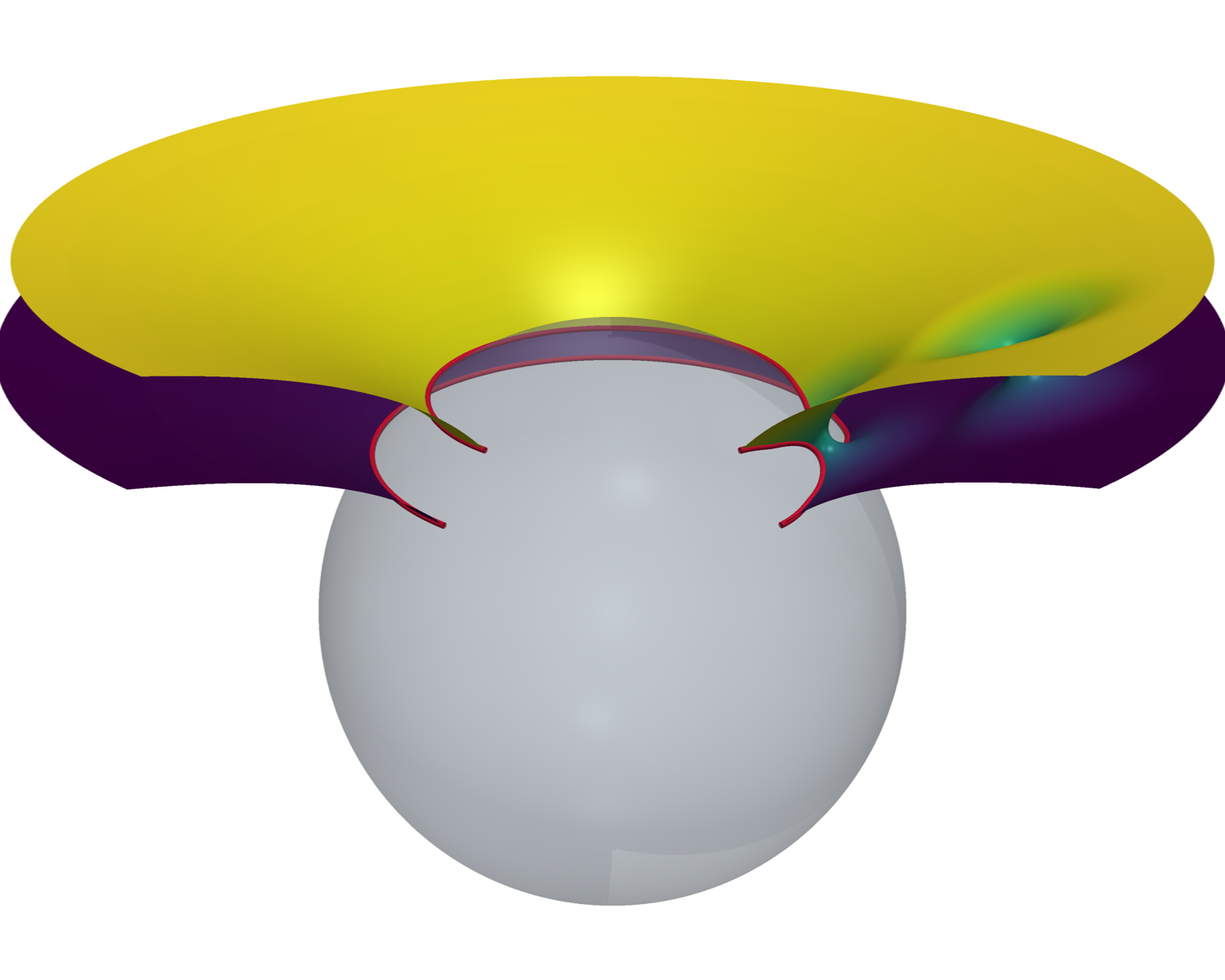}\caption*{$(2,1,1)$}\end{subfigure}
\begin{subfigure}{0.32\textwidth}\includegraphics[width=\linewidth]{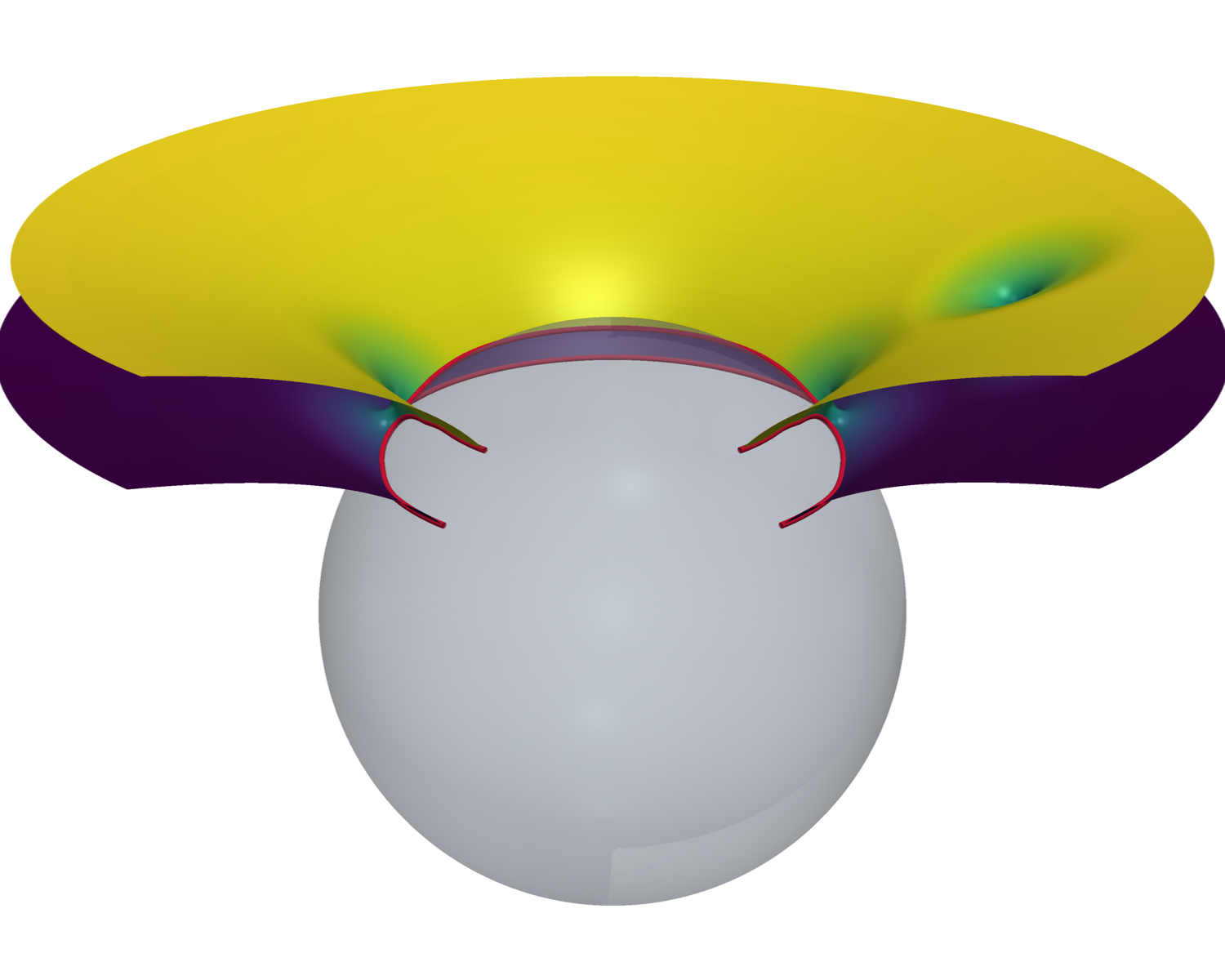}\caption*{$(2,2,1)$}\end{subfigure}
\caption{The annulus $\mathcal C_{4/5}$ of type $(1,1,0)$, and schematic pictures of the doubling construction of
Theorem~\ref{thm:doubling} for small types $(\ell,b,\gamma)$ with $\ell=2$. The two sheets are drawn as normal graphs
over $\mathcal C_c$, $c=0.7$, coloured by their normal displacement. A wedge is removed to show both sheets. Boundary curves on the unit
sphere are drawn in red. The half-catenoidal bridges sit on the sphere at the edges of the wedge; for $\gamma=1$ a full
catenoidal neck is visible on the upper sheet. Neck sizes and the position of the full neck are exaggerated (in the
construction the full necks lie at distance comparable to $d^{-10}$ from the origin). These are schematic pictures, not the approximate surfaces of Section~\ref{sec:pregluing} and not numerical solutions:
they are built from truncated logarithmic fields and catenoidal profiles over the base annulus, the boundary is placed on
the sphere, but neither the free-boundary condition nor minimality is imposed. For each picture the boundary curves were
assembled into closed curves and counted, and the genus was computed from the Euler characteristic; both agree with the
type.}
\label{fig:doubling}
\end{figure}

\begin{figure}[t]
\centering
\begin{subfigure}{0.40\textwidth}\includegraphics[width=\linewidth]{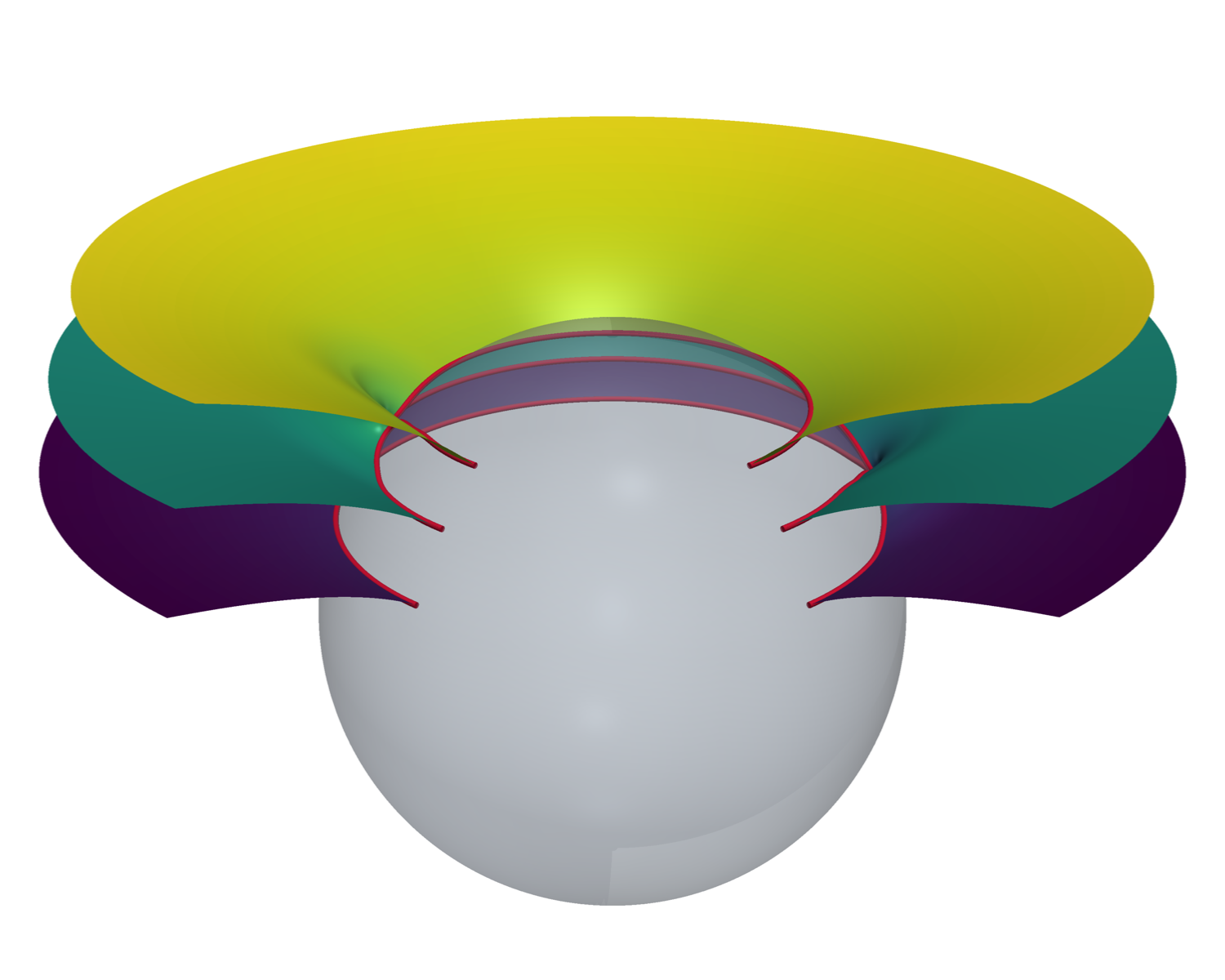}\caption*{$(3,1,0)$}\end{subfigure}
\begin{subfigure}{0.40\textwidth}\includegraphics[width=\linewidth]{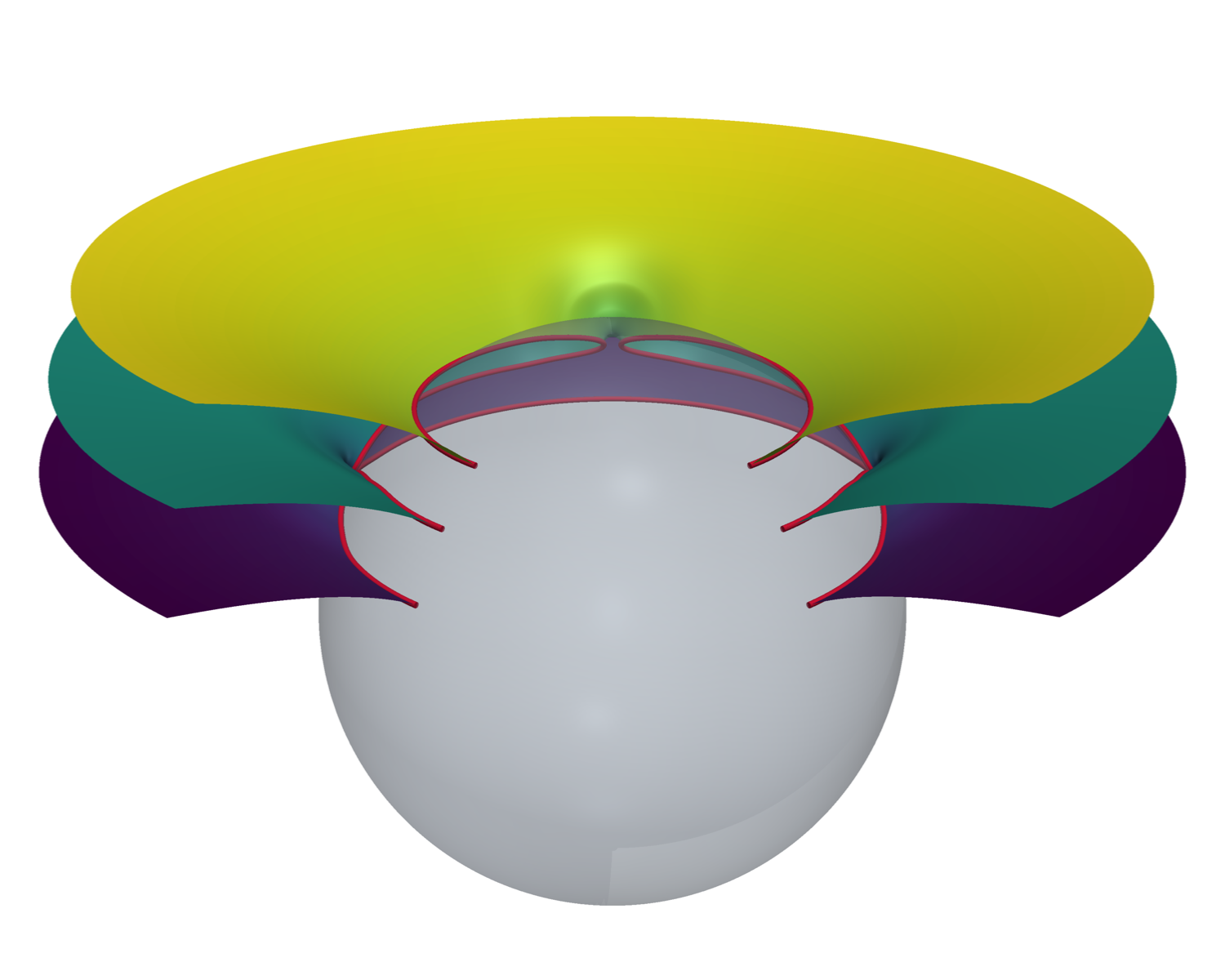}\caption*{$(3,2,0)$}\end{subfigure}\\[2pt]
\begin{subfigure}{0.40\textwidth}\includegraphics[width=\linewidth]{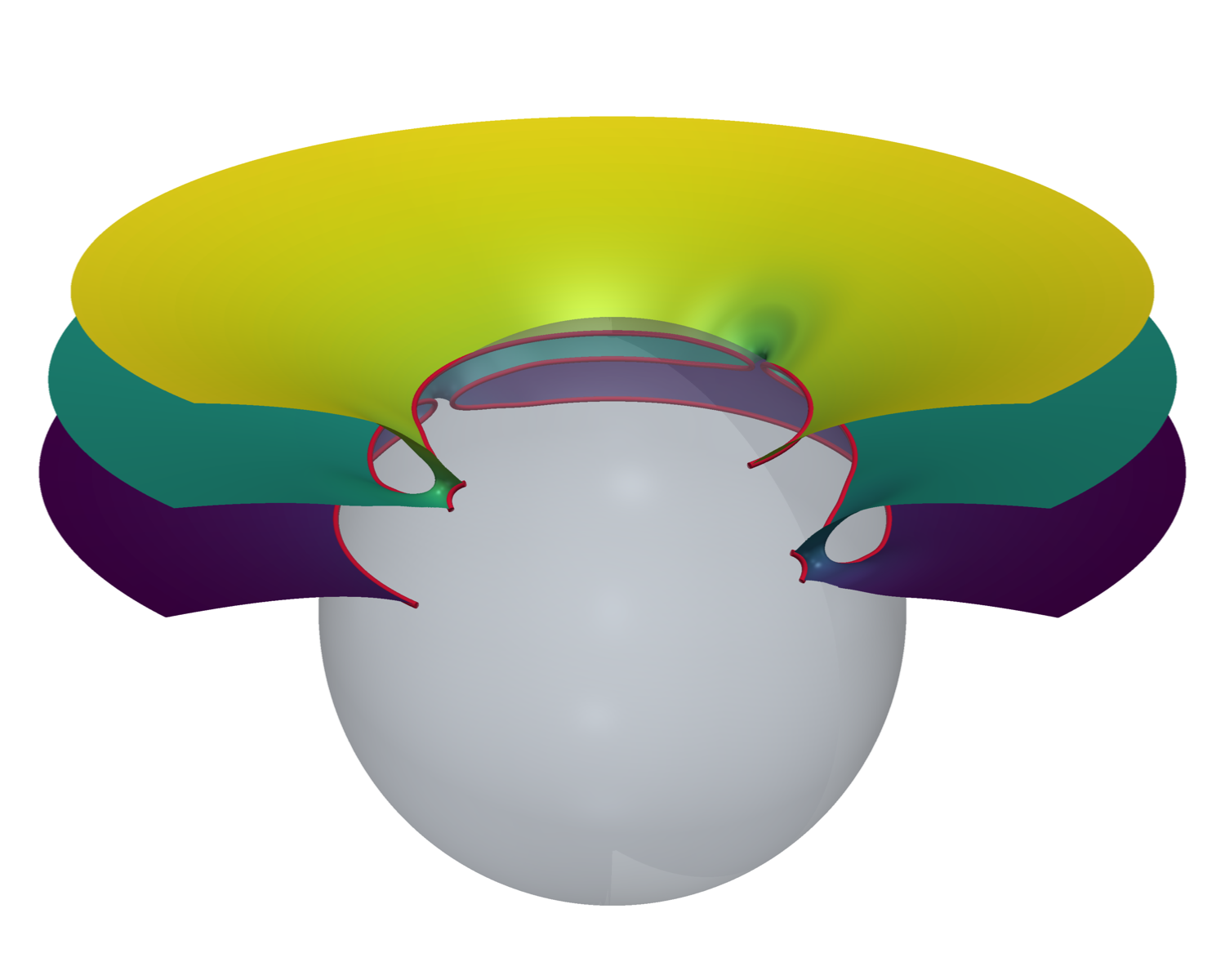}\caption*{$(3,1,1)$}\end{subfigure}
\begin{subfigure}{0.40\textwidth}\includegraphics[width=\linewidth]{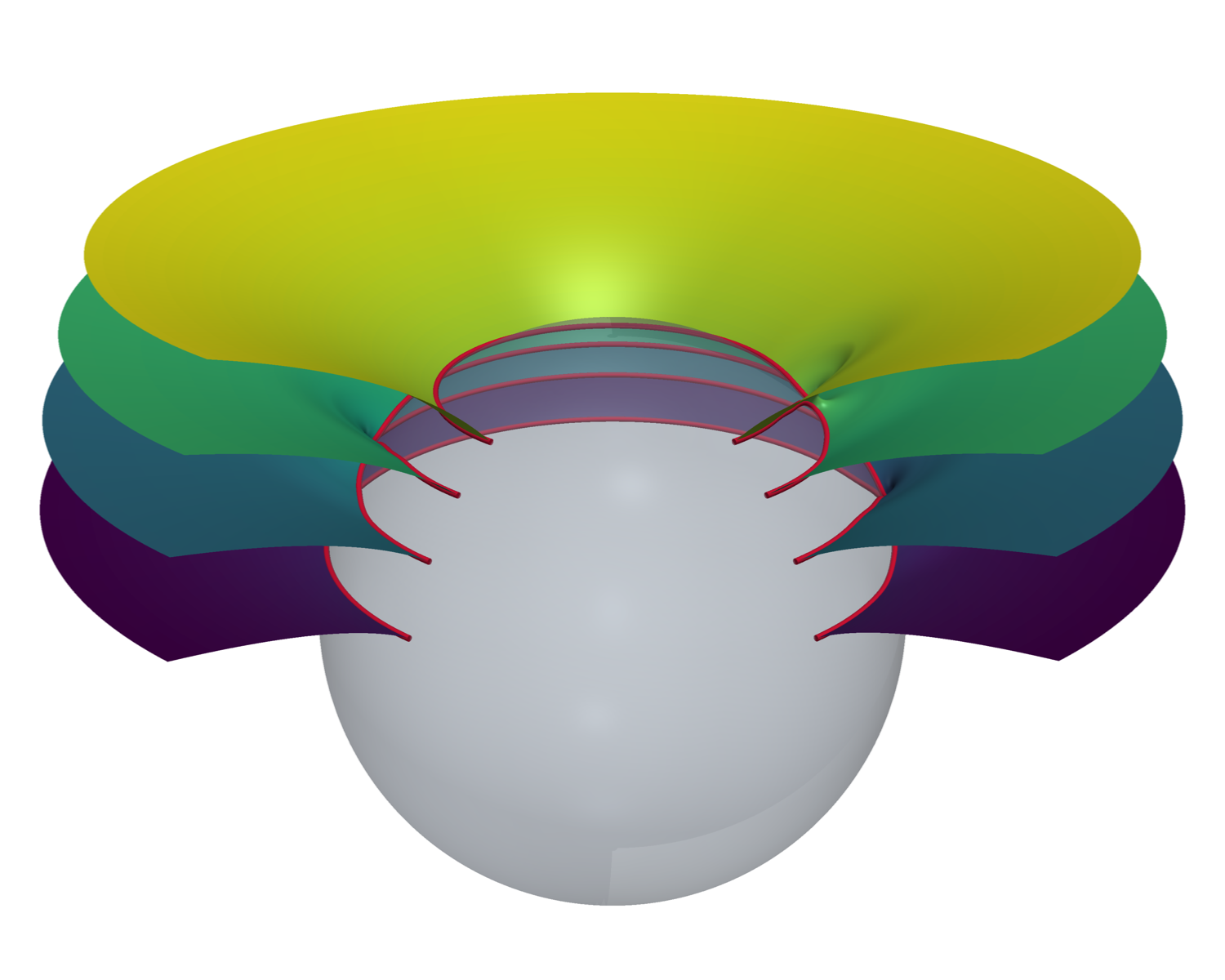}\caption*{$(4,1,0)$}\end{subfigure}
\caption{Schematic pictures of the stacking construction of Theorem~\ref{thm:stacking}: $\ell$ sheets drawn as normal
graphs over $\mathcal C_{4/5}$ and joined at the sphere by half-catenoidal bridges, with the same conventions as in
Figure~\ref{fig:doubling}; in particular the free-boundary condition is not imposed in these pictures. The sheets
diverge logarithmically at infinity. In $(3,1,1)$ the two bridges of the second
interface sit in different gaps between the two bridges of the first interface; this produces the handle.}
\label{fig:stacking}
\end{figure}

\paragraph{Why the ends help.} In the ball every sheet is compact. The separation of two nearby sheets is then a Green
function of the Jacobi operator with logarithmic poles at the bridges, and its mean must be produced by the poles
themselves. In the doubling and stacking constructions in the ball, large separation forces many bridges, which is one
reason why they reach only asymptotic ranges of types. Outside the ball the sheets have catenoidal ends, and the Jacobi operator of $\mathcal C_c$ has a
rotationally invariant Jacobi field $\Upsilon$ (Lemma~\ref{lem:radial}) that satisfies the free-boundary condition and grows
linearly in the cylindrical coordinate, that is, logarithmically in the distance. Adding multiples of $\Upsilon$ separates
the sheets by an amount that grows at infinity, so the ends become nested catenoidal ends with different logarithmic
growth. This absorbs the mean of the separation, and a fixed number of bridges suffices. Since $\Upsilon$ is, up to a
factor, the variation field of the family $c\mapsto\mathcal C_c$ (Remark~\ref{rem:upsilon}), the sheets are not
translates of one annulus: away from the bridges they are, to leading order, nearby members of this family, with
slightly different boundary circles on $\Sph$. In the stable case
$c=4/5$ the field $\Upsilon$ is positive and increasing; this gives the stacking. At the critical value $c=2^{-1/2}$ the
field $\Upsilon$ becomes bounded and the Robin Green function of the rotationally invariant mode has a factor $d^{-1}$;
coupling $d$ to $|\log\tau|$ balances the doubling.

A second difference concerns symmetries. Translations do not preserve the free-boundary problem outside the ball; the
Killing fields that do are the rotations. Every rotation gives a flux identity. For the rotations about horizontal axes the
contributions of the ends involve their dipole coefficients, which are free, while for the rotation about the axis of the
base annulus the contribution of every end vanishes. Thus there is no translational balancing condition, and the
finite-dimensional obstructions are the heights and tangential forces of the bridges together with one axial identity
(Sections~\ref{sec:finite} and~\ref{sec:nonlinear}).

\paragraph{Structure of the proof.} The proof is a gluing argument in the tradition of
Kapouleas~\cite{Kapouleas1997} and Traizet~\cite{Traizet2002}. Losses are tracked as powers of $s=|\log\tau|$.
\begin{enumerate}[nosep]
\item The finite matching problem for the heights and forces of the bridges is solved exactly, with a quantitative force
ball (Section~\ref{sec:finite}).
\item The half bridges are written in a conformal chart in which the sphere becomes a plane. An auxiliary metric makes
the sphere totally geodesic. Together these keep the free-boundary condition exact at every step
(Section~\ref{sec:pregluing}).
\item A projected linear problem on the broken approximate surface is inverted with a bound polynomial in $s$, and the
inverse is transferred to the actual surface (Section~\ref{sec:linear}).
\item Retaining the exact catenoidal tails gives an initial error of order $\tau^2$ up to polynomial factors. The
nonlinear correction and a Brouwer argument for the bridge forces then finish the construction
(Section~\ref{sec:nonlinear}).
\item Embeddedness, the topology and the ends are verified in Section~\ref{sec:proof}.
\end{enumerate}

\paragraph{What is not claimed.} The gluing scale depends on the type; no estimate uniform in the type is asserted. The
geometric arrangement is not prescribed either. In our surfaces all boundary curves are thin loops near one circle of
$\Sph$, and the ends are parallel, as they must be for embedded surfaces (we do not claim that they are coaxial). Which arrangements of the
boundary curves on $\Sph$ can occur is a different question; see Section~\ref{sec:remarks}.

\paragraph{Acknowledgements and preparation of the manuscript.} The constructions and the proofs in this paper were
developed with substantial assistance from AI systems (an OpenAI Codex agent produced the first version of the arguments,
and several independent AI-based audits checked the formulas symbolically and numerically). All statements and proofs
have been reorganized and checked by the author. The manuscript has not yet been verified by an independent human
expert, and it should be read as a research draft.

\section{Definitions and conventions}\label{sec:setting}

\begin{definition}\label{def:efbms}
Let $\Sigma$ be a smooth connected orientable surface with compact nonempty boundary $\partial\Sigma$. An
\emph{exterior free-boundary minimal surface} is a smooth proper immersion $X\colon\Sigma\to\Ext$ with the following
properties.
\begin{enumerate}[nosep]
\item $X$ is minimal.
\item $X^{-1}(\Sph)=\partial\Sigma$.
\item Along $\partial\Sigma$ the outward unit conormal $\eta$ of $\Sigma$ equals $-X$.
\end{enumerate}
It is \emph{embedded} if $X$ is injective. We write $\n$ for a unit normal field, $A$ for the second fundamental form
with respect to $\n$, $A(V,W)=\langle D_VW,\n\rangle$, $H=\operatorname{tr}A$ for the mean curvature (the sum of the principal curvatures) and
$\cL_\Sigma=\Delta_\Sigma+|A|^2$ for the Jacobi operator.
\end{definition}

The third condition says that $\Sigma$ meets $\Sph$ orthogonally; the conormal points into the ball. We call the
surface \emph{exterior} even though only $X(\Sigma)\subset\Ext$ is required. The boundary condition of Jacobi fields is
derived in Lemma~\ref{lem:annulus}.

\begin{definition}\label{def:type}
Let $X\colon\Sigma\to\Ext$ be an exterior free-boundary minimal surface with finite total curvature $\int_\Sigma|A|^2<\infty$.
\begin{enumerate}[nosep]
\item $\Sigma$ is conformally equivalent to $\overline\Sigma\setminus\{p_1,\dots,p_\ell\}$, where $\overline\Sigma$ is a
compact Riemann surface with boundary and the $p_k$ are interior points; see~\cite{JorgeMeeks1983,Kozyrev2026}.
\item The number $\ell$ of punctures is the number of \emph{ends}, $b$ is the number of boundary components, and the
\emph{genus} $\gamma$ is the genus of the closed surface obtained from $\overline\Sigma$ by gluing a disk to every boundary
component. The triple $(\ell,b,\gamma)$ is the \emph{type}.
\item An end is \emph{regular of multiplicity one} if, after a rigid motion, it is the graph of a function $x_3=f(x_1,x_2)$
over the complement of a disk, with $f=\mu\log r+\zeta+O(r^{-1})$ and $r=|(x_1,x_2)|$. It is \emph{catenoidal} if $\mu\neq0$
and \emph{planar} if $\mu=0$. We call $\mu$ its \emph{growth}.
\item Regular ends are \emph{parallel} if their limiting tangent planes are parallel.
\end{enumerate}
\end{definition}

Two embedded regular ends with nonparallel limiting planes intersect, so the ends of an embedded surface are parallel.

\paragraph{Polynomial bookkeeping.} In Sections~\ref{sec:pregluing}--\ref{sec:proof} a configuration is fixed: a type
together with the finite data of Section~\ref{sec:finite}. A parameter $\tau\in(0,\tau_0)$ is then sent to zero and we put
\[
 s=|\log\tau|.
\]
For nonnegative quantities $Q_1$ and $Q_2$ we write $Q_1\lst Q_2$ if $Q_1\leq Cs^NQ_2$ with constants $C,N$ depending only on
the fixed configuration. These constants may be increased finitely many times. We write $\Ostar(Q_2)$ for any quantity of
absolute value $\lst Q_2$. The notation never hides a negative power of $\tau$. In the doubling the configuration depends
on the parameter $d$, which tends to zero with $\tau$. There the constants $C,N$ depend only on $(b,\gamma)$ and are uniform
for $d\in(0,d_0)$: every constant is bounded by a power of $d^{-1}$ (Lemma~\ref{lem:uniformgeometry}), and $d^{-1}\leq Cs$, so it is
again of the form $Cs^N$. We write $N_0,N_1,\dots$ for further fixed exponents of this kind. We write
$Q_1\asymp Q_2$ if $C^{-1}Q_2\leq Q_1\leq CQ_2$ with $C$ depending only on the configuration (no powers of $s$), and
$Q_1\sim Q_2$ if $Q_1/Q_2\to1$.

\paragraph{Scaled H\"older norms.} For a function $u$ on a domain of geometric scale $r>0$ (a disk or annulus of radius
comparable to $r$ in a chart, or a half-disk or half-annulus) and $k\geq0$, $\alpha\in(0,1)$, put
\[
 \|u\|_{k,\alpha;r}=\sum_{j=0}^{k}r^j\|D^ju\|_{\infty}+r^{k+\alpha}[D^ku]_{\alpha}.
\]
We write $C^{k,\alpha}_{;r}$ for the space with this norm, and we say that $u=\Ostar(Q)$ \emph{in} $C^{k,\alpha}_{;r}$ if
$\|u\|_{k,\alpha;r}\lst Q$; similarly $u=O(Q)$ in $C^{k,\alpha}_{;r}$ if $\|u\|_{k,\alpha;r}\leq CQ$. Products and compositions obey the
usual estimates after rescaling to unit size. At a straight boundary edge they are one-sided estimates. We use
$\alpha\in(0,1)$ fixed throughout.

\section{The rotational annuli and their Jacobi fields}\label{sec:annuli}

\begin{definition}\label{def:annulus}
For $c\in(0,1)$ put
\[
 \hc=\sqrt{1-c^2},\qquad t_0=\artanh c,\qquad d=1-2c^2,\qquad c_*=2^{-1/2},
\]
and for $t\geq t_0$, $\theta\in\R/2\pi\Z$,
\[
 X_c(t,\theta)=\bigl(\lambda(t)\cos\theta,\ \lambda(t)\sin\theta,\ \hc+c\hc\,(t-t_0)\bigr),\qquad \lambda(t)=c\hc\cosh t .
\]
The image $\mathcal C_c=X_c([t_0,\infty)\times\R/2\pi\Z)$ is the \emph{rotational annulus}. We write $w(t)=\tanh t$.
\end{definition}

The number $d$ is the \emph{criticality parameter}: $d<0$ for $c>c_*$ and $d>0$ for $c<c_*$. The two cases used below
are $c=4/5$, where $\hc=3/5$ and $d=-7/25$, and $c\uparrow c_*$, that is $d\downarrow0$.

\begin{lemma}\label{lem:annulus}
For every $c\in(0,1)$ the following hold.
\begin{enumerate}[nosep]
\item $X_c$ is a proper minimal embedding, conformal with metric $\lambda(t)^2(dt^2+d\theta^2)$. Its unit normal is
$\n=(-\sech t\cos\theta,-\sech t\sin\theta,\tanh t)$.
\item $|X_c|$ is strictly increasing in $t$ and equals one exactly on $\{t=t_0\}$. Along this circle the outward conormal
is $\eta=-X_c$. Thus $\mathcal C_c$ is an embedded exterior free-boundary minimal annulus, of type $(1,1,0)$, with one
catenoidal end of growth $\mu(c)=c\hc$.
\item $\lambda^2\cL_{\mathcal C_c}=\cL$, where $\cL=\partial_t^2+\partial_\theta^2+2\sech^2t$.
\item If $V$ is the variation field of a smooth family of exterior free-boundary minimal surfaces through $\mathcal C_c$,
its normal component $u=\langle V,\n\rangle$ satisfies $\cL u=0$ and the \emph{Robin condition}
\[
 \cB u:=\partial_tu-cu=0\qquad\text{at }t=t_0.
\]
\end{enumerate}
\end{lemma}

\begin{proof}
The map $X_c$ is a vertical translate of the catenoid of neck radius $c\hc$ in isothermal coordinates. This gives the
metric, the normal, and $|A|^2=2\lambda^{-2}\sech^2t$, hence the third assertion. At $t=t_0$ one has $\cosh t_0=\hc^{-1}$,
$\tanh t_0=c$, so $X_c(t_0,\theta)=(c\cos\theta,c\sin\theta,\hc)$ lies on $\Sph$. The outward conormal is
$-\partial_tX_c/|\partial_tX_c|=-(\tanh t_0\cos\theta,\tanh t_0\sin\theta,\sech t_0)=-X_c$. Further,
$\tfrac12\partial_t|X_c|^2=\lambda\lambda'+(\hc+c\hc(t-t_0))c\hc>0$ for $t\geq t_0$.

For the boundary condition, let $X_\epsilon$ be a family of exterior free-boundary minimal surfaces with
$\partial_\epsilon X_\epsilon|_{\epsilon=0}=V=u\n+Y$, $Y$ tangent. Along the boundary $\langle X,V\rangle=0$ and
$\langle X,\n\rangle=0$, so $Y$ is tangent to $\partial\Sigma$. Differentiating $\langle\n,X\rangle=0$ along the boundary
tangent $T$ gives $A(T,\eta)=0$. The variation of the normal is $-\nabla u-A(Y,\cdot)^\sharp$. Differentiating
$\langle\n_\epsilon,X_\epsilon\rangle=0$ therefore gives
\[
 0=\langle-\nabla u-A(Y,\cdot)^\sharp,X\rangle+u=\partial_\eta u+u ,
\]
since $X=-\eta$ and $A(Y,\eta)=0$. Finally $\eta=-\lambda(t_0)^{-1}\partial_t$ and $\lambda(t_0)=c$, so
$\partial_\eta u+u=0$ is $\partial_tu=cu$.
\end{proof}

\begin{lemma}[Rotationally invariant Jacobi fields]\label{lem:radial}
The functions
\[
 w=\tanh t,\qquad \Upsilon=d+\bigl(2c-d(t-t_0)\bigr)w,\qquad \Upsilon^\sharp=\bigl(1+c(t-t_0)\bigr)w-c
\]
satisfy $\cL w=\cL\Upsilon=\cL\Upsilon^\sharp=0$ and the identities
\begin{equation}\label{eq:wronskian}
 \Upsilon\,\partial_t\Upsilon^\sharp-\partial_t\Upsilon\,\Upsilon^\sharp=1,\qquad c\Upsilon+d\Upsilon^\sharp=w,\qquad
 \Upsilon\,\partial_tw-\partial_t\Upsilon\,w=d .
\end{equation}
Moreover $\cB\Upsilon=0$, $\Upsilon(t_0)=1$, $\cB w=d$, $\Upsilon^\sharp(t_0)=0$, and for $t\geq t_0$
\begin{equation}\label{eq:Upsilondiff}
 \Upsilon(t)w(t')-w(t)\Upsilon(t')=d\bigl\{w(t')-w(t)+(t'-t)w(t)w(t')\bigr\}.
\end{equation}
If $d<0$, then $\Upsilon\geq1$ and $\Upsilon$ is increasing on $[t_0,\infty)$.
\end{lemma}

\begin{proof}
The rotationally invariant solutions of $\cL u=0$ are spanned by $w$ (vertical translations) and $t\tanh t-1$
(dilations). We have $\Upsilon=-d(t\tanh t-1)+(2c+dt_0)w$ and $\Upsilon^\sharp=c(t\tanh t-1)+(1-ct_0)w$. The identities
are direct computations with $w'=1-w^2$, $w(t_0)=c$, $w'(t_0)=\hc^2$ and $2c^2+d=1$. For $d<0$ both summands of
$\Upsilon-d=(2c+|d|(t-t_0))w$ are positive and increasing.
\end{proof}

\begin{remark}\label{rem:upsilon}
The field $\Upsilon$ is the unique rotationally invariant Jacobi field with $\cB\Upsilon=0$ and $\Upsilon(t_0)=1$. Hence, up
to a factor, it is the normal component of the variation field of the family $c\mapsto\mathcal C_c$.
It grows like $-dt$, that is, logarithmically in the distance, with coefficient proportional to the derivative of the
growth $\mu(c)=c\hc$. At $c=c_*$ this derivative vanishes, $d=0$, and $\Upsilon=2c_*w$ is bounded. This is the bounded
rotationally invariant Jacobi field mentioned in the introduction. By~\cite[Proposition~4.1]{MazetMendes2022} the annulus
$\mathcal C_c$ is stable for $c\geq c_*$ and has index one for $c<c_*$; their parameter is $\alpha=\arccos c$, with threshold
$\pi/4$. We only use the explicit formulas.
\end{remark}

\section{Green functions and pole fields}\label{sec:green}

Throughout this section $c\in(0,1)$ with $d\neq0$. We use the half-cylinder $\{t\geq t_0\}$ with its flat conformal
coordinates. Points are written $p=(t_p,\theta_p)$, and $|\cdot|$ denotes the flat distance in these coordinates
(angles taken modulo $2\pi$).

\subsection{Fourier modes}

For $n\geq1$ put
\begin{equation}\label{eq:modes}
 \nu_n=n^2+2cn-d=(n+c)^2-\hc^2,\qquad y_n(t)=e^{-n(t-t_0)}\bigl(n+w(t)\bigr),\qquad
 \varkappa_n=\frac{n^2-2cn-d}{\nu_n},
\end{equation}
and $\tilde y_1(t)=\hc\{(1+c(t-t_0))\sech t+c\sinh t\}$. The functions $y_n\cos n\theta$ solve $\cL u=0$ and decay. One
checks $\cB y_n=-\nu_n$ at $t_0$, and that $\tilde y_1\cos\theta$ solves $\cL u=0$ with $\tilde y_1(t_0)=1$ and
$\cB\tilde y_1=0$. If $d\leq2c$, in particular if $d\leq1/2$, then $\nu_n\geq n^2$ for all $n\geq1$.

\begin{lemma}[Robin Green coefficients]\label{lem:green}
For $n\geq0$ let $G_n(t,t')$ be the solution of
\[
 (\partial_t^2-n^2+2\sech^2t)\,G_n(\cdot,t')=\delta_{t'},\qquad \cB G_n(\cdot,t')=0\text{ at }t_0,
\]
that is bounded as $t\to\infty$. Write $t_<=\min(t,t')$, $t_>=\max(t,t')$, $w_<=w(t_<)$ and $w_>=w(t_>)$. Then
\begin{align*}
 G_0(t,t')&=\frac{\Upsilon(t_<)\,w(t_>)}{d},\qquad
 G_1(t,t')=-\frac{\tilde y_1(t_<)\,y_1(t_>)}{\nu_1},\\
 G_n(t,t')&=-\frac{e^{-n|t-t'|}(n-w_<)(n+w_>)+\varkappa_ne^{-n(t+t'-2t_0)}(n+w(t))(n+w(t'))}{2n(n^2-1)}\qquad(n\geq2).
\end{align*}
For $n\geq2$, uniformly in $t,t'\geq t_0$,
\[
 G_n(t,t')=-\frac{e^{-n|t-t'|}+e^{-n(t+t'-2t_0)}}{2n}+O\Bigl(\frac{e^{-n|t-t'|}+e^{-n(t+t'-2t_0)}}{n^2}\Bigr).
\]
\end{lemma}

\begin{proof}
In each case the displayed function is a product of a solution satisfying the Robin condition at $t_0$ and a solution
bounded at infinity, normalized by the Wronskian. For $n=0$ the Wronskian of $\Upsilon$ and $w$ is $d$
by~\eqref{eq:wronskian}. For $n\geq2$ the solutions $e^{\pm n t}(n\mp w)$ have Wronskian $2n(n^2-1)$ up to sign, and
the reflection coefficient $\varkappa_n$ is fixed by $\cB G_n=0$. The uniform estimate follows from $|w|\leq1$,
$(n-w_<)(n+w_>)=n^2+O(n)$ and $\varkappa_n=1+O(1/n)$.
\end{proof}

The Green function of $\cL$ with the Robin condition and bounded behaviour at the end is
$\frac1{2\pi}\sum_{n\geq0}\iota_nG_n(t,t')\cos n(\theta-\theta')$, with $\iota_0=1$ and $\iota_n=2$ for $n\geq1$.

\begin{definition}[Pole fields]\label{def:poles}
For $\vartheta\in\R/2\pi\Z$ and an interior point $p=(t_p,\theta_p)$, $t_p>t_0$, put
\begin{align*}
 P_\vartheta(t,\theta)&=\frac{w(t)}{2d}-\sum_{n\geq1}\frac{y_n(t)}{\nu_n}\cos n(\theta-\vartheta),\\
 I_p(t,\theta)&=G_0(t,t_p)-\frac{w(t_p)\,\Upsilon(t)}{d}+2\sum_{n\geq1}G_n(t,t_p)\cos n(\theta-\theta_p).
\end{align*}
\end{definition}

\begin{lemma}\label{lem:poles}
In the sense of distributions on $\{t\geq t_0\}$,
\[
 \cL I_p=2\pi\delta_p,\quad \cB I_p=0;\qquad \cL P_\vartheta=0,\quad \cB P_\vartheta=\pi\delta_\vartheta\text{ on }\{t=t_0\}.
\]
The field $P_\vartheta$ is bounded on $\{t\geq t_0+1\}$. The sum of the nonzero Fourier modes of $I_p$ is bounded away from $p$, and for
$t>t_p$ its mean equals $w(t)-w(t_p)+(t-t_p)w(t)w(t_p)$, which grows linearly in $t$. Moreover $I_p-\log|(t,\theta)-p|$ is continuous near
$p$, and $P_\vartheta-\log|(t,\theta)-(t_0,\vartheta)|$ is continuous near the boundary point $(t_0,\vartheta)$. The
rotational mean of $I_p$ vanishes for $t<t_p$, and the rotational mean of $P_\vartheta$ is $w/(2d)$.
\end{lemma}

\begin{proof}
The equations follow from Lemma~\ref{lem:green} and the Fourier expansion of $\delta$. The subtracted multiple of
$\Upsilon$ in $I_p$ is a homogeneous solution with $\cB\Upsilon=0$; it makes the mean of $I_p$ equal to
$\Upsilon(t)(w(t_p)-w(t_p))/d=0$ for $t<t_p$. The local logarithms follow from the leading terms of the Fourier series:
with $(h,\phi)=(t-t_p,\theta-\theta_p)$ one has $\sum_n n^{-1}e^{-n|h|}\cos n\phi=-\log|1-e^{-|h|+\mathrm i\phi}|$. For $P_\vartheta$ note that $\partial_t\log|\cdot|$ on the
boundary line is $\pi\delta$ in the half-plane.
\end{proof}

\subsection{Regular values and the self force}

If near a point $p$ a function has the form $u=c_u\log|(t,\theta)-p|+v$ with a constant $c_u$ and $v$ continuous, we write
$\reg_pu=v(p)$ and, when $v$ is $C^1$, $\nabla\reg_pu=\nabla v(p)$. The same notation is used at a boundary point
$(t_0,\vartheta)$. There $v$ may contain a Lipschitz term proportional to $|\theta-\vartheta|$ (it corresponds to the term
$a r\mathfrak a(\phi)$ of Lemma~\ref{lem:gapexpansion}), and $\partial_\theta\reg_\vartheta u$ denotes the derivative at $\vartheta$ of the
part of $v(t_0,\cdot)$ that is odd in $\theta-\vartheta$.

\begin{lemma}[Self force]\label{lem:self}
Put
\[
 R_1(t)=2G_1(t,t)+1-\frac{\sech^2t}{4}-\sum_{n\geq2}\frac{\varkappa_ne^{-2n(t-t_0)}(n+w(t))^2}{n(n^2-1)} .
\]
Then $2G_1(t,t)=-w(t)-(t-t_0+c^{-1})\sech^2t$, and at an interior pole
\[
 \partial_t\reg_pI_p=\tfrac12\bigl(1+R_1'(t_p)\bigr),\qquad \partial_\theta\reg_pI_p=0 .
\]
\end{lemma}

\begin{proof}
The formula for $G_1(t,t)$ follows from the product $\tilde y_1y_1$, using $\hc e^{t_0}=1+c$ and $\nu_1=2c(1+c)$. Summing
the direct terms with $\sum_{n\geq2}[n(n^2-1)]^{-1}=1/4$ shows that the regular diagonal of the bounded Green field
$I_p+w(t_p)\Upsilon/d$ equals $\Upsilon(t)w(t)/d+R_1(t)$. This field is symmetric in its two arguments, so its derivative
in one argument at the diagonal is half the derivative of the diagonal. Subtracting the homogeneous term,
\[
 \partial_t\reg_pI_p=\frac{\Upsilon'w+\Upsilon w'}{2d}+\frac{R_1'}2-\frac{w\Upsilon'}{d}=\frac12+\frac{R_1'}{2}
\]
by~\eqref{eq:wronskian}. Reflection in $\theta_p$ gives the angular statement. Terms of the form $r^2\log r$, with $r=|(t,\theta)-p|$, in
the local expansion have vanishing first jet.
\end{proof}

\begin{lemma}[Physical radius]\label{lem:physradius}
Let $p$ be an interior point and $h=(h_t,h_\theta)=(t,\theta)-p$. Let $r$ be the length of the orthogonal projection of
$X_c(p+h)-X_c(p)$ to the tangent plane of $\mathcal C_c$ at $X_c(p)$. Then
\[
 \log r=\log|h|+\log\lambda(t_p)+\tfrac12w(t_p)\,h_t+O(|h|^2).
\]
\end{lemma}

\begin{proof}
The tangential part of the second derivative of a conformal immersion with factor $\lambda$ contributes
$\lambda'/(2\lambda)=w/2$ to the linear term; the normal component has no linear term at a tangent plane.
\end{proof}

\subsection{\texorpdfstring{Kernels on the boundary and in the limit $d\to0$}{Kernels on the boundary and in the limit d to 0}}

For $a>-1$ put
\[
 \cK_a(\theta)=-\sum_{n\geq1}\frac{\cos n\theta}{n+a},\qquad f_r(\theta)=\log|1-re^{\mathrm i\theta}|\quad(0<r\leq1).
\]
Then $\cK_0=f_1$, and
\begin{equation}\label{eq:Kintegral}
 \cK_a=f_1-a\int_0^1r^{a-1}f_r\,dr .
\end{equation}

\begin{lemma}[Boundary kernel]\label{lem:bkernel}
$P_\vartheta(t_0,\theta)=\cK^{(c)}(\theta-\vartheta)$, where
\[
 \cK^{(c)}=\frac{c}{2d}+\frac12\bigl(\cK_{c-\hc}+\cK_{c+\hc}\bigr).
\]
For $a\geq0$ one has $\cK_a''<0$ on $(0,2\pi)$. In particular $(\cK^{(4/5)})''<0$, and
$\cK_*=\frac12(\cK_0+\cK_{\sqrt2})$ satisfies $\cK_*''<0$.
\end{lemma}

\begin{proof}
At $t_0$ one has $y_n(t_0)/\nu_n=(n+c)/((n+c)^2-\hc^2)=\frac12[(n+c-\hc)^{-1}+(n+c+\hc)^{-1}]$, which gives the first
identity. Differentiating~\eqref{eq:Kintegral} twice and using $\int_0^1ar^{a-1}dr=1$ gives
\[
 \cK_a''=-a\int_0^1r^{a-1}\bigl(f_r''-f_1''\bigr)dr ,
\]
and, with $u=\cos\theta$,
\[
 f_r''-f_1''=\frac{(1-r)^2\{(1-r)^2+2r(1-u)(u+2)\}}{2(1-u)(1-2ru+r^2)^2}>0\qquad(0<r<1).
\]
Hence $\cK_a''<0$ for $a>0$. For $a=0$, $\cK_0''=f_1''=-\frac14\csc^2(\theta/2)<0$. For $c=4/5$ the indices are $1/5$
and $7/5$.
\end{proof}

\begin{lemma}[Limits as $d\to0$]\label{lem:kernellimit}
Let $d\downarrow0$, $c=\sqrt{(1-d)/2}$ and $T=t_0+10\log(1/d)$. For points $t=T+x$, $t'=T+x'$ with $x,x'$ in a fixed
compact set, put $\mathsf z=e^{x+\mathrm i\theta}$ and $\mathsf z'=e^{x'+\mathrm i\theta'}$. Then, with every fixed number
of derivatives,
\begin{align*}
 &I_{(T+x',\theta')}(T+x,\theta)\to\log|\mathsf z-\mathsf z'|-\log|\mathsf z'|\quad(\mathsf z\neq\mathsf z'),\qquad
 R_1(T+x)\to0,\qquad \partial_t\reg I\to\tfrac12,\\
 &P_\vartheta(T+x,\theta)-\frac{w(T+x)}{2d}=O(d^{10}),\qquad I_{(T+x,\theta)}(t_0,\vartheta)=O(d^{10}),
\end{align*}
the first two errors being $O((1+T)d^{20})$. Here $\reg I$ denotes $\reg_pI_p$ at $p=(T+x,\theta)$. On the boundary, $\cK^{(c)}-c/(2d)$ and its regular diagonal converge
smoothly to $\cK_*$ and its regular diagonal.
\end{lemma}

\begin{proof}
With $h=t-t'$ and $v=e^{-|h|+\mathrm i(\theta-\theta')}$ the series can be summed:
\begin{align*}
 \sum_{n\geq2}\frac{v^n}{n(n^2-1)}&=\Bigl(1-\frac v2-\frac1{2v}\Bigr)\log(1-v)+\frac{3v}4-\frac12,\\
 \sum_{n\geq2}\frac{v^n}{n^2-1}&=\frac12\Bigl\{\bigl(v^{-1}-v\bigr)\log(1-v)+1+\frac v2\Bigr\}.
\end{align*}
Inserting these into Lemma~\ref{lem:green} expresses $I_p$ off its pole as $\log|1-v|$ plus the mean
term $\mathbf1_{t>t'}\{w(t)-w(t')+(t-t')w(t)w(t')\}$, plus terms whose coefficients are $O((1+T)e^{-2T})$, plus the
reflected series, which is $O(e^{-2T})$ with all derivatives. With $h_+=\max(h,0)$ we have $h_++\log|1-v|=\log|\mathsf z-\mathsf z'|-\log|\mathsf z'|$. Since also
$e^{-2T}=O(d^{20})$, the first limit follows. The summed form bounds one-sided derivatives of every order on both
sides of $h=0$; away from $\theta=\theta'$ they agree, because the jump of the Green function is supported on the
diagonal. The other interior statements follow from Lemma~\ref{lem:self} and the explicit decay of $y_n(T+x)/\nu_n$.
On the boundary, $c\pm\hc\to c_*\pm c_*$, and the integral representation~\eqref{eq:Kintegral} is smooth in $a>-1$; this
also covers the slightly negative index $c-\hc<0$ for $d>0$. The regular diagonal of $\cK_a$ at $\theta=0$ is
$\gamma_{\mathrm E}+\psi(1+a)$, where $\gamma_{\mathrm E}$ is the Euler constant and $\psi$ the digamma function, and this
is smooth in $a$.
\end{proof}

\section{The finite matching problem}\label{sec:finite}

\subsection{Heights and forces}\label{subsec:matching}

The surfaces are built from finitely many sheets, each a normal graph over $\mathcal C_c$. Consecutive sheets are joined
by \emph{bridges}: catenoidal necks centred at interior points (\emph{full bridges}) or halves of catenoidal necks centred
at boundary points and cut by $\Sph$ (\emph{half bridges}). A bridge $e$ has a centre $p_e$ and a waist $a_e>0$. Let $v$
be half the difference of the normal heights of the two sheets that $e$ joins, the \emph{half gap}. Near $p_e$ the
catenoid of waist $a_e$ has half height $a_e\arcosh(r/a_e)=a_e\log(2r/a_e)+O(a_e^3/r^2)$ in terms of the physical radius
$r$. A first-order Jacobi field $v$ with a logarithmic pole of coefficient $a_e$ at $p_e$ can only match the bridge if
\begin{equation}\label{eq:matchingheuristic}
 v=a_e\log\frac{2r}{a_e}+\beta_e\cdot x+O(a_er^2(1+|\log r|)^{N_1})\qquad\text{near }p_e .
\end{equation}
Here $x$ is the position in the tangent plane, $\beta_e$ is a vector, the \emph{regular gradient}, and at a half bridge only
the component of $\beta_e$ along $\partial\Sigma$ is present. At a half bridge there is one more term of size $a_er$, made
explicit in Lemma~\ref{lem:gapexpansion}; it is forced by the curvature of $\Sph$. The constant term in~\eqref{eq:matchingheuristic} is the
\emph{height equation}. The regular gradient $\beta_e$ is not forced to vanish at the level of Jacobi fields. It is the
\emph{force} of the bridge, and it is removed only at the end, by a degree argument (Section~\ref{sec:nonlinear}). In
conformal coordinates Lemma~\ref{lem:physradius} converts~\eqref{eq:matchingheuristic} into the conditions below. They
involve the conformal regular values of Section~\ref{sec:green}, with physical factor $\lambda(t_p)$ at an interior centre
and $\lambda(t_0)=c$ at a boundary centre.

\subsection{Doubling with positive genus}\label{subsec:doubling}

Fix integers $b\geq1$ and $\gamma\geq1$. Let $d\in(0,1/2)$ be small and
\[
 c=\sqrt{(1-d)/2},\qquad T=t_0+10\log(1/d).
\]
Index the full centres by $i\in\Z/\gamma\Z$ and the half centres by $j\in\Z/b\Z$, and write
\[
 p_i=(T+x_i,\theta_i),\qquad \vartheta_j\in\R/2\pi\Z .
\]
The full waists are $\tau\omega_i$ and the half waists $\tau\kappa_j$, with dimensionless \emph{waist ratios}
$\omega_i,\kappa_j>0$. We normalize $\gamma^{-1}\sum_i\omega_i=1$ and $\sum_ix_i=0$. With a further real parameter
$\mathfrak b$ the \emph{doubling field} is
\begin{equation}\label{eq:doublingfield}
 u=\tau\Bigl\{\frac{\mathfrak b}{d}\,\Upsilon+\sum_{j}\kappa_jP_{\vartheta_j}+\sum_i\omega_iI_{p_i}\Bigr\}.
\end{equation}
The two sheets are the normal graphs of $-u$ and $u$, so $u$ is the half gap. We impose the reflection symmetry
$\theta\mapsto-\theta$, whose mirror passes through the centres with $\theta_0=0$ and $\vartheta_0=0$: $\omega_{-i}=\omega_i$, $x_{-i}=x_i$, $\theta_{-i}=-\theta_i$, $\kappa_{-j}=\kappa_j$, $\vartheta_{-j}=-\vartheta_j$.
Put $\varsigma=d\log\tau$.

\begin{definition}[Residuals]\label{def:residuals}
The scaled height residuals are
\[
 \cH_i=d\,\reg_{p_i}\frac u\tau+\varsigma\omega_i-d\omega_i\log\frac{2\lambda(T+x_i)}{\omega_i},\qquad
 \cH^\partial_j=d\,\reg_{\vartheta_j}\frac u\tau+\varsigma\kappa_j-d\kappa_j\log\frac{2c}{\kappa_j}.
\]
The normalized force residuals are
\[
 \cF_{i,t}=\frac{\partial_t\reg_{p_i}u}{\tau\omega_i}-\frac{w(T+x_i)}2,\qquad
 \cF_{i,\theta}=\frac{\partial_\theta\reg_{p_i}u}{\tau\omega_i},\qquad
 \cF^\partial_j=\frac{\partial_\theta\reg_{\vartheta_j}u}{\tau\kappa_j}.
\]
We keep one equation for each reflection orbit and only the angular equations allowed by the symmetry. The resulting
map is denoted $\mathcal M_d$. Its unknowns form the vector $\mathbf y=(\varsigma,\mathfrak b,\omega,\kappa,x,\theta,\vartheta)$ of
all data, and we write $F$ for the vector of the independent force residuals.
\end{definition}

By Lemma~\ref{lem:physradius}, the vanishing of $\cH_i$ is the constant term of~\eqref{eq:matchingheuristic}. The force
coordinates relate to the physical regular gradients by
\begin{equation}\label{eq:forcegradient}
 \beta_{p_i}=\frac{\tau\omega_i}{\lambda(T+x_i)}\,(\cF_{i,t},\cF_{i,\theta}),\qquad
 \beta_{\vartheta_j}=\frac{\tau\kappa_j}{c}\,\cF^\partial_j .
\end{equation}
Indeed, the linear term $\frac12w\,h_t$ of Lemma~\ref{lem:physradius}, multiplied by the coefficient $\tau\omega_i$ of the
logarithm, is subtracted in $\cF_{i,t}$. Without the term $-w/2$ the force map would be off by a quantity of order one.

\begin{proposition}[Uniform limit]\label{prop:limit}
On a fixed neighbourhood where all waists are positive and all centres are separated, $\mathcal M_d$ converges in every
$C^m$ norm, with error $O(d\log(1/d))$, to the following map. Let $\bar\kappa=b^{-1}\sum_j\kappa_j$ and
$\mathsf z_i=e^{x_i+\mathrm i\theta_i}$. The limiting heights are
\[
 \cH_i^0=\varsigma\omega_i+2c_*\mathfrak b+\tfrac12b\bar\kappa,\qquad (\cH^\partial_j)^0=\varsigma\kappa_j+\mathfrak b+\tfrac12c_*b\bar\kappa .
\]
The limits of the complex full force $\cF_{i,t}-\mathrm i\cF_{i,\theta}$ and of the boundary force are
\[
 \frac1{\omega_i}\Bigl(-\mathfrak b+\sum_{i'\neq i}\omega_{i'}\frac{\mathsf z_i}{\mathsf z_i-\mathsf z_{i'}}\Bigr),\qquad
 \frac1{\kappa_j}\sum_{j'\neq j}\kappa_{j'}\cK_*'(\vartheta_j-\vartheta_{j'}).
\]
\end{proposition}

\begin{proof}
Apply Lemma~\ref{lem:kernellimit}. In the heights, $\Upsilon(T+x)=2c_*+O(d\log(1/d))$, $w(T+x)=1+O(d^{20})$ and
$d\log\lambda(T+x)=O(d\log(1/d))$. At the boundary $\Upsilon(t_0)=1$, and the large constant of $P_\vartheta$ is $c/(2d)$;
multiplied by $d\kappa_{j'}$ it gives $\frac12c\kappa_{j'}$. For the radial forces, $\partial_t\Upsilon(T+x)/d=-1+o(1)$ and
$\partial_tw(T+x)/d=o(1)$, and Lemma~\ref{lem:self} cancels the target $w/2$ in the limit. Differentiating
$\log|\mathsf z_i-\mathsf z_{i'}|$ gives the real part of $\mathsf z_i/(\mathsf z_i-\mathsf z_{i'})$ in the radial direction and
minus its imaginary part in the angular direction. All remaining terms are controlled with all derivatives by
Lemma~\ref{lem:kernellimit}.
\end{proof}

\begin{theorem}[Invertibility of the limit]\label{thm:invertible}
Let $\Gamma=c_*(\gamma-1+\tfrac b2)$. The limiting map of Proposition~\ref{prop:limit} vanishes at
\begin{equation}\label{eq:limitpoint}
 \omega_i=1,\quad \kappa_j=c_*,\quad x_i=0,\quad \theta_i=\frac{2\pi i}{\gamma},\quad \vartheta_j=\frac{2\pi j}{b},\quad
 \mathfrak b=\frac{\gamma-1}2,\quad \varsigma=-\Gamma ,
\end{equation}
and its derivative there is invertible on the reflection-invariant space with the two normalizations above.
\end{theorem}

\begin{proof}
At~\eqref{eq:limitpoint}, $\bar\kappa=c_*$ and $\sum_{i'\neq i}\mathsf z_i/(\mathsf z_i-\mathsf z_{i'})=(\gamma-1)/2$ for the
$\gamma$-th roots of unity; the boundary forces vanish by symmetry. Hence all residuals vanish.

\emph{Counting.} Let $a_n=\lfloor n/2\rfloor+1$ and $o_n=\lfloor(n-1)/2\rfloor$ be the numbers of all reflection orbits
and of free orbits on $\Z/n\Z$. There are $a_\gamma+a_b$ height equations and $\gamma+o_b$ force equations. The unknowns
are:
\begin{itemize}[nosep]
\item $\varsigma$ and $\mathfrak b$;
\item $a_\gamma-1$ full waists and $a_b$ half waists;
\item $\gamma-1$ full shape variables;
\item $o_b$ boundary phases.
\end{itemize}
The two numbers agree.

\emph{Waists and means.} Differences of height equations on zero-mean waist variations equal $-\Gamma\Id$, and the
leading heights do not depend on positions. For the mean full height, the mean half height
and the mean full radial force, the derivative in $(\varsigma,\bar\kappa,\mathfrak b)$ is
\[
 \begin{pmatrix}1&b/2&2c_*\\ c_*&bc_*/2-\Gamma&1\\ 0&0&-1\end{pmatrix},\qquad\det=\Gamma>0 .
\]
For the last row we use two identities. At a regular polygon
$\sum_{i\neq i'}\Ree[\mathsf z_i/(\mathsf z_i-\mathsf z_{i'})]=(\gamma-1)/2$ for every $i'$, so zero-mean waist changes do not
change the mean radial force. The identity $\sum_i\sum_{i'\neq i}\mathsf z_i/(\mathsf z_i-\mathsf z_{i'})=\gamma(\gamma-1)/2$
holds for all separated configurations, so position changes do not change it either.

\emph{Full shapes.} Put
$\Phi_i(\mathsf z)=\frac{\gamma-1}{2\mathsf z_i}-\sum_{i'\neq i}(\mathsf z_i-\mathsf z_{i'})^{-1}$. At the regular polygon the
derivative of the full force in the positions differs from $D\Phi$ by the invertible multiplication by $-\mathsf z_i$. For
the variations $\delta\mathsf z_i=\mathsf z_i^k$, $1\leq k\leq\gamma$, a direct summation gives
\[
 (D\Phi\,\delta\mathsf z)_i=\tfrac12(k-1)(k-\gamma-1)\,\mathsf z_i^{k-2}.
\]
These vectors form a basis, and the kernel consists of $\delta\mathsf z_i=\mathsf h\,\mathsf z_i$ with $\mathsf h\in\mathbb C$, that is
dilation and rotation. Reflection leaves only real $\mathsf h$, and $\sum_ix_i=0$ removes it. The remaining real cokernel is the mean radial force, already treated. For
$\gamma=1$ the shape space is empty.

\emph{Boundary shapes.} By Lemma~\ref{lem:bkernel}, $\cK_*''<0$ away from $0$. The derivative of the boundary force has
quadratic form $\sum_{j<j'}\cK_*''(2\pi(j-j')/b)(\dot\vartheta_j-\dot\vartheta_{j'})^2$ in the infinitesimal phase changes
$\dot\vartheta_j$, whose kernel consists of rotations. Reflection excludes
them; for $b\leq2$ there are no boundary phases. In the limit the full and the boundary shape blocks do not interact.
Hence all blocks are invertible.
\end{proof}

\begin{theorem}[Exact doubling data]\label{thm:finitedoubling}
For every $b,\gamma\geq1$ and every sufficiently small $d>0$ there are exact reflection-invariant data at which all
height and force residuals vanish.
\begin{enumerate}[nosep]
\item They satisfy
\[
 \log\tau=-\frac\Gamma d+O(\log(1/d)),\qquad \omega_i=1+O(d\log\tfrac1d),\qquad \kappa_j=c_*+O(d\log\tfrac1d),
\]
\[
 \mathfrak b=\frac{\gamma-1}2+O(d\log\tfrac1d),\qquad x_i=O(d\log\tfrac1d),
\]
and the angles converge to the two regular polygons.
\item Keeping all heights matched, the $\gamma+\lfloor(b-1)/2\rfloor$ independent force residuals $F$ (Definition~\ref{def:residuals})
can be prescribed arbitrarily on the ball $|F|\leq d^2$. Every waist $a_e(F)$ satisfies $|\log(a_e(F)/a_e(0))|\leq Cd$.
\item All derivatives of the data in $F$ are bounded by powers of $d^{-1}$, after dividing waist derivatives by the
waist.
\end{enumerate}
\end{theorem}

\begin{proof}
On a fixed ball around~\eqref{eq:limitpoint} where waists stay positive and centres separated,
Proposition~\ref{prop:limit} gives uniform $C^2$ convergence and Theorem~\ref{thm:invertible} a fixed inverse bound at the
centre. On a sufficiently small fixed ball the Newton map built with the inverse of the limiting derivative is a
contraction with constant $1/2$, and its initial residual is $O(d\log(1/d))$. Hence there is a zero at distance
$O(d\log(1/d))$. The same argument gives a family of data $\mathbf y(F)$ with $\mathcal M_d(\mathbf y(F))=(0,F)$ on a fixed
force ball, with all derivatives in $F$ bounded. Thus $|\varsigma(F)-\varsigma(0)|+|\omega(F)-\omega(0)|+|\kappa(F)-\kappa(0)|\leq C|F|$. For $|F|\leq d^2$,
dividing the first difference by $d$ gives $|\log\tau(F)-\log\tau(0)|\leq Cd$, and positive lower bounds for $\omega_i$ and
$\kappa_j$ give the waist comparison. Derivatives of $\tau=e^{\varsigma/d}$ divided by $\tau$ cost powers of $d^{-1}$.
\end{proof}

From now on $\tau$ denotes the value at $F=0$, and $s=|\log\tau|$. By Theorem~\ref{thm:finitedoubling}, $d\cdot s\to\Gamma$, so
$d^{-1}\leq Cs$ and the force ball $|F|\leq d^2$ contains the ball of radius $s^{-2}/C$. Moreover, above all full centres
the rotational mean $U$ of $u$ satisfies
\begin{equation}\label{eq:outgoing}
 \partial_t\Bigl(\frac Uw\Bigr)=C_{\rm out}\bigl(1+\csch^2t\bigr),\qquad
 C_{\rm out}=\tau\Bigl(\sum_i\omega_iw(T+x_i)-\mathfrak b\Bigr),\qquad \frac{C_{\rm out}}\tau\to\frac{\gamma+1}2 .
\end{equation}
Indeed, for $t>\max_i(T+x_i)$ the mean is $\tau\{\Upsilon(\mathfrak b-\sum_i\omega_iw(T+x_i))+w(\frac12b\bar\kappa+\sum_i\omega_i\Upsilon(T+x_i))\}/d$,
and $\partial_t(\Upsilon/w)=-d/w^2$ by~\eqref{eq:wronskian}.

\subsection{Doubling with genus zero}

\begin{proposition}[Dihedral data]\label{prop:genuszero}
Let $b\geq1$, $d>0$ small, $c=\sqrt{(1-d)/2}$, $\Gamma_0=bc_*/4$, $\vartheta_j=2\pi j/b$ and
\[
 E_b=\gamma_{\mathrm E}+\tfrac12\Bigl\{\psi\Bigl(1+\frac{c-\hc}b\Bigr)+\psi\Bigl(1+\frac{c+\hc}b\Bigr)\Bigr\},\qquad
 a=e^{-\Gamma_0/d},
\]
\[
 U_0=a\Bigl\{\log\frac{2c}a-\log b-E_b\Bigr\},\qquad \mathfrak b=\frac{dU_0}a-\frac{bc}2,\qquad
 u=\frac{a\mathfrak b}d\,\Upsilon+a\sum_{j}P_{\vartheta_j}.
\]
Then every height equation holds exactly with waist $a$ at each $\vartheta_j$, all forces vanish by symmetry, and
\[
 \mathfrak b\to-\Gamma_0,\qquad U(t_0)=U_0\sim\frac{a\Gamma_0}d,\qquad \partial_t\Bigl(\frac Uw\Bigr)=-a\mathfrak b\bigl(1+\csch^2t\bigr)>0,
\]
where $U$ is the rotational mean of $u$.
\end{proposition}

\begin{proof}
Summing only Fourier indices divisible by $b$ and using $\sum_{n\geq1}(1/n-1/(n+v))=\gamma_{\mathrm E}+\psi(1+v)$, the
regular value of $\sum_jP_{\vartheta_j}$ at a centre is $\frac{bc}{2d}+\log b+E_b$. Hence
$\reg_{\vartheta_j}u=\frac{a\mathfrak b}d+a(\frac{bc}{2d}+\log b+E_b)=a\log\frac{2c}a$, which is the height equation.
Reflection through each centre is a symmetry, so the regular gradients vanish. Finally $d\log(1/a)=\Gamma_0$ and
$bc_*/2=2\Gamma_0$ give $\mathfrak b\to-\Gamma_0$, and the mean identity is computed as in~\eqref{eq:outgoing}.
\end{proof}

In this case we write $\tau=a$, so that $s=\Gamma_0/d$, and there are no forces.

\subsection{Stacking}\label{subsec:stacking}

Fix $\ell\geq3$, $b\geq1$, $\gamma\geq0$ and $c=4/5$, so that $d=-7/25$ and $\cK^{(4/5)}=-\frac{10}7+\frac12(\cK_{1/5}+\cK_{7/5})$.
We write $\cK=\cK^{(4/5)}$. There are $\ell$ sheets and $\ell-1$ interfaces; interface $i$ joins sheets $i$ and $i+1$
and carries $n_i$ half bridges, where
\begin{equation}\label{eq:counts}
 n_1=\gamma+b,\qquad n_2=\gamma+1,\qquad n_i=1\ (3\leq i\leq\ell-1),\qquad
 E=\sum_{i=1}^{\ell-1}n_i=\ell+b+2\gamma-2 .
\end{equation}
The bridges are $e=(i,\mathsf a)$, $1\leq\mathsf a\leq n_i$, with boundary angles $\vartheta_{i\mathsf a}$ and waists
$a_{i\mathsf a}$. Put
\[
 J_i=\sum_{\mathsf a}a_{i\mathsf a}P_{\vartheta_{i\mathsf a}}\ (1\leq i\leq\ell-1),\qquad J_0=J_\ell=0 .
\]
For heights $\alpha_1,\dots,\alpha_\ell$ with $\sum_k\alpha_k=0$ the \emph{sheet fields} and the half gaps are
\begin{equation}\label{eq:sheetfields}
 f_k=\alpha_k\Upsilon+J_{k-1}-J_k,\qquad v_i=\frac{f_{i+1}-f_i}2=h_i\Upsilon+J_i-\frac{J_{i-1}+J_{i+1}}2,\qquad
 h_i=\frac{\alpha_{i+1}-\alpha_i}2 .
\end{equation}
At a bridge $e=(i,\mathsf a)$ the height and the force residuals are
\[
 \cH_e=\reg_{\vartheta_{i\mathsf a}}v_i-a_{i\mathsf a}\log\frac{2c}{a_{i\mathsf a}},\qquad
 \cF_e=\frac{\partial_\theta\reg_{\vartheta_{i\mathsf a}}v_i}{a_{i\mathsf a}} .
\]
By Lemma~\ref{lem:bkernel},
\begin{equation}\label{eq:stackforce}
 a_e\cF_e=\sum_{\mathsf a'\neq\mathsf a}a_{i\mathsf a'}\cK'(\vartheta_{i\mathsf a}-\vartheta_{i\mathsf a'})
 -\frac12\sum_{j=i\pm1}\sum_{\mathsf a'}a_{j\mathsf a'}\cK'(\vartheta_{i\mathsf a}-\vartheta_{j\mathsf a'}).
\end{equation}
By Lemma~\ref{lem:physradius} with $\lambda(t_0)=c$, the physical regular gradient of the half gap at $e$, which is tangent to
$\partial\Sigma$, is
\begin{equation}\label{eq:forcegradientstack}
 \beta_e=\frac{a_e}c\,\cF_e .
\end{equation}
Multiply~\eqref{eq:stackforce} by $a_e$ and sum over all bridges. Since $\cK'$ is odd, every pair of bridges appears twice
with opposite signs, and
\begin{equation}\label{eq:forcesum}
 \sum_ea_e^2\cF_e=0\qquad\text{for every configuration.}
\end{equation}
This is the linearization of the axial Killing identity used in Section~\ref{sec:nonlinear}.

\begin{theorem}[Exact stacking data]\label{thm:finitestacking}
There are a ratio $\mathfrak q\in(0,1)$, angles $\vartheta^0_{i\mathsf a}$, a radius $\varepsilon_F>0$ and $\tau_0>0$ with the following
properties. Put $\varpi_i=\mathfrak q^{i-1}$. For every $\tau\in(0,\tau_0)$ and every $F=(F_e)_{e\neq(1,1)}$ with $|F|\leq\varepsilon_F$
there are data
\[
 a_{i\mathsf a}=\tau\varpi_ie^{\xi_{i\mathsf a}/s},\quad \sum_{\mathsf a}\xi_{i\mathsf a}=0,\qquad
 h_i=\tau\varpi_i(s+\hat h_i),\qquad \vartheta_{i\mathsf a},\quad \vartheta_{11}=0,
\]
such that $\cH_e=0$ for every bridge, $\cF_e=F_e$ for $e\neq(1,1)$, and $\cF_{(1,1)}=-a_{11}^{-2}\sum_{e\neq(1,1)}a_e^2F_e$.
\begin{enumerate}[nosep]
\item $\xi_{i\mathsf a}$ and $\hat h_i$ are $O(1)$, and $\vartheta_{i\mathsf a}-\vartheta^0_{i\mathsf a}=O(1/s+|F|)$; hence
$a_{i\mathsf a}=\tau\varpi_i(1+O(1/s))$ and $h_i=\tau\varpi_i(s+O(1))>0$.
\item The data depend smoothly on $F$, with derivatives of every fixed order bounded by powers of $s$.
\item The angles $\vartheta^0_{i\mathsf a}$ are the phases at $1/s=0$ and $F=0$ (see the proof). They have the following
structure. The $\vartheta^0_{1\mathsf a}$ differ by $O(\mathfrak q)$ from the vertices of a regular $n_1$-gon. The $\vartheta^0_{2\mathsf a}$ lie in $n_2$ distinct gaps between
consecutive first-interface angles. For $i\geq3$ the angle $\vartheta^0_{i1}$ lies strictly inside a gap between
consecutive centres of interface $i-1$.
\item Same-sheet centres stay separated: centres of interfaces $i$ and $i\pm1$ are uniformly apart.
\item The outgoing growth differences $\lim_{t\to\infty}(f_{i+1}-f_i)/t=-2dh_i=\frac{14}{25}h_i$ are positive.
\end{enumerate}
\end{theorem}

\begin{proof}
\emph{Phases at $1/s=0$.} Divide~\eqref{eq:stackforce} for interface $1$ by $\tau\varpi_1$ and
for interface $i\geq2$ by $\tau\varpi_{i-1}$. At $1/s=0$ the waists are exactly $\tau\varpi_i$ and the equations become
\[
 \sum_{\mathsf a'\neq\mathsf a}\cK'(\vartheta_{1\mathsf a}-\vartheta_{1\mathsf a'})
 -\frac{\mathfrak q}2\sum_{\mathsf a'}\cK'(\vartheta_{1\mathsf a}-\vartheta_{2\mathsf a'})=\text{RHS},
\]
\[
 \mathfrak q\sum_{\mathsf a'\neq\mathsf a}\cK'(\vartheta_{i\mathsf a}-\vartheta_{i\mathsf a'})
 -\frac12\sum_{\mathsf a'}\cK'(\vartheta_{i\mathsf a}-\vartheta_{i-1,\mathsf a'})
 -\frac{\mathfrak q^2}2\sum_{\mathsf a'}\cK'(\vartheta_{i\mathsf a}-\vartheta_{i+1,\mathsf a'})=\text{RHS}.
\]
The right-hand sides are the prescribed forces times the same factors. At $\mathfrak q=0$ the first block is the force
of the regular polygon. By Lemma~\ref{lem:bkernel} its derivative has the negative definite quadratic form
$\sum_{\mathsf a<\mathsf a'}\cK''(\vartheta_{1\mathsf a}-\vartheta_{1\mathsf a'})(\dot\vartheta_{1\mathsf a}-\dot\vartheta_{1\mathsf a'})^2$
in the infinitesimal phase changes $\dot\vartheta_{1\mathsf a}$, modulo rotations. With $\vartheta_{11}=0$ fixed and its
equation omitted, it is invertible. For $i\geq2$ and $\mathfrak q=0$ the
equation says that $\vartheta_{i\mathsf a}$ is a critical point of $\sum_{\mathsf a'}\cK(\theta-\vartheta_{i-1,\mathsf a'})$.
On every gap between consecutive centres of interface $i-1$ this function is strictly concave and tends to $-\infty$ at
both ends, so it has a unique critical point, a nondegenerate maximum. We choose the gaps as in item~(3). The derivative
in all phases is block triangular with invertible diagonal blocks. The implicit function theorem fixes a sufficiently
small $\mathfrak q>0$, which is kept fixed from now on. It also gives nondegenerate phases for all small prescribed
forces.

\emph{Heights.} Divide $\cH_e$ by $\tau\varpi_i$. Then $\log(2c/a_{i\mathsf a})=s+\log(2c/\varpi_i)-\xi_{i\mathsf a}/s$,
and $s(e^{\xi/s}-1)$ extends smoothly to $1/s=0$. At $1/s=0$ the residual is
\[
 \hat h_i-\xi_{i\mathsf a}+\mathsf C_{i\mathsf a}(\vartheta)-\log\frac{2c}{\varpi_i},
\]
where $\mathsf C_{i\mathsf a}$ is the regular value of the fixed kernel combination at the centre; it involves the fixed
ratios $\varpi_j/\varpi_i$. The derivative in $(\hat h_i,\xi_{i\mathsf a})$, with $\sum_{\mathsf a}\xi_{i\mathsf a}=0$, is an
isomorphism: the mean of the equations determines $\hat h_i$ and the zero-mean part determines $\xi_{i\cdot}$.

\emph{The full map.} At $1/s=0$ the force residuals do not depend on $(\hat h,\xi)$. The derivative of the full map in
$(\vartheta,\hat h,\xi)$ is therefore block triangular with invertible diagonal blocks, and the implicit function theorem in
the parameters $(1/s,F)$ gives the family with the stated bounds. The omitted force is determined
by~\eqref{eq:forcesum}. Item~(5) follows from $\Upsilon\sim-dt$ and the boundedness of $J_i$ at infinity.
\end{proof}

\begin{remark}\label{rem:distinctgaps}
The choice of distinct gaps in item~(3) carries the topology (Proposition~\ref{prop:topology}). In this construction a
gap contains at most one second-interface centre, since the maximum in a gap is unique. Combinatorially, bands placed in a
single gap would produce genus zero and $n_1+n_2-1$ boundary components.
\end{remark}

\section{The approximate surfaces}\label{sec:pregluing}

In this section a configuration is fixed: either the stacking data of Theorem~\ref{thm:finitestacking}, or the doubling
data of Theorem~\ref{thm:finitedoubling} or Proposition~\ref{prop:genuszero}, together with a force parameter $F$ in the
corresponding ball. The notation $\lst$ of Section~\ref{sec:setting} refers to this configuration, and all estimates are
uniform in $F$. We put
\[
 \rho=\sqrt\tau,\qquad R=\tfrac52\rho .
\]
All waists satisfy $a_e\asymp\tau$, with constants depending on the configuration. In the doubling these constants
are uniform in $F$ by Theorem~\ref{thm:finitedoubling}(2). The \emph{end region} of a sheet is $\{t\geq t_{\rm end}\}$, where
$t_{\rm end}=t_0+1$ in the stacking and $t_{\rm end}=T+2$ in the doubling; it lies beyond all centres.

\subsection{The auxiliary metric}

Fix a smooth radial function $\Omega>0$ on $\Ext$ with $\Omega(x)=|x|^{-1}$ for $|x|\leq2$ and $\Omega=1$ for $|x|\geq3$, and
put $g_A=\Omega^2g_E$, where $g_E$ is the Euclidean metric. In the collar $\{1\leq|x|\leq2\}$ we have
$g_A=d\varrho^2+g_{\Sph}$ with $\varrho=\log|x|$. This is a product metric, $\Sph$ is totally geodesic, and the reflection
$\varrho\mapsto-\varrho$, that is the inversion $x\mapsto x/|x|^2$, is an isometry near $\Sph$.

\begin{lemma}[Exact free boundary of auxiliary graphs]\label{lem:auxgraph}
Let $\Sigma\subset\Ext$ be a surface meeting $\Sph$ orthogonally along $\partial\Sigma$. Let $\n_A=\Omega^{-1}\n$ be its
$g_A$-unit normal and $\eta_A$ its $g_A$-unit outward conormal.
\begin{enumerate}[nosep]
\item Along $\partial\Sigma$ the vector $\n_A$ is tangent to $\Sph$, $\eta_A$ is normal to $\Sph$, and the second
fundamental form satisfies $A(T,\eta_A)=0$ for the boundary tangent $T$.
\item Let $u$ be a $C^{2}$ function on $\Sigma$ with $\partial_{\eta_A}u=0$ on $\partial\Sigma$, small in $C^1$ relative to
the local scale. Then the surface $\Sigma_u=\{\exp^{g_A}_x(u(x)\n_A(x))\}$ meets $\Sph$ orthogonally along its boundary,
its boundary lies on $\Sph$, and $\Sigma_u\subset\Ext$.
\item For a Euclidean normal speed $f$ put $u=\Omega f$. On $\partial\Sigma$ the Robin condition $\partial_\eta f+f=0$ of
Lemma~\ref{lem:annulus} is equivalent to $\partial_{\eta_A}u=0$.
\end{enumerate}
\end{lemma}

\begin{proof}
(1) Orthogonality is conformally invariant, so $\eta_A=\pm\partial_\varrho$ and $\n_A$ is tangent to $\Sph$.
Differentiating $\langle\n_A,\partial_\varrho\rangle_{g_A}=0$ along $T$ gives $A(T,\eta_A)=0$, since $\partial_\varrho$ is parallel
for the product metric and $\Sph$ is totally geodesic.

(2) At a boundary point the geodesic $t\mapsto\exp(t\n_A)$ starts tangent to the totally geodesic sphere and stays in it.
Let $E(t)=\exp_p(t\n_A(p))$ for a boundary point $p$. Differentiate $\exp_{x(v)}(u(x(v))\n_A(x(v)))$ along a curve $x(v)$
leaving the boundary in the direction $\eta_A$. Since $\partial_{\eta_A}u=0$, the result is the Jacobi field $\mathsf j$ along $E$
with $\mathsf j(0)=\eta_A$ and $\mathsf j'(0)=\nabla_{\eta_A}\n_A$. By (1), $\nabla_{\eta_A}\n_A=-k_\partial\eta_A$ for a function $k_\partial$.
In the product metric the curvature term $R(\partial_\varrho,E')E'$ vanishes, so $\mathsf j(t)=(1-k_\partial t)\eta_A$ is normal to $\Sph$. Hence the tangent plane of
$\Sigma_u$ at its boundary contains the normal of $\Sph$. For small relative $C^1$ size the coordinate $\varrho$ still
increases from $0$ into $\Sigma_u$.

(3) On $\Sph$ we have $\eta=-x$ and $\partial_\eta\Omega=1$, while $\partial_{\eta_A}=\Omega^{-1}\partial_\eta=\partial_\eta$
there. Thus $\partial_{\eta_A}(\Omega f)=\partial_\eta f+f$.
\end{proof}

\begin{lemma}[The normal projection preserves the normal line of $\Sph$]\label{lem:projection}
Let $\Sigma$ be as in Lemma~\ref{lem:auxgraph} and $\Pi$ its $g_A$-normal projection, defined on a tubular neighbourhood.
At a point $\exp(t\n_A(p))$ with $p\in\partial\Sigma$ the differential of $\Pi$ maps the normal line of $\Sph$ onto the
conormal line of $\Sigma$ at $p$:
\[
 d\Pi(\eta_A)=(1-k_\partial t)^{-1}\eta_A(p),
\]
with no tangential component, provided $|k_\partial t|<1/2$. Moreover $\Pi$ maps points of $\Sph$ to $\partial\Sigma$.
\end{lemma}

\begin{proof}
This is the computation in the proof of Lemma~\ref{lem:auxgraph}(2) for $u$ constant along the normal: the tubular
map $(p,t)\mapsto\exp_p(t\n_A)$ has differential $(1-k_\partial t)\eta_A$ in the conormal direction, and its restriction to
$\partial\Sigma\times\R$ parametrizes $\Sph$ locally.
\end{proof}

\subsection{The sphere chart}

Write $Y=(x,y,z)\in\R^3$ and put
\[
 \mathfrak d(Y)=1-x+\tfrac14|Y|^2,\qquad \Xi(Y)=e_1+\frac{Y-\frac12|Y|^2e_1}{\mathfrak d(Y)} .
\]

\begin{lemma}\label{lem:chart}
$\Xi$ is a conformal diffeomorphism from a neighbourhood of $0$ onto a neighbourhood of $e_1$, with $\Xi(0)=e_1$ and
$d\Xi(0)=\Id$. Moreover:
\begin{enumerate}[nosep]
\item $\Xi^*g_E=\mathfrak d^{-2}g_E$;
\item $|\Xi(Y)|^2-1=2x/\mathfrak d(Y)$, so $\Xi$ maps $\{x=0\}$ into $\Sph$ and $\{x>0\}$ into $\R^3\setminus\Ball$;
\item $\Xi(-x,y,z)=\Xi(x,y,z)/|\Xi(x,y,z)|^2$;
\item $\Xi^*g_A=\bigl((1+\tfrac14|Y|^2)^2-x^2\bigr)^{-1}g_E$ near $0$;
\item $\Omega(\Xi(Y))=1-x+O(|Y|^2)$;
\item $\Xi^{-1}(X)=2|X+e_1|^{-2}(|X|^2-1,\,2X_2,\,2X_3)$.
\end{enumerate}
\end{lemma}

\begin{proof}
All identities are direct computations; we checked them symbolically. Property (3) says that the reflection in
$\{x=0\}$ corresponds to the inversion in $\Sph$, so (4) is invariant under $x\mapsto-x$. Note that (4) gives
$\Xi^*g_A=(1+O(|Y|^2))g_E$.
\end{proof}

In the chart $\Xi$ the sphere is the plane $\{x=0\}$ and the exterior is $\{x\geq0\}$. A graph $z=f(x,y)$ over the half
plane $\{x\geq0\}$ meets $\Sph$ orthogonally exactly when $f_x=0$ on $\{x=0\}$. Rotations about axes through $0$ commute
with the inversion. Hence the maps $\Xi^{-1}\circ O\circ\Xi$, with $O$ a rotation fixing the line $\R e_1$ or mapping
$e_1$ to a nearby point of $\Sph$, preserve the plane $\{x=0\}$ and its normal line.

\subsection{Bridges}

We describe the base annulus and the two sheets adjacent to a bridge $e$ near its centre $p_e$ in a \emph{vertical
chart}. At a full centre this is a Euclidean tangent frame with third axis along $\n$. At a half centre
it is the chart $\Xi$, composed with a rotation, such that the centre corresponds to $Y=0$ and the plane $\{x=0\}$ to
$\Sph$. Let $\mathsf B_e$ be the graph of the base in this chart, let the chart have conformal factor $\mathfrak d_0^{-1}$
($\mathfrak d_0=\mathfrak d$ at a half centre and $\mathfrak d_0=1$ at a full centre), and put $\mathsf W=\sqrt{1+|D\mathsf B_e|^2}$. Let
$f_{\rm lo}$ and $f_{\rm up}$ be the sheet scalars of the two sheets joined by $e$ (Definition~\ref{def:pregluing}: $f_i$ and $f_{i+1}$
in the stacking, $-u$ and $u$ in the doubling). We define the \emph{linear vertical models}
\begin{equation}\label{eq:linmodels}
 v_e=\mathfrak d_0\mathsf W\,\frac{f_{\rm up}-f_{\rm lo}}2,\qquad m_e=\mathfrak d_0\mathsf W\,\frac{f_{\rm up}+f_{\rm lo}}2 .
\end{equation}
To first order, $\mathfrak d_0\mathsf W f$ is the vertical height of the normal displacement $f\n$, so $v_e$ and $m_e$ solve the vertical
linearized equation of the base, and they satisfy the Neumann condition at a half centre. They are linear objects: the
actual sheets, which are auxiliary normal exponential graphs, differ from the vertical graphs $\mathsf B_e+m_e\pm v_e$ by
$\Ostar(\tau^2+\tau^3/r)$ in $C^{4,\alpha}_{;r}$ on the annuli $r\asymp R$, by Lemma~\ref{lem:graphchange} below. In $m_e$ the
logarithmic poles of the bridge $e$ cancel: in the stacking $f_i+f_{i+1}=(\alpha_i+\alpha_{i+1})\Upsilon+J_{i-1}-J_{i+1}$, whose poles
belong to other interfaces and are separated from $p_e$, and in the doubling $m_e=0$ exactly. Hence $m_e$ is smooth near $p_e$ with
$C^{6,\alpha}$ norm $\Ostar(\tau)$, and $v_e$ is the \emph{half gap}. After the recentering of Lemma~\ref{lem:recenter} below the
smooth \emph{middle graph} $B_e=\mathsf B_e+m_e$ satisfies:
\begin{itemize}[nosep]
\item $B_e(0)=0$ and $DB_e(0)=0$;
\item $B_e$ has bounded $C^{6,\alpha}$ jets;
\item its mean-curvature defect is $\lst\tau^2$;
\item at a half centre $\partial_xB_e(0,y)=0$.
\end{itemize}
With $r=|(x,y)|$ we write
\begin{equation}\label{eq:middlejet}
 B_e=\tfrac12\mathfrak k_e(x^2-y^2)+\tfrac14\mathfrak t_e(x^2+y^2)+\Ostar(r^3)\ \text{in }C^{4,\alpha}_{;r},\qquad \mathfrak t_e=\Ostar(\tau^2),
\end{equation}
after a rotation of the horizontal coordinates. This defines the trace-free coefficient $\mathfrak k_e$; the trace
coefficient $\mathfrak t_e$ is small because the mean-curvature defect of $B_e$ is $\lst\tau^2$. At a half centre the mixed
coefficient vanishes by the Neumann condition, so no rotation is needed. Let
\[
 \mathrm{Cat}_a(\sigma,\phi)=(a\cosh\sigma\cos\phi,\ a\cosh\sigma\sin\phi,\ a\sigma),\qquad
 \n_a=(-\sech\sigma\cos\phi,\ -\sech\sigma\sin\phi,\ \tanh\sigma)
\]
be the catenoid of waist $a$ and its unit normal, and $\tilde r=a\cosh\sigma$. Put
\begin{gather*}
 \mathfrak a(\phi)=-\cos\phi-\phi\sin\phi,\qquad
 \mathfrak e(\phi)=\Bigl(\frac32+\frac{\pi^2}8-\frac{\phi^2}2\Bigr)\cos\phi+\frac32\phi\sin\phi,\\
 U_{\rm sph}=\mathfrak a(\phi)\cosh\sigma+\mathfrak e(\phi)\sech\sigma .
\end{gather*}

We write $\mathrm{Cat}_a^x,\mathrm{Cat}_a^y$ for the first two components of $\mathrm{Cat}_a$.

\begin{definition}[Bridges]\label{def:bridge}
With $a=a_e$ and $\mathfrak k=\mathfrak k_e$, the \emph{full bridge} is the image of $(\sigma,\phi)\in\R\times\R/2\pi\Z$
under
\[
 Y=\mathrm{Cat}_a+B_e(\mathrm{Cat}_a^x,\mathrm{Cat}_a^y)\,e_3+\tfrac14\mathfrak k\,a^2\tanh\sigma\cos2\phi\;\n_a .
\]
The \emph{half bridge} is the image of $(\sigma,\phi)\in\R\times[-\pi/2,\pi/2]$ under $\Xi(Y)$, where
\[
 Y=\mathrm{Cat}_a+B_e(\mathrm{Cat}_a^x,\mathrm{Cat}_a^y)\,e_3+a^2\bigl(U_{\rm sph}+\tfrac14\mathfrak k\tanh\sigma\cos2\phi\bigr)\n_a .
\]
\end{definition}

\begin{lemma}[Exact edges]\label{lem:halfedge}
At $\phi=\pm\pi/2$ one has $\mathfrak a'=\mathfrak e'=0$ and $(\cos2\phi)'=0$. The point $Y$ and its $\sigma$-derivative
lie in $\{x=0\}$, and its $\phi$-derivative is normal to $\{x=0\}$. The $x$-component of $Y$ is a positive function times
$\cos\phi$ for small $a$ and $r$. Consequently the half bridge meets $\Sph$ orthogonally along its two edges and lies in
$\Ext$.
\end{lemma}

\begin{proof}
Indeed $\mathfrak a'=-\phi\cos\phi$ and $\mathfrak e'=\frac12\phi\cos\phi-(\frac{\pi^2}8-\frac{\phi^2}2)\sin\phi$, and both
vanish at $\phi=\pm\pi/2$. At these angles the $x$-components of $\mathrm{Cat}_a$, $\partial_\sigma \mathrm{Cat}_a$, $\n_a$ and $\partial_\sigma\n_a$
vanish, so $Y,\partial_\sigma Y\in\{x=0\}$. Further, $\partial_\phi \mathrm{Cat}_a$ and $\partial_\phi\n_a$ are multiples of $e_1$. The
$\phi$-derivative of $B_e(\mathrm{Cat}_a^x,\mathrm{Cat}_a^y)$ is $\partial_xB_e\,\partial_\phi \mathrm{Cat}_a^x=0$, because $\partial_\phi \mathrm{Cat}_a^y=0$ there and
$\partial_xB_e=0$ on $\{x=0\}$. The $\phi$-derivative of the coefficient of $\n_a$ vanishes. Hence $\partial_\phi Y\parallel e_1$.
The $x$-component of $Y$ equals $\cos\phi\,(a\cosh\sigma-a^2\sech\sigma(U_{\rm sph}+\frac14\mathfrak k\tanh\sigma\cos2\phi))$ plus the
contribution of $B_e e_3$, which has zero $x$-component. Since $\Xi$ is conformal and maps $\{x=0\}$ to $\Sph$ and $\{x>0\}$ to the
exterior (Lemma~\ref{lem:chart}), the claims follow.
\end{proof}

Figure~\ref{fig:zoom} shows schematic pictures of a half bridge and a full bridge. The correction terms in
Definition~\ref{def:bridge} are chosen by the following identities on the unit neck. Let
$\cL_{\rm cat}=\partial_\sigma^2+\partial_\phi^2+2\sech^2\sigma$ be the Jacobi operator of $\mathrm{Cat}_1$ in the coordinates
$(\sigma,\phi)$, and let $Q(Y)=xY-\frac12|Y|^2e_1$ be the quadratic Taylor term of $\Xi$.

\begin{lemma}[Unit-neck identities]\label{lem:neckidentities}
\[
 q_{\rm sph}:=\langle Q(\mathrm{Cat}_1),\n_1\rangle=\Bigl(\sigma\sinh\sigma-\tfrac12\cosh\sigma+\tfrac12\sigma^2\sech\sigma\Bigr)\cos\phi,
 \qquad \cL_{\rm cat}\,q_{\rm sph}=2\cosh\sigma\cos\phi,
\]
\[
 \cL_{\rm cat}U_{\rm sph}=-2\cosh\sigma\cos\phi,\qquad
 \cL_{\rm cat}\Bigl(\tfrac12\cosh^2\sigma\tanh\sigma\cos2\phi+\tfrac14\tanh\sigma\cos2\phi\Bigr)=0 .
\]
\end{lemma}

\begin{proof}
Direct computation (we verified it symbolically). The first two identities show that $a^2U_{\rm sph}\n_a$ cancels the normal part
of the quadratic term of the sphere chart. The last identity shows that $\frac14\mathfrak k a^2\tanh\sigma\cos2\phi\,\n_a$
cancels the normal component of the trace-free quadratic part of $B_e$ evaluated on $\mathrm{Cat}_a$.
\end{proof}

\begin{lemma}[Curvature of the bridges]\label{lem:bridgeerror}
Let $C_0$ be a sufficiently large fixed constant.
\begin{enumerate}[nosep]
\item On unit patches of the cylinder $(\sigma,\phi)$, after division by the local scale $\tilde r=a\cosh\sigma$, the immersion
of the bridge differs from the centred catenoid $\mathrm{Cat}_a$ by $\Ostar(\tilde r)$ in $C^{4,\alpha}$; that is, the unscaled
difference is $\Ostar(\tilde r^2)$. Its normalized mean curvature satisfies $\|\tilde r^2H\|_{0,\alpha}\lst\tilde r^3+\tau^2\tilde r^2$.
\item (Exact tails.) For $C_0a\leq r\leq4\rho$, with $r$ the horizontal radius in the vertical chart, the two tails are graphs
\[
 f_\pm=B_e\pm(\mathfrak h_0+\mathfrak h_1)+\mathsf w_\pm,\qquad \mathfrak h_0=a\arcosh(r/a),\qquad \|\mathsf w_\pm\|_{4,\alpha;r}\lst a^2,
\]
where $\mathfrak h_1=ar\,\mathfrak a(\phi)$ at a half centre and $\mathfrak h_1=0$ at a full centre.
\item On the same range $\|r^2H\|_{0,\alpha;r}\lst ar^2+a^2+\tau^2r^2$. In particular the normalized curvature is $\lst\tau^2$ on
the whole bridge $\{r\leq4\rho\}$.
\end{enumerate}
The constants do not depend on the length of the neck.
\end{lemma}

\begin{proof}
(1) On a unit patch $\mathrm{Cat}_a$ and its derivatives of positive order are $O(\tilde r)$, and $|a\sigma|\leq C\tilde r$. Every
correction in Definition~\ref{def:bridge} has jets $O(\tilde r^2)$; for instance $a^2\cosh\sigma=a\tilde r\leq\tilde r^2$. The Taylor
expansion of $\Xi$ is $e_1+Y+Q(Y)+O(|Y|^3)$ with all derivatives. By Lemma~\ref{lem:neckidentities} the normal components of
the quadratic terms are cancelled at the linear level, and the trace coefficient $\mathfrak t_e$ in~\eqref{eq:middlejet}
contributes $\lst\tau^2\tilde r^2$. Tangential variations do not change the vanishing mean curvature of $\mathrm{Cat}_a$ to first
order. On a scale-$\tilde r$ patch the normalized curvature of an immersion $X$ is $\tilde r\,\mathbf H(X/\tilde r)$, where $\mathbf H$
denotes the mean curvature as a smooth function of the jets of the immersion. Its quadratic remainder in a displacement $V$
is bounded by $C\|V\|^2_{2,\alpha}/\tilde r=O(\tilde r^3)$, and the cubic jets contribute $O(\tilde r^3)$.

(2) At a half centre the horizontal radius of the bridge is
\[
 r=\tilde r-a^2\mathfrak a-a^4\mathfrak e\,\tilde r^{-2}-\tfrac14\mathfrak k\,a^3\tilde r^{-1}\tanh\sigma\cos2\phi,\qquad\tilde r=a\cosh\sigma .
\]
At a full centre the terms with $\mathfrak a$, $\mathfrak e$ are absent. For $r\geq C_0a$ this map is positive and
increasing on both tails. Its inverse differs from the identity by $\Ostar(a^2)$ in $C^{4,\alpha}_{;r}$; this follows from the
implicit function theorem on a scaled annulus, where the small parameter is $a/r\leq C_0^{-1}$. Composing $B_e$ and $\mathfrak h_0$ with
the inverse changes them by $\Ostar(a^2r)$ and $\Ostar(a^3/r)$. The trace-free correction has vertical size $\Ostar(a^2)$. At a half
centre the remaining vertical correction is $\pm(a\tilde r\mathfrak a+a^3\mathfrak e/\tilde r)\sqrt{1-a^2/\tilde r^2}$, whose
difference from $\pm ar\mathfrak a$ is $\Ostar(a^3+a^3/r)\lst a^2$. All estimates hold with four scaled derivatives because they
are smooth functions of $a/r$ on a scaled annulus.

(3) For $\mathbf p\in\R^2$ let $T(\mathbf p)=\Id-\mathbf p\otimes\mathbf p/(1+|\mathbf p|^2)$, and put $E_0(f)=T(Df):D^2f$. The function
$\mathfrak h_0$ satisfies exactly
\[
 \mathfrak h_0'=\frac a{\sqrt{r^2-a^2}},\qquad \mathfrak h_0''=-\frac{ar}{(r^2-a^2)^{3/2}},\qquad E_0(\mathfrak h_0)=\frac{\mathfrak h_0''}{1+(\mathfrak h_0')^2}+\frac{\mathfrak h_0'}r=0 .
\]
With $\mathbf p=DB_e$ and $\mathbf q=D(\pm\mathfrak h_0)$ one has exactly
$E_0(B_e\pm\mathfrak h_0)-E_0(B_e)=[T(\mathbf p+\mathbf q)-T(\mathbf p)]:D^2B_e+[T(\mathbf p+\mathbf q)-T(\mathbf q)]:D^2(\pm\mathfrak h_0)$. In
scaled H\"older norms $\mathbf p=\Ostar(r)$, $D^2B_e=\Ostar(1)$, $\mathbf q=O(a/r)$ and $D^2\mathfrak h_0=O(a/r^2)$. Since $DT(0)=0$,
the two coefficient differences are bounded by $C(|\mathbf p||\mathbf q|+|\mathbf q|^2)$ and $C(|\mathbf p||\mathbf q|+|\mathbf p|^2)$. Multiplying by $r^2$ gives $\lst ar^2+a^2$, which is the full-centre estimate before adding $\mathsf w_\pm$.
At a half centre the exact graph equation in the chart $\Xi$ is, up to a positive factor,
\begin{equation}\label{eq:halfgraph}
 \mathcal E(f)=E_0(f)+\frac{f+(2-x)f_x-yf_y}{1-x+\frac14(x^2+y^2+f^2)} .
\end{equation}
Its first-order term at $B_e$ is $2\partial_x+\Ostar(r)D$. The term $\mathfrak h_1$ satisfies $\Delta \mathfrak h_1=-2a\cos\phi/r$, because
$\mathfrak a''+\mathfrak a=-2\cos\phi$, and therefore
\[
 \Delta \mathfrak h_1+2\partial_x\mathfrak h_0=\frac{2a^3\cos\phi}{r\sqrt{r^2-a^2}\,(r+\sqrt{r^2-a^2})},
\]
whose normalized size is $O(a^3/r)\lst a^2$. The remaining terms are bounded, after multiplication by $r^2$, by
$\Ostar(ar^2+a^2)$, using $D\mathfrak h_1=O(a)$, $D^2\mathfrak h_1=O(a/r)$ and $\mathfrak h_0=O(as)$. On these graph patches $|Df|+r|D^2f|$ is bounded,
so the normalized curvature is Lipschitz from scaled $C^{2,\alpha}$ heights to $C^{0,\alpha}$ and $\mathsf w_\pm$ contributes $\lst a^2$.
The middle defect contributes $\lst\tau^2r^2$. For $r\leq C_0a$ use (1). Since $a\asymp\tau$ and $r\leq4\rho$, all terms are
$\lst\tau^2$.
\end{proof}

\begin{remark}\label{rem:sharptail}
Replacing $\mathfrak h_0$ by its logarithmic approximation $a\log(2r/a)$ before estimating the curvature produces separate errors
of order $a^3/r^2$, which are of order $\tau^2$ only for $r\gtrsim\sqrt\tau$, and leads to weaker exponents. Keeping the exact
catenoidal tail retains the cancellation $E_0(\mathfrak h_0)=0$.
\end{remark}

\paragraph{Force profiles.} The finite data prescribe nonzero regular gradients $\beta_e$, of size
$|\beta_e|\lst a_e|F|$ by~\eqref{eq:forcegradient} and~\eqref{eq:forcegradientstack}. To match them we add first-harmonic profiles to the bridges. Fix an
even smooth function $\chi\geq0$, not identically zero, supported in $(-1,1)$, put $I_\chi=\int_\R\chi\sech^2\sigma\,d\sigma$ and
\[
 p_{\rm F}=\frac{\chi\sech\sigma}{\pi I_\chi},\qquad p_{\rm H}=2p_{\rm F} .
\]
Let $W_{\rm F}$ be the even solution of $W''+(2\sech^2\sigma-1)W=p_{\rm F}$ with $W(0)=W'(0)=0$, and $W_{\rm H}=2W_{\rm F}$. For $\sigma\geq1$,
$W_{\rm F}=(\sinh\sigma+\sigma\sech\sigma)/(4\pi)+c_{\rm F}\sech\sigma$ for a constant $c_{\rm F}$. We add the normal heights
\[
 4\pi a\,W_{\rm F}(\sigma)\bigl(\beta_{e,1}\cos\phi+\beta_{e,2}\sin\phi\bigr)\quad\text{(full)},\qquad
 2\pi a\,\beta_e\,W_{\rm H}(\sigma)\sin\phi\quad\text{(half)} ,
\]
where at a half centre $\beta_e$ is the tangential component. On the two tails these produce the vertical linear terms
$\pm\beta_e\cdot x$. The remaining mismatch is $\lst a|\beta_e|r+a^2|\beta_e|(1+|\sigma|)/r+|\beta_e|^2r$. The normalized
curvature they add is $\lst a|\beta_e|+|\beta_e|r^2+|\beta_e|^2r\lst\tau^2$; its first part is an explicit linear source supported in
the waist core $\{|\sigma|<1\}$. The half profile is odd in $\phi$ about $\phi=0$, has zero $\phi$-derivative at
$\phi=\pm\pi/2$, and preserves Lemma~\ref{lem:halfedge}.

\subsection{Local expansions and coordinate changes}

\begin{lemma}[Local expansion of the half gap]\label{lem:gapexpansion}
In the vertical chart at a centre $p_e$, the half gap $v_e$ of the sheet fields satisfies on annuli $R\leq r\leq s^{-N_0}/C$
\[
 v_e=a_e\log\frac{2r}{a_e}+\mathbf 1_{\rm H}\,a_er\,\mathfrak a(\phi)+\beta_e\cdot x+\Ostar\bigl(a_er^2(1+|\log r|)^{N_2+1}\bigr)\quad\text{in }C^{4,\alpha}_{;r},
\]
in particular with remainder $\Ostar(a_er^2)$ for $r\asymp R$, where $\mathbf1_{\rm H}=1$ at half centres and $0$ at full centres, and $\beta_e$ is the regular gradient
of~\eqref{eq:forcegradient} or~\eqref{eq:forcegradientstack}. At a half centre only the tangential component $\beta_e\,y$ is present.
\end{lemma}

\begin{proof}
Put $v^0_e=a_e\log r+a_e\log(2/a_e)+\mathbf1_{\rm H}a_er\mathfrak a(\phi)$, so the constant is the height equation. The
equation of the half gap in the vertical chart has principal coefficients $\Id+O(r^2)$, drift $2\partial_x+O(r)D$ at a half
centre (from~\eqref{eq:halfgraph}) and $O(r)D$ at a full centre, and bounded potential. Subtracting the logarithm from
the explicit fields of Section~\ref{sec:green} and using interior or smooth Neumann estimates on annuli gives the weak
bound $\|v_e-v^0_e\|_{j,\alpha;r}\lst a_er(1+|\log r|)^{N_1}$ for $j\leq6$. In the half case
$\Delta(a_er\mathfrak a)=-2a_e\cos\phi/r$ cancels the singular drift term $2\partial_x(a_e\log r)$. Hence $\tilde v=v_e-v^0_e$
satisfies $\Delta\tilde v=q_w$ with $\|q_w\|_{2,\alpha;r}\lst a_e(1+|\log r|)^{N_2}$, and there is no point mass at the centre since
$\tilde v=O(a_er|\log r|^{N_1})$. At a half centre $\tilde v_x=0$ on $\{x=0\}$; extend $\tilde v$ and $q_w$ evenly as a weak Poisson equation. The
logarithmic potential of $q_w$, minus its value and first derivative at the centre, is $\lst a_er^2(1+|\log r|)^{N_2+1}$ (split
the integral at $|y|=2r$ and use the bound $C|y|^{-2}$ for the second derivatives of the kernel outside). The value
at the centre vanishes by the weak bound, the first derivative is $\beta_e$ by definition of the regular gradient, and its
$x$-component vanishes at a half centre by the reflection. Finally return to the half-domain and apply one-sided smooth
Neumann estimates for four derivatives; the even extension itself is not claimed to be $C^{4,\alpha}$. On $r\asymp R$ every
logarithm is $\Ostar(1)$.
\end{proof}

\begin{lemma}[Recentering]\label{lem:recenter}
Let $\mathsf B$ be a graph with $\mathsf B(0)=D\mathsf B(0)=0$, at a half centre also $\mathsf B_x(0,y)=0$, and let $m$ be a smooth
displacement with $\|m\|_{C^{6,\alpha}}\leq\epsilon$, at a half centre also $m_x(0,y)=0$. Put $m_0=m(0)$ and
$\mathfrak p=m_y(0)$. There is a map $\mathcal T_e$ that moves $(0,0,m_0)$ to the origin and makes the tangent plane of $\mathsf B+m$
there horizontal. At a full centre $\mathcal T_e$ is a rigid motion. At a half centre $\mathcal T_e=\Xi^{-1}\circ \mathsf Q^{\top}\circ\Xi$ for a rotation
$\mathsf Q$, and $\mathcal T_e$ preserves $\{x=0\}$ and its normal line. If $\|v\|_{6,\alpha;r}\leq\epsilon$, $\epsilon/r=o(1)$, and
$v_x=0$ on the edge, then:
\begin{itemize}[nosep]
\item $\mathcal T_e$ changes the horizontal coordinates of $\mathsf B+m\pm v$ by at most $C(\epsilon r+\epsilon^2)$ in
$C^{4,\alpha}_{;r}$;
\item the transformed graphs, written over the new horizontal coordinates, differ from $B\pm v$ by at most
$C(\epsilon^2+\epsilon^3/r)$, where $B$ is the transformed middle graph.
\end{itemize}
\end{lemma}

\begin{proof}
At a half centre put $\psi_0=2\arctan(m_0/2)$, $\mathsf n=(\cos\psi_0,0,\sin\psi_0)$, $\mathsf q=(-\sin\psi_0,0,\cos\psi_0)$, and
let $\mathsf Q$ have columns $\mathsf n$, $(e_2+\mathfrak p\mathsf q)/\sqrt{1+\mathfrak p^2}$ and
$(\mathsf q-\mathfrak pe_2)/\sqrt{1+\mathfrak p^2}$. Then $\mathsf n=\Xi(0,0,m_0)$, the second column is the unit boundary
tangent of $\Xi(x,y,\mathsf B+m)$ at the marked point, and
\[
 \mathcal T_e(0)=\Bigl(0,-\frac{m_0\mathfrak p}{\sqrt{1+\mathfrak p^2}},-\frac{m_0}{\sqrt{1+\mathfrak p^2}}\Bigr),
\]
so the horizontal shift of the old origin is quadratic in the small data. The generators of $\mathcal T_e$ at $m_0=\mathfrak p=0$
are $(-xz/2,-yz/2,-1+(x^2+y^2-z^2)/4)$ and $(0,z,-y)$. Taylor expansion in $(m_0,\mathfrak p)$ on
$z=\mathsf B+m\pm v=O(r^2+\epsilon)$, after scaling to a unit annulus, gives the first bound. The composition estimate
\[
 \|f\circ(\Id+\mathsf d)-f\|_{k,\alpha;r}\leq Cr^{-1}\|f\|_{k+1,\alpha;r}\|\mathsf d\|_{k,\alpha;r}\qquad(k\leq4)
\]
and the inverse function theorem on scaled patches give the second. By Lemma~\ref{lem:chart}(3), $\mathcal T_e$ commutes with the
reflection in $\{x=0\}$, so it preserves this plane and its normal line, and its first component is $x$ times a positive
factor. At a full centre the same computation applies to a rigid motion.
\end{proof}

\begin{lemma}[Vertical and auxiliary normal graphs]\label{lem:graphchange}
Over a smooth base $\mathsf B$ with $\mathsf B(0)=D\mathsf B(0)=0$ let $v$ be a vertical Jacobi scalar with
$\|v\|_{6,\alpha;r}\leq\epsilon$, $\epsilon/r=o(1)$. The vertical graph of $v$ and the auxiliary normal exponential graph
of the corresponding auxiliary normal scalar differ by $C(\epsilon^2+\epsilon^3/r)$ in $C^{4,\alpha}_{;r}$.
\end{lemma}

\begin{proof}
Write $\xi\in\R^2$ for the horizontal coordinate of the vertical chart, let $\mathsf W^2=1+|D\mathsf B|^2$, and put $\mathfrak d_0=\mathfrak d$
at a half centre and $\mathfrak d_0=1$ at a full centre, so that the chart has conformal factor $\mathfrak d_0^{-1}$. The Euclidean normal
scalar of $v$ is $f=v/(\mathfrak d_0\mathsf W)$. The auxiliary exponential map displaces by $f\n$ up to a quadratic
remainder with bounded coefficients, so the point over $\xi$ moves to a point $(\xi',z')$ with
\[
 \xi'=\xi-\frac{vD\mathsf B}{\mathsf W^2}+O(\epsilon^2),\qquad z'=\mathsf B(\xi)+\frac v{\mathsf W^2}+O(\epsilon^2).
\]
The first-order part of $z'-\mathsf B(\xi')-v(\xi')$ cancels exactly, because
$v\mathsf W^{-2}-v+D\mathsf B\cdot vD\mathsf B\,\mathsf W^{-2}=0$. Since $D\mathsf B=O(r)$, the horizontal change is
$O(\epsilon r+\epsilon^2)$; the composition estimate of the previous proof gives the claim.
\end{proof}

\subsection{The pregluing}

\begin{definition}[Pregluing]\label{def:pregluing}
The \emph{sheet scalars} are the Euclidean normal heights $f_k$ of~\eqref{eq:sheetfields} in the stacking, and $\mp u$
of~\eqref{eq:doublingfield} or Proposition~\ref{prop:genuszero} in the doubling. The \emph{outer sheets} are the
auxiliary normal exponential graphs over $\mathcal C_c$ of the scalars $\Omega f_k$, respectively $\mp\Omega u$, outside the
disks $\{r\leq4\rho\}$ around the centres. In the recentred vertical chart of a bridge $e$, with $r$ the horizontal radius
and a fixed cutoff function, the surface consists of:
\begin{enumerate}[nosep]
\item the bridge, including its force profile, for $r\leq\rho$;
\item a radial interpolation between the bridge tails and the \emph{plateau graphs} $B_e\pm v_e$ for $\rho\leq r\leq2\rho$;
\item the plateau graphs for $2\rho\leq r\leq3\rho$;
\item a radial interpolation between the plateau graphs and the outer sheets for $3\rho\leq r\leq4\rho$.
\end{enumerate}
The resulting surface is denoted $M=M(\tau,F)$; it is the \emph{pregluing}. The sheets are oriented alternately. On $M$
we fix a smooth positive scale function $\mathfrak r$:
\begin{itemize}[nosep]
\item on a bridge $\mathfrak r=a_e\cosh\sigma$, which is comparable to $r$;
\item near a centre $\mathfrak r$ is comparable to the distance $r$ to the centre, for $r\leq s^{-N_0}/C$;
\item on the ends $\mathfrak r=\lambda(t)$;
\item elsewhere $\mathfrak r$ and $\mathfrak r^{-1}$ are bounded by powers of $s$.
\end{itemize}
\end{definition}

At half centres all component graphs have vanishing $x$-derivative on $\{x=0\}$, and the cutoffs are radial. Hence the
interpolations preserve the exact free-boundary condition and the exterior side, by Lemmas~\ref{lem:auxgraph}
and~\ref{lem:halfedge}. Here $N_0$ is chosen so that the disks of radius $s^{-N_0}/C$ around the centres of one sheet are disjoint; in the stacking the
centres of one sheet are uniformly separated (Theorem~\ref{thm:finitestacking}(4)), and in the doubling they are separated by
a fixed distance in conformal coordinates. In the doubling the full centres lie at physical distance comparable to
$\lambda(T)$ from the axis, which is large but bounded by a power of $s$, since $\lambda(T)=O(d^{-10})$ and $d^{-1}\leq Cs$.

\paragraph{Spaces.} The \emph{height space} $\cX$ consists of $C^{2,\alpha}$ functions on $M$ satisfying the Neumann condition
$\partial_{\eta_A}u=0$ on $\partial M$ and having on every end region the form $u=u_\infty+\tilde u$, where $u_\infty$ is a constant
and $e^{\delta t}\tilde u\in C^{2,\alpha}$; here $0<\delta<1$ is fixed. Its norm is the supremum of the scaled norms
$\|u\|_{2,\alpha;\mathfrak r}$ on patches of scale $\mathfrak r$, plus $|u_\infty|+\|e^{\delta t}\tilde u\|_{2,\alpha}$ on the end regions, computed
in unit $(t,\theta)$-patches. The \emph{source space} $\cY$ consists of functions $q$ with finite norm $\sup\|q\|_{0,\alpha;\mathfrak r}$,
with the weight $e^{\delta t}$ on the end regions. Sources are always normalized: the normalized mean curvature is $\Omega_M\mathfrak r^2H$, where
$\Omega_M=\Omega|_M$ is fixed on the initial surface $M$.

\begin{lemma}[Uniform geometry in the doubling]\label{lem:uniformgeometry}
Consider the doubling, with $d\in(0,d_0)$, $c=\sqrt{(1-d)/2}$ and $T=t_0+10\log(1/d)$.
\begin{enumerate}[nosep]
\item For every $k$ there is $C_k$, independent of $d$, such that on every patch $P=[t_1,t_1+1]\times[\theta_1,\theta_1+1]$ with
$t_1\geq t_0$ the rescaled immersion $\lambda(t_1)^{-1}(X_c-X_c(t_1,\theta_1))$ has $C^k$ norm at most $C_k$, and its induced
metric lies between $C'^{-1}$ and $C'$ times the flat metric, with $C'$ independent of $d$. The same holds for $\Omega$ along $\mathcal C_c$ and for the
coefficients of $\cL$ and $\cB$.
\item On every such patch with $t_1\leq t_{\rm end}$, away from the \emph{pole disks} $\{r\leq s^{-N_0}/C\}$ around the centres, the fields $u/\tau$ and their regular parts at
the centres have $C^k$ norms at most $C_kd^{-N}$. The Green kernels satisfy the bounds of Lemma~\ref{lem:green}, with the
single factor $d^{-1}$ in the zero mode.
\item For $t\leq t_{\rm end}$ one has $c\leq\lambda(t)\leq Cd^{-10}$ and $e^{\delta t}\leq Cd^{-10}$.
\end{enumerate}
Consequently every constant in Sections~\ref{sec:pregluing}--\ref{sec:proof} is bounded by $Cd^{-N}$, with $C,N$ depending only on
$(b,\gamma)$, and hence by $Cs^N$.
\end{lemma}

\begin{proof}
(1) On $P$, $\lambda(t)/\lambda(t_1)=\cosh t/\cosh t_1$ lies in $[e^{-1},e]$ and has bounded derivatives, and the vertical
component $c\hc(t-t_1)/\lambda(t_1)=(t-t_1)/\cosh t_1$ is bounded with all derivatives. The parameter $c$ ranges in a compact
subset of $(0,1)$. (2) In~\eqref{eq:doublingfield} the factor $d^{-1}$ appears only in front of $\Upsilon$ and in the means
$w/(2d)$ of the boundary fields, and the finite data and their derivatives are bounded by powers of $d^{-1}$
(Theorem~\ref{thm:finitedoubling}(3)). (3) This follows from $\lambda(t)=c\hc\cosh t$ and $e^{T}=e^{t_0}d^{-10}$.

The constants of Sections~\ref{sec:pregluing}--\ref{sec:proof} arise from four sources:
\begin{itemize}[nosep]
\item local elliptic, Schauder and Taylor estimates on unit conformal patches of the sheets, or on patches of the necks,
whose geometry is that of a standard catenoid up to $\Ostar$-corrections;
\item the Green kernel bounds;
\item the finite data and their derivatives;
\item conversions between conformal and physical scaled norms, and the end weight.
\end{itemize}
Each is bounded by $Cd^{-N}$ by (1)--(3), and only finitely many such operations are performed.
\end{proof}

\begin{proposition}[Initial error]\label{prop:initialerror}
$\|\Omega_M\mathfrak r^2H_M\|_{\cY}\lst\tau^2$, uniformly on the force ball.
\end{proposition}

\begin{proof}
On the bridges, including the force profiles, this is Lemma~\ref{lem:bridgeerror} together with the bound for the
profiles. On the inner matching annulus $\rho\leq r\leq2\rho$, Lemmas~\ref{lem:bridgeerror}(2) and~\ref{lem:gapexpansion}
and the profile mismatch bound the difference between the tails and the plateau graphs in $C^{4,\alpha}_{;r}$ by
\[
 \lst a^2+\frac{a^3}{R^2}+aR^2+\tau^2\lst\tau^2,
\]
where $a^3/R^2$ is the difference between $\mathfrak h_0$ and $a\log(2r/a)$, and $aR^2$ is the remainder of
Lemma~\ref{lem:gapexpansion}. The term $\tau^2$ accounts for the recentering (Lemma~\ref{lem:recenter} with
$\epsilon\lst\tau$, giving $\epsilon^2+\epsilon^3/R\lst\tau^2$). On the outer matching annulus Lemma~\ref{lem:graphchange} gives
the same bound. Scaled cutoff derivatives are bounded, so the linearized curvature of a mismatch of size $\lst\tau^2$ is
$\lst\tau^2$, and the quadratic part is $\lst\tau^4/R$.

On the outer sheets the first variation vanishes because the scalars are Jacobi fields. If $g$ is the induced metric of
$\mathcal C_c$ and $A^2=A\circ A$, the Euclidean normal graph metric is $g-2fA+f^2A^2+df\otimes df$, and its quadratic curvature terms are bounded by
$Cr^2(f^2+|f||Df|+|Df|^2+|f||D^2f|+|Df|^2|D^2f|)$ with bounded base coefficients. With $|f|+r|Df|+r^2|D^2f|\lst\tau$ this is
$\lst\tau^2+\tau^3/r^2\lst\tau^2$ for $r\geq R$. The auxiliary exponential map only changes bounded smooth coefficients. On an
end $f=\Ostar(\tau(1+t))$ and the base curvature decays, so the normalized error is $\lst\tau^2(1+t)^Ne^{-t}$, which is bounded in
the weighted norm since $\delta<1$.
\end{proof}

\subsection{Topology of the pregluing}

\begin{proposition}\label{prop:topology}
The pregluing $M$ is connected and orientable, and has type $(\ell,b,\gamma)$, defined as in Definition~\ref{def:type}: $\ell$ sheets and $E=\ell+b+2\gamma-2$
half bridges in the stacking, and two sheets, $b$ half bridges and $\gamma$ full bridges in the doubling.
\end{proposition}

\begin{proof}
Truncate the ends and cap each of them by a disk. Each sheet becomes a disk whose boundary circle lies on $\Sph$. A half
bridge attaches an untwisted band between the boundary circles of two consecutive disks; a full bridge removes two interior
disks and inserts a cylinder. The alternating orientations of the sheets extend across the bridges, so the surface is
orientable. The Euler characteristic of the capped surface $\overline M$ is
\[
 \chi(\overline M)=\#\{\text{sheets}\}-\#\{\text{half bridges}\}-2\,\#\{\text{full bridges}\},
\]
which is $\ell-E=2-2\gamma-b$ in the stacking and $2-b-2\gamma$ in the doubling. It therefore suffices to show that the
number of boundary components is $b$.

\emph{Stacking.} The $n_1$ bands between disks $1$ and $2$ give a connected genus-zero surface with $n_1$ boundary cycles,
one for each gap between consecutive first-interface centres. The first band from disk $3$ joins disk $3$ to one of these
cycles; the number of cycles is unchanged. Every further band of the second interface ends on disk $3$, whose boundary
now belongs to the joined cycle, and on a different cycle of disk $2$, because the second-interface centres lie in
distinct gaps (Theorem~\ref{thm:finitestacking}(3)). Each such band merges two cycles. After the $n_2$ bands there are
$n_1-n_2+1=b$ cycles. Every later disk is attached by a single band, which does not change the number of cycles.

\emph{Doubling.} The $b$ bands between the two disks give $b$ cycles, and full bridges do not meet the boundary.
Figure~\ref{fig:bands} shows both counts for small types.
\end{proof}

\begin{figure}[t]
\centering
\includegraphics[width=\textwidth]{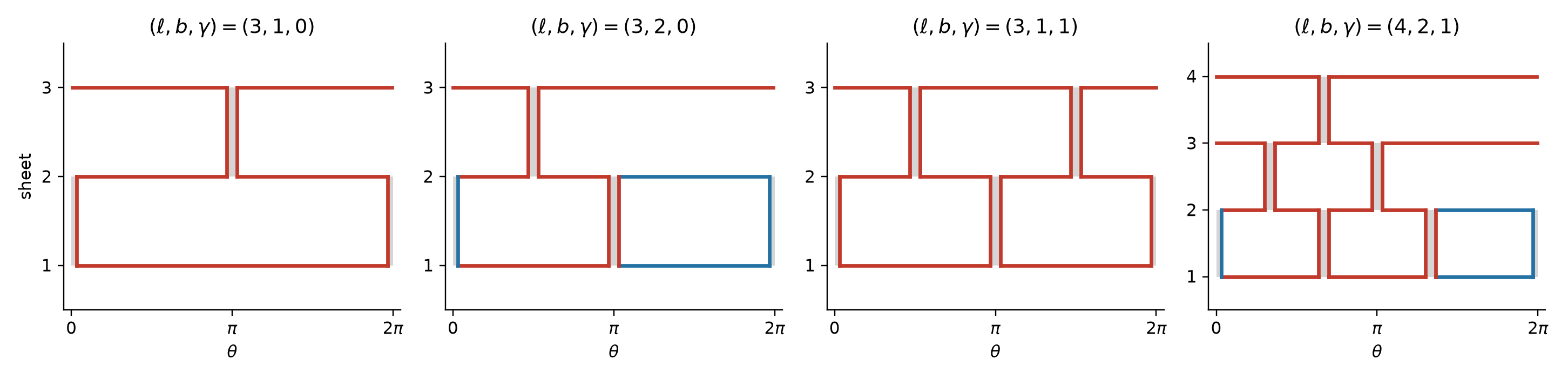}\\[4pt]
\includegraphics[width=0.78\textwidth]{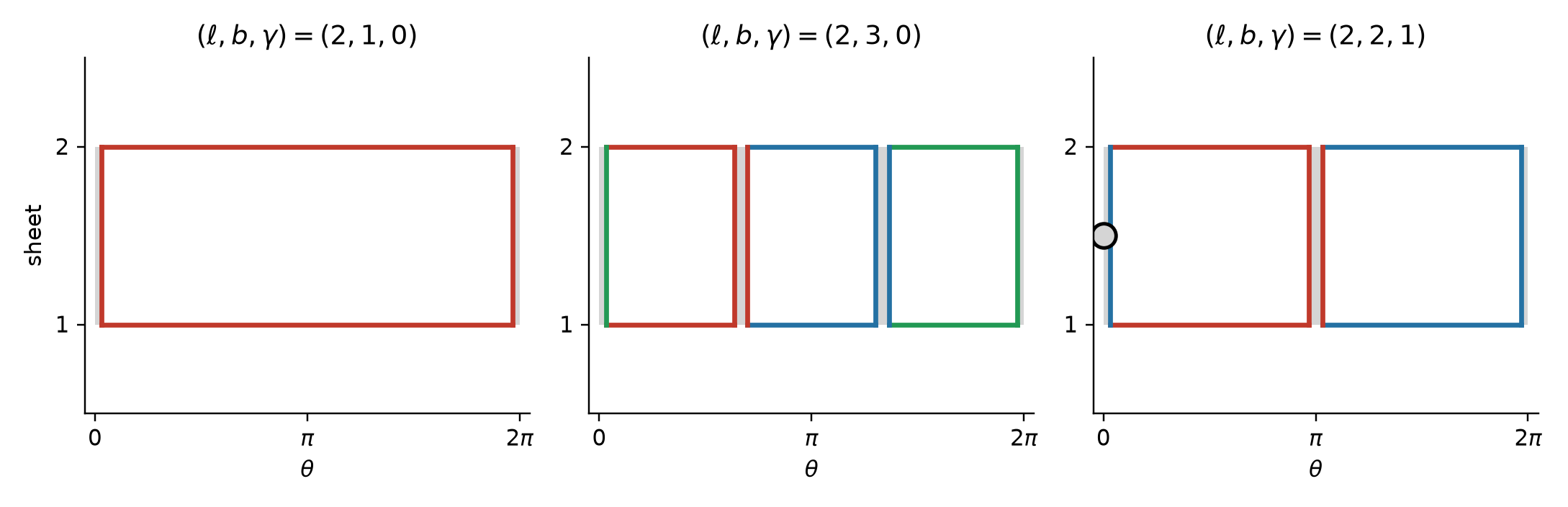}
\caption{The boundary of the capped surface, unrolled: the boundary circle of sheet $k$ is the horizontal line at height
$k$, the half bridges are the grey bands, and full bridges are drawn as grey discs. Each boundary cycle has its own colour.
Top: stacking, with $n_1=\gamma+b$, $n_2=\gamma+1$ and $n_i=1$ afterwards; the second-interface bands lie in distinct gaps of
the first interface. Bottom: doubling. The count of cycles is $b$ in each case (Proposition~\ref{prop:topology}).}
\label{fig:bands}
\end{figure}

\begin{figure}[t]
\centering
\begin{subfigure}{0.48\textwidth}\includegraphics[width=\linewidth]{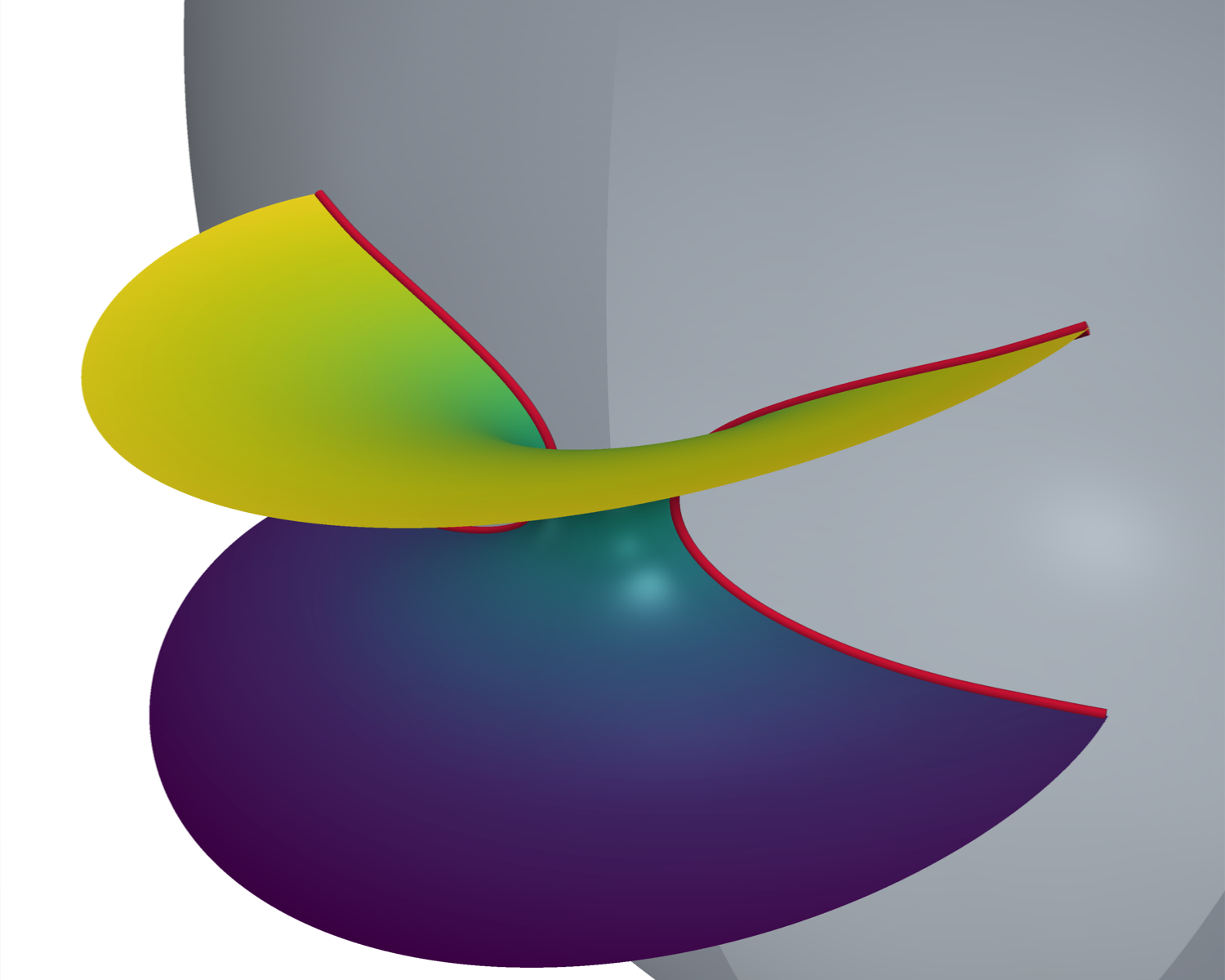}\caption{A half bridge.}\end{subfigure}
\hfill
\begin{subfigure}{0.48\textwidth}\includegraphics[width=\linewidth]{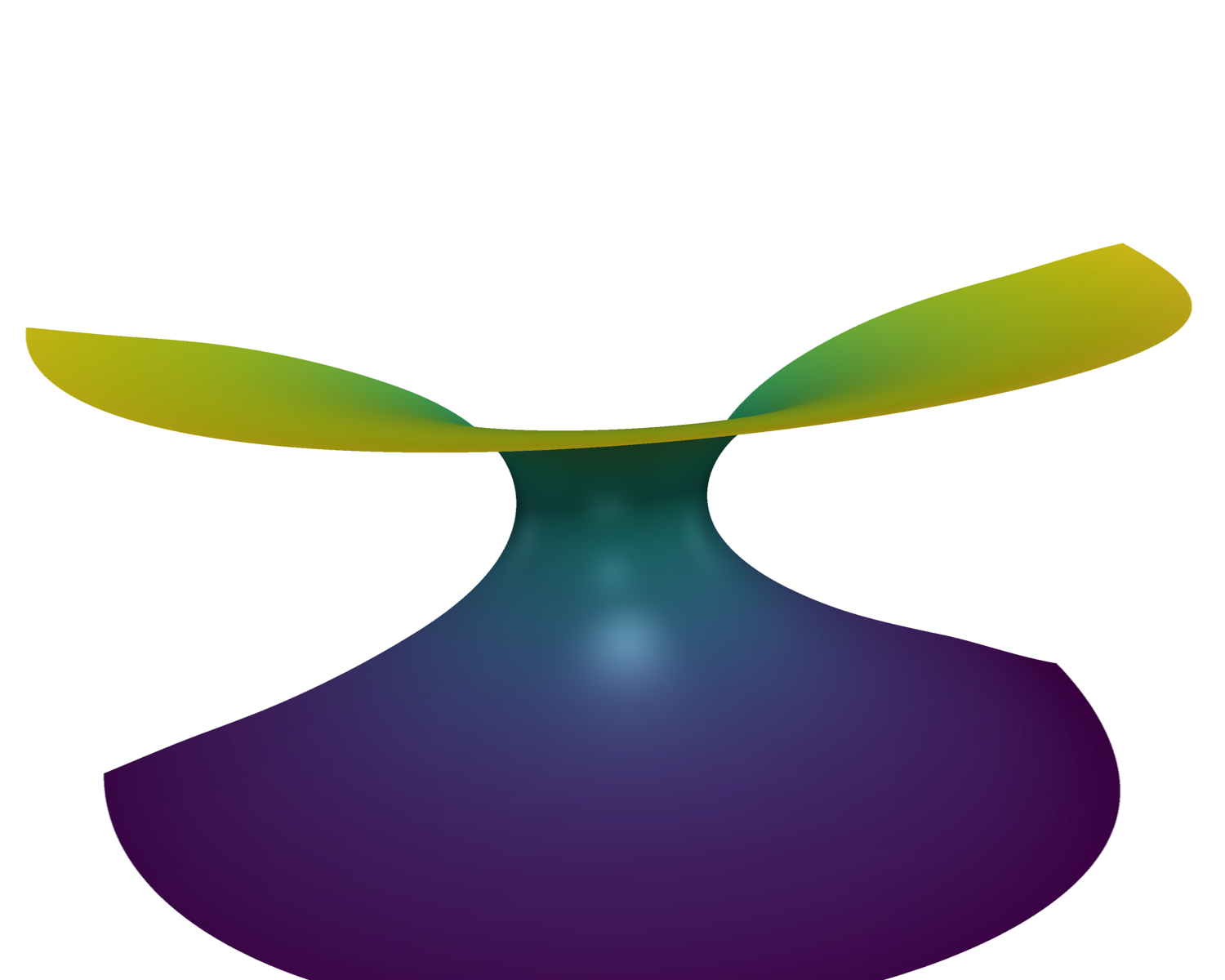}\caption{A full bridge, cut open.}\end{subfigure}
\caption{Schematic local pictures (with the conventions of Figure~\ref{fig:doubling}; the free-boundary condition is not
imposed). (a) A half-catenoidal bridge joining two consecutive sheets at the unit sphere (grey); the red curves are the
boundary on the sphere. (b) A full bridge (a catenoidal neck) between the two sheets of a doubling, cut along a plane
through its axis.}
\label{fig:zoom}
\end{figure}

\section{The linear problem}\label{sec:linear}

\subsection{Broken spaces and the projected operator}

Cut $M$ along the circles $\{r=R\}$ around all centres; at a half centre the cut is a half circle meeting $\Sph$ at its
endpoints. This splits $M$ into \emph{sheets with holes} and \emph{necks}. Each bridge $e$ has two cuts, one on each
incident sheet; we call them its two \emph{incidences}. The neck of $e$ is parametrized by
\[
 [-S_e,S_e]\times\mathbb I_e,\qquad S_e=\arcosh(R/a_e)=\tfrac12s+O(1),\qquad
 \mathbb I_e=\R/2\pi\Z\ \text{or}\ [-\tfrac\pi2,\tfrac\pi2],
\]
with Neumann conditions on the sides $\phi=\pm\pi/2$ of a half neck.

\paragraph{Reference pieces.} The \emph{reference sheet} is $\mathcal C_c$ with round conformal disks or half disks of
physical radius $R$ removed at the centres. Its operator is the normalized conjugated Jacobi operator
\[
 \widehat\cL_Cu=\Omega_C\,\mathfrak r^2\cL_{\mathcal C_c}(\Omega_C^{-1}u),\qquad \Omega_C=\Omega|_{\mathcal C_c},
\]
where $\cL_{\mathcal C_c}$ is the Jacobi operator of $\mathcal C_c$ and $\mathfrak r$ is a scale function on the reference sheet chosen as in Definition~\ref{def:pregluing}. Since $\Omega=1$ on $\Sph$ and for $|x|\geq3$, $\Omega_C=1$ at every centre
and on the ends. By Lemma~\ref{lem:auxgraph}(3) the Robin condition becomes the Neumann condition. The \emph{reference
neck} of $e$ is the cylinder above with the operator $\cL_{\rm cat}$ defined before Lemma~\ref{lem:neckidentities}, which is the normalized
Jacobi operator of $\mathrm{Cat}_{a_e}$.

\paragraph{Traces.} At a half cut we use
\[
 T^{k,\alpha}_{\rm H}=\{f\in C^{k,\alpha}[-\tfrac\pi2,\tfrac\pi2]:f'(\pm\tfrac\pi2)=0\},\qquad k=1,2,
\]
and at a full cut $T^{k,\alpha}_{\rm F}=C^{k,\alpha}(\R/2\pi\Z)$. Their subspaces of functions with zero mean are denoted
$T^{k,\alpha}_0$. Even reflection at the endpoints identifies $T^{k,\alpha}_{\rm H}$ with the even periodic functions. The basis of
$T_{\rm H}$ is $1$, $\cos2n\phi$ and $\sin(2n+1)\phi$; in particular the first sine mode is present. The operator $|\partial_\phi|$,
which multiplies the frequency $n$ by $n$, is an isomorphism $T^{2,\alpha}_0\to T^{1,\alpha}_0$, for instance by the periodic Hilbert
transform. The value trace of a height is its restriction to a cut. The derivative trace is the normalized outgoing
derivative: $\mathcal N=-R\partial_r$ on a sheet and $\pm\partial_\sigma$ at the ends of a neck. On the actual surface the
derivative is taken along the auxiliary conormal; see Lemma~\ref{lem:conormal}.

We write $(\mathsf e_n)$ for the Fourier basis of each trace space: $1,\cos n\phi,\sin n\phi$ at a full cut and
$1,\cos2n\phi,\sin(2n+1)\phi$ at a half cut, and $\mathsf f_n$ for the coefficients of a trace $\mathsf f$.

\paragraph{Sources and gauges.} For every first angular mode $\mathsf e_1(\phi)$ that survives the imposed symmetries
($\cos\phi,\sin\phi$ at a full neck, $\sin\phi$ at a half neck) we adjoin the core source $p(\sigma)\mathsf e_1(\phi)$, with
$p=p_{\rm F}$ or $p_{\rm H}$ as in Section~\ref{sec:pregluing}. We also impose a \emph{gauge}: the corresponding Fourier coefficient
of the height at the waist $\sigma=0$. The source coefficients form a vector $\mathbf z$, and we write $\Psi\mathbf z$ for the sum of
the sources and $\Lambda u$ for the gauges.

\begin{definition}\label{def:broken}
The \emph{broken height space} is the product of the $C^{2,\alpha}$ spaces of the pieces, with the scaled norms of
Section~\ref{sec:pregluing}, the end weight $e^{\delta t}$ and one free constant per end. The \emph{broken source space} is
the corresponding product of $C^{0,\alpha}$ spaces. The \emph{projected broken operator} is
\[
 \cA(u,\mathbf z)=\bigl(\widehat\cL u-\Psi\mathbf z,\ j_0(u),\ j_1(u),\ \Lambda u\bigr).
\]
Here $\widehat\cL$ is $\widehat\cL_C$ on sheets and $\cL_{\rm cat}$ on necks. Further, $j_0$ is the value of the neck minus the value
of the sheet at each incidence, and $j_1$ is the sum of the
outgoing derivatives of the neck and the sheet. Both are taken after converting the neck heights to the upward sheet
normal. We write $\cA_0$ for the reference operator and $\cA_M$ for the operator of the pregluing.
\end{definition}

\subsection{The filled sheet}

\begin{lemma}[Filled inverse]\label{lem:filled}
Let $d\neq0$. The problem $\cL u=q$ on $\{t\geq t_0\}$ with $\cB u=0$ at $t_0$ and $u=u_\infty+\tilde u$, $e^{\delta t}\tilde u$
bounded, has a unique solution. In the cylinder norms
\[
 |u_\infty|+\|e^{\delta t}\tilde u\|_{C^{2,\alpha}}\leq C(1+|d|^{-1})\,\|e^{\delta t}q\|_{C^{0,\alpha}},
\]
where $C$ does not depend on $c$ in the ranges used below.
\end{lemma}

\begin{proof}
For the zero mode let $u_0=\int G_0(t,t')q_0(t')\,dt'$ and $u_\infty=d^{-1}\int_{t_0}^\infty\Upsilon q_0\,dt'$. By~\eqref{eq:Upsilondiff},
\[
 u_0(t)-u_\infty=(w(t)-1)u_\infty+\int_t^\infty\bigl\{w(t')-w(t)+(t'-t)w(t)w(t')\bigr\}q_0(t')\,dt'.
\]
Since $|\Upsilon(t')|\leq C(1+|d|t')$, the constant costs $C(1+|d|^{-1})$; the integral has weighted norm
$\leq C\int_0^\infty(2+x)e^{-\delta x}dx$; and $w-1$ decays like $e^{-2t}$. No unbounded factor of $t$ appears. For $n\geq2$,
Lemma~\ref{lem:green} bounds the Green coefficients, and for $n=1$ the explicit $G_1$ obeys the same bound, by $Cn^{-1}(e^{-n|t-t'|}+e^{-n(t+t'-2t_0)})$, and integration against
$e^{-\delta t'}$ costs $C/(n(n-\delta))$. Local Schauder estimates, interior and along the Robin boundary, give the
$C^{2,\alpha}$ bound. Uniqueness: a bounded rotationally invariant solution is a multiple of $w$, but $\cB w=d\neq0$; the
other homogeneous solutions are multiples of $y_n$, and $\nu_n\neq0$.
\end{proof}

\begin{lemma}[Dependence on the criticality parameter]\label{lem:duniform}
In the doubling ($d\downarrow0$, $c=\sqrt{(1-d)/2}$, $T=t_0+10\log(1/d)$) every constant of this section is bounded by $Cd^{-N}$,
hence by $Cs^N$. More precisely:
\begin{enumerate}[nosep]
\item The coefficients of $\cL$ do not depend on $c$, and $\cB=\partial_t-c$ depends smoothly on $c$. Local Schauder
constants on unit $(t,\theta)$-patches, interior or along $t=t_0$, are uniform.
\item For $n\geq1$, $\nu_n\geq n^2$, and the bounds of Lemma~\ref{lem:green} for $n\geq2$ are uniform in $t,t'\geq t_0$ and in
$d$. The coefficient $G_1$ is bounded uniformly.
\item The only singular factor is $d^{-1}$ in the zero mode: in $G_0$ and in the end constant of Lemma~\ref{lem:filled}.
\item On $\{t\leq T+O(1)\}$ the physical scale satisfies $c\leq\lambda\leq Cd^{-10}$. Conversion between conformal and
physical scaled norms there costs at most $Cd^{-N}$. On the ends the scale is $\lambda$ in both.
\item By Theorem~\ref{thm:finitedoubling} and Proposition~\ref{prop:genuszero}, $d^{-1}\leq Cs$.
\end{enumerate}
\end{lemma}

\begin{proof}
Items (1)--(3) follow from the explicit formulas in Section~\ref{sec:green}; for (2) we use $c\geq1/2$ and $d\leq1/2$. Item
(4) is $\lambda(T)=c\hc\cosh T\leq Cd^{-10}$. The constants in Lemmas~\ref{lem:sheet}--\ref{lem:neckDN} and
Theorem~\ref{thm:reference} below are built from these quantities by finitely many products and Neumann series, so each is bounded by a power of $d^{-1}$;
see also Lemma~\ref{lem:uniformgeometry}.
\end{proof}

\subsection{Sheets with holes}

The \emph{monopoles} are solutions $\mathsf p_j$ of the conjugated equation on the filled sheet, with the Neumann condition, of
the form $\mathsf p_j=\log r_j+\text{(regular remainder)}$ near the centre $j$, with the expansion~\eqref{eq:monopole} below, where $r_j$ is the \emph{physical} distance, and bounded at the
ends. They are the conjugated pole fields $\Omega_CI_p$ and $\Omega_CP_\vartheta$, shifted by multiples of $\Omega_C\Upsilon$ and
normalized. Their regular constants form the matrix $\sK=(\sK_{ij})$:
\begin{equation}\label{eq:monopole}
 \mathsf p_j=\log r_j+\sK_{jj}+\Ostar\bigl(r_j(1+|\log r_j|)^2\bigr),\qquad \mathsf p_j=\sK_{ij}+\Ostar(r_i)\ \ (i\neq j),
\end{equation}
in scaled $C^{2,\alpha}$ norms near the centres. Since $r_j$ is physical, $\sK_{jj}$ contains the term $-\log\lambda(t_j)$, where $t_j$ is the $t$-coordinate of the centre, of
Lemma~\ref{lem:physradius}. The estimate~\eqref{eq:monopole} follows from the logarithmic Green bounds and a dyadic
decomposition of the logarithmic potential of the remainder; at a half centre only the flat Poisson equation is reflected
evenly. Put
\[
 \sG=(\log R)\Id+\sK .
\]

\begin{lemma}[Extension across a hole]\label{lem:extension}
A $C^{0,\alpha}$ function on $\{R<r<2R\}$ of scaled norm $A$ has an extension to $\{r<R\}$ of supremum and scaled norm at
most $CA$, with $C$ independent of $R$; the same holds for half disks.
\end{lemma}

\begin{proof}
Scale to $R=1$, reflect radially by $r\mapsto2-r$ on $\{\frac12<r<1\}$, multiply by a fixed cutoff, and extend by zero.
\end{proof}

\begin{lemma}[Harmonic lifts]\label{lem:lift}
For $\mathsf f\in T^{2,\alpha}_0$ at one cut there is a homogeneous solution $E_R\mathsf f$ on the sheet with holes, bounded at the ends,
with polynomially bounded height norm, whose traces at all cuts are $\mathsf f+\Ostar(R)\|\mathsf f\|$ for the values and
$|\partial_\phi|\mathsf f+\Ostar(R)\|\mathsf f\|$ for the derivatives. The first terms are placed at the given cut, and zero at all the
others.
\end{lemma}

\begin{proof}
We keep two normalizations apart. In a physical conformal coordinate $z$ around the centre, in which the metric of
$\mathcal C_c$ is $\hat\lambda^2|dz|^2$, the conjugated Jacobi operator has the form $\hat\lambda^{-2}\mathsf P$ with
\[
 \mathsf P=\Delta+\mathsf b\cdot D+\mathsf V ,
\]
where $\mathsf b$ and $\mathsf V$ have bounded jets ($\mathsf b=2\partial_x+O(r)$ at a half centre), and $\widehat\cL_C=\mathfrak r^2\hat\lambda^{-2}\mathsf P$.
Let $\cQ_C$ be the inverse of $\hat\lambda^{-2}\mathsf P$ on the filled sheet with the Neumann condition, with the end spaces of
Lemma~\ref{lem:filled}; it is obtained from Lemma~\ref{lem:filled} by conjugation with $\Omega_C$.

Let $P^{\rm ext}_R\mathsf f=\sum_{n\geq1}(R/r)^n\mathsf f_n\mathsf e_n(\phi)$ be the exterior Poisson extension. At a half cut, even reflection of
the data gives the Neumann solution and its $C^{2,\alpha}$ regularity at the corners. Let $v_R$ be its product with a radial
cutoff supported in the coordinate disk. Since $P^{\rm ext}_R\mathsf f$ is harmonic, the \emph{unnormalized} residual
$q_R^{\rm un}=\mathsf Pv_R$ consists of the first-order terms and the cutoff terms, and
\[
 |q_R^{\rm un}|\lst\Bigl(\frac R{r^2}+\frac Rr\Bigr)\|\mathsf f\|,\qquad\text{equivalently}\qquad
 \|r^2q_R^{\rm un}\|_{0,\alpha;r}\lst(R+Rr)\|\mathsf f\| ,
\]
for $R\leq r$ in the coordinate disk; the cutoff region contributes $\lst R\|\mathsf f\|$. (In the model $\mathsf P=\Delta+2\partial_x$,
$v_R=Ry/r^2$ one has $\mathsf Pv_R=-4Rxy/r^4$, of size $R^{-1}$ at $r=R$.) Extend $\hat\lambda^{-2}q_R^{\rm un}$ into the hole by
Lemma~\ref{lem:extension}; inside the hole the extension has size $\lst R^{-1}\|\mathsf f\|$. Put $E_R\mathsf f=v_R-\cQ_C(\hat\lambda^{-2}q_R^{\rm un})$
on the sheet with holes; it is exactly homogeneous. The logarithmic potential of the extended source is
$\lst R(1+|\log R|)^2\|\mathsf f\|$: on the annuli integrate $Rr^{-2}$ against $r\,dr\,d\theta$, and inside the hole use its area $O(R^2)$.
With the uniform weighted Green bound for compactly supported sources this gives
$\|\cQ_C(\hat\lambda^{-2}q_R^{\rm un})\|_\infty\lst R(1+|\log R|)^2\|\mathsf f\|$. Schauder estimates on filled patches of radius $R$ around the
cut then give the same bound for the scaled $C^{2,\alpha}$ norm. These patches belong to the smooth filled sheet, so no
corner estimate is used.
\end{proof}

\begin{lemma}[Sheets with holes]\label{lem:sheet}
Assume $\|\sG\|+\|\sG^{-1}\|\lst1$. Every value trace $\mathsf m+\mathsf f$, with mean vector $\mathsf m$ and zero-mean part $\mathsf f$, is attained by
a unique homogeneous solution on the sheet with holes that is bounded at the ends. Its derivative trace is
\[
 \mathcal N(\mathsf m+\mathsf f)=-\sG^{-1}\mathsf m+|\partial_\phi|\mathsf f+\Ostar(R)\bigl(|\mathsf m|+\|\mathsf f\|\bigr).
\]
Every normalized source $q$ has a particular solution with traces, of norm $\lst\|q\|_{\cY}$.
\end{lemma}

\begin{proof}
For strengths $\mathsf b$ the field $\sum_j\mathsf b_j\mathsf p_j$ has value traces $\sG\mathsf b+\Ostar(R)|\mathsf b|$ and derivative
traces $-\mathsf b+\Ostar(R)|\mathsf b|$, by~\eqref{eq:monopole}. Together with Lemma~\ref{lem:lift}, the map from
$(\mathsf b,\mathsf f)$ to all value traces is $\diag(\sG,\Id)+\Ostar(R)$, and a Neumann series inverts it. Uniqueness: let $u$ have
zero traces. Pair $u$, unconjugated, by Green's formula with the solution constructed for an arbitrary trace supported in an
open cut arc. The Robin boundary and the ends contribute nothing, so the conormal derivative of $u$ vanishes on every
cut. Extending $u$ by zero into the holes gives a weak solution on the filled sheet, which vanishes by
Lemma~\ref{lem:filled}. For a source $q$, the unnormalized right-hand side is $\lst r^{-2}\|q\|_{\cY}$ near a hole; extend it
with size $\lst R^{-2}\|q\|_{\cY}$, and use
$\sup_p\int_{R<|y|<1}(1+|\log|p-y||)|y|^{-2}dy\leq C(1+|\log R|)^2$. For the end constant use the integral mass of
the compactly supported part, which is $\lst1$, and the uniform weighted bound of the Green kernel for compactly
supported sources; the part on the ends is treated by Lemma~\ref{lem:filled}.
\end{proof}

\subsection{Necks}

On a neck put $Z(\sigma)=\tanh\sigma$ and $Z^\sharp(\sigma)=\sigma\tanh\sigma-1$, the rotationally invariant Jacobi fields of the
catenoid (Wronskian one), and for the first mode $y^{\rm cat}_1=\sech\sigma$ and $\tilde y^{\rm cat}_1=\sinh\sigma+\sigma\sech\sigma$ (Wronskian two).

\begin{lemma}[Particular neck solutions]\label{lem:neckpart}
For every $C^{0,\alpha}$ source $q$ on the neck and every gauge vector $\mathsf g$ there is a pair $(u,\mathbf z)$ with
$\cL_{\rm cat}u-\Psi\mathbf z=q$ and prescribed gauges, with height, source coefficients and traces bounded by
$C(1+S_e)^2(\|q\|+|\mathsf g|)$.
\end{lemma}

\begin{proof}
For the zero mode use the kernel $Z^\sharp(\sigma)Z(t)-Z(\sigma)Z^\sharp(t)=(\sigma-t)Z(\sigma)Z(t)+Z(\sigma)-Z(t)$, bounded by
$|\sigma-t|+2$. For a first mode extend $q$ across the ends on unit collars, choose $\mathbf z=-\int y^{\rm cat}_1q/\int y^{\rm cat}_1p$, put
$q'=q+\mathbf zp$, so that $\int_\R y^{\rm cat}_1q'=0$, and use
\[
 u(\sigma)=y^{\rm cat}_1(\sigma)\int_0^\sigma y^{\rm cat}_1(t)^{-2}\int_{-\infty}^ty^{\rm cat}_1(x)q'(x)\,dx\,dt .
\]
The inner integral is bounded by $C\|q'\|e^{-|t|}$ (use the left tail for $t\leq0$ and minus the right tail for $t\geq0$), so
the growth of $(y^{\rm cat}_1)^{-2}$ is cancelled by $y^{\rm cat}_1(\sigma)$. Adding a multiple of $y^{\rm cat}_1$ prescribes the gauge. For $n\geq2$ the
maximum principle for $-\partial_\sigma^2+n^2-2\sech^2\sigma$ with Dirichlet data gives $\|u_n\|_\infty\leq\|q_n\|_\infty/(n^2-2)$.
Summation, smooth approximation and local cylinder estimates of fixed size give the full source; at half necks reflect
across the Neumann sides first.
\end{proof}

\begin{lemma}[Neck traces]\label{lem:neckDN}
With zero gauges and the first-mode sources allowed, the map from nonconstant end values to nonconstant end derivatives
of homogeneous neck solutions is
\[
 |\partial_\phi|\oplus|\partial_\phi|+E_e,\qquad \|E_e\|_{(T^{2,\alpha}_0)^2\to(T^{1,\alpha}_0)^2}\leq C(1+S_e)e^{-2S_e} .
\]
It attains arbitrary nonconstant end values. All homogeneous neck solutions with zero data vanish.
\end{lemma}

\begin{proof}
For $n=1$ let $W$ be the even solution of $\cL_{\rm cat}(W)=p$ with zero Cauchy data at $0$. On the positive tail
$W=\mathsf k_p\tilde y^{\rm cat}_1+\mathsf k_vy^{\rm cat}_1$, with $\mathsf k_p=\frac12\int_0^1y^{\rm cat}_1p>0$. Thus $W$ supplies equal end values with zero gauge,
using its source with coefficient $1/W(S_e)=O(e^{-S_e})$, and $\tilde y^{\rm cat}_1$ supplies opposite end values. Both have outgoing
derivative ratio $1+O((1+S_e)e^{-2S_e})$. Without the source the even solution $y^{\rm cat}_1/y^{\rm cat}_1(S_e)$ would have waist value
$\cosh S_e$; the source and gauge pair avoids this exponential loss. For $n\geq2$ use $e^{n\sigma}(n-\tanh\sigma)$ and
$e^{-n\sigma}(n+\tanh\sigma)$: the diagonal derivatives are $n+O(\sech^2S_e/n)$ and the off-diagonal ones $O(ne^{-2nS_e})$.
Summation with $|\mathsf f_n|\leq Cn^{-2-\alpha}\|\mathsf f\|_{2,\alpha}$ gives the norm bound.
\end{proof}

\subsection{Transmission}

Order the incidences so that each bridge contributes two consecutive rows, lower sheet first. Let $\sU$ and $\sV$ have the
columns $(1,1)$ and $(-1,1)$ on the two rows of a bridge and zero elsewhere. A radial neck field $\mathsf c_eZ+\mathsf c^\sharp_eZ^\sharp$,
converted to the upward sheet normal, has value and derivative vectors
\[
 \mathsf m=\sU\diag(k)\mathsf c+\sV\diag(k^\sharp)\mathsf c^\sharp,\qquad \mathsf b=\sU\diag(k')\mathsf c+\sV\diag(k^{\sharp\prime})\mathsf c^\sharp,
\]
with $k_e=\tanh S_e$, $k'_e=\sech^2S_e$, $k^\sharp_e=S_ek_e-1$, $k^{\sharp\prime}_e=k_e+S_ek'_e$. Let $\sG$ also denote the direct
sum of the sheet matrices in the order of the incidences, and put
\begin{equation}\label{eq:radialmatrix}
 \sM=\bigl[\sU\diag(k)-\sG\sU\diag(k'),\ \ \sV\diag(k^\sharp)-\sG\sV\diag(k^{\sharp\prime})\bigr],\qquad
 \widetilde\sM=\sM\,\diag(\Id,s^{-1}\Id).
\end{equation}

\begin{theorem}[Reference inverse]\label{thm:reference}
Assume that $\sG$, $\sG^{-1}$, $\widetilde\sM$ and $\widetilde\sM^{-1}$ have norms $\lst1$. Then $\cA_0$ is an isomorphism, for arbitrary
sources, compatible value and derivative jumps and gauges, and its inverse $\cQ_0$ has norm $\lst1$. The same holds on the
invariant subspaces of any finite symmetry group commuting with $\cA_0$.
\end{theorem}

\begin{proof}
Subtract the particular solutions of Lemmas~\ref{lem:sheet} and~\ref{lem:neckpart}, including the prescribed gauges. The
remaining problem is homogeneous, with zero gauges and jumps of polynomially bounded norm. Let $\mathsf f$ be the zero-mean neck
value at all incidences, so that the neck value is $\mathsf m+\mathsf f$, and prescribe the sheet value $\mathsf m+\mathsf f-j_0$. Then the value jumps
are exact. Write $(\cdot)_0$ for the zero-mean part and $\av$ for the vector of means of a trace over each incidence. By
Lemmas~\ref{lem:sheet} and~\ref{lem:neckDN} the derivative equations become
\[
 2|\partial_\phi|\mathsf f=(j_1)_0+|\partial_\phi|(j_0)_0,\qquad \mathsf m-\sG\mathsf b=\av j_0-\sG\av j_1,
\]
up to couplings of norm $\lst\varepsilon_3$, where $\varepsilon_3=R+\max_e(1+S_e)e^{-2S_e}$. The
second equation is the system with matrix $\sM$. Since $\varepsilon_3$ is smaller than any negative power of $s$, Neumann series solve
the first system for $\mathsf f$ and then the second for $\xi=(\mathsf c,s\mathsf c^\sharp)$. Because $|Z|\leq1$ and $|Z^\sharp|\leq1+S_e$, the
variables $\xi$ control the whole radial neck height. Reconstruction bounds all solutions, end constants and source
coefficients. For zero data the same systems give $\xi=0$, $\mathsf f=0$ and zero traces, and the uniqueness parts of
Lemmas~\ref{lem:sheet} and~\ref{lem:neckDN} give $u=0$ and $\mathbf z=0$. A symmetry commuting with $\cA_0$ commutes with its unique
inverse.
\end{proof}

\subsection{Verification of the matrix hypotheses}

\begin{proposition}[Stacking]\label{prop:matrixstack}
For the stacking data, $\sG=-\frac s2\Id+O(1)$ and $\widetilde\sM\to[\sU,\sV]$, whose columns are orthogonal of length $\sqrt2$.
Hence the hypotheses of Theorem~\ref{thm:reference} hold.
\end{proposition}

\begin{proof}
The regular constants are bounded and $\log R=-\frac s2+O(1)$. Further $k\to1$, $sk'\to0$, $k^\sharp/s\to\frac12$ and
$k^{\sharp\prime}\to1$, so the two column blocks of $\widetilde\sM$ tend to $\sU-\sG\sU k'\to\sU$ and
$s^{-1}(\sV k^\sharp-\sG\sV k^{\sharp\prime})\to\frac12\sV+\frac12\sV=\sV$. Distinct incidences are distinct rows, so orthogonality
holds even when several bridges meet the same sheet.
\end{proof}

\begin{proposition}[Doubling with positive genus]\label{prop:matrixdouble}
Let $b,\gamma\geq1$, $\Gamma=c_*(\gamma-1+\frac b2)$, $\mathsf D=\diag(2\Id_\gamma,\Id_b)$ and $\mathsf v=(\mathbf1_\gamma,c_*\mathbf1_b)^\top$. For a single
sheet $d\,\sK\to\sK_0=c_*\mathsf v\mathsf v^\top\mathsf D$, whose only nonzero eigenvalue is $\Theta=c_*(2\gamma+\frac b2)$, and
\[
 \frac{\sG}{s}\to\frac{\sK_0}{\Gamma}-\frac12\Id .
\]
After passing to the common and difference components of the two incidences at each centre, $\widetilde\sM$ tends, up to fixed
invertible normalizations, to $\diag(\Id,\ \Id-\sK_0/\Gamma)$. The hypotheses of Theorem~\ref{thm:reference} hold.
\end{proposition}

\begin{proof}
The limit of $d\sK$ follows from the zero Fourier coefficients. The full/full, full/half, half/full and half/half entries
tend to $2c_*$, $\frac12$, $1$ and $\frac12c_*$. The logarithmic diagonal corrections, including $-\log\lambda(t_j)$, are
$O(\log(1/d))$ and disappear after multiplication by $d$. Since $d\cdot s\to\Gamma$ uniformly on the force ball, the limit of $\sG/s$
follows. The sheet block has eigenvalues $-\frac12$ and $\Theta/\Gamma-\frac12>0$. The difference block has eigenvalues $1$ and
$1-\Theta/\Gamma=-c_*(\gamma+1)/\Gamma\neq0$. The matrix $\sK_0$ is not symmetric, so eigenvalues alone do not bound the inverse.
However $\mathsf D^{1/2}\sK_0\mathsf D^{-1/2}=c_*(\mathsf D^{1/2}\mathsf v)(\mathsf D^{1/2}\mathsf v)^\top$ is symmetric, and $\mathsf D^{1/2}$ has condition number
$\sqrt2$, which gives the singular-value bounds.
\end{proof}

\begin{proposition}[Doubling with genus zero]\label{prop:matrixzero}
For the data of Proposition~\ref{prop:genuszero}, $d\sK\to\frac12c_*\mathbf1\mathbf1^\top$, with nonzero eigenvalue $2\Gamma_0$. The sheet
limit has eigenvalues $-\frac12,\frac32$ and the difference block $1,-1$ (the zero-sum space being absent for $b=1$). The
dihedral symmetry removes every first-mode source and gauge. The hypotheses of Theorem~\ref{thm:reference} hold on the
invariant subspace.
\end{proposition}

\begin{proof}
As in Proposition~\ref{prop:matrixdouble}, with $d\cdot s=\Gamma_0$. The reflection through each half centre is a symmetry of the
data, and $\sin\phi$ is odd under it.
\end{proof}

\paragraph{Source and gauge count.} For $\gamma\geq1$ the reflection-invariant problem retains exactly
$\gamma+\lfloor(b-1)/2\rfloor$ source and gauge pairs:
\begin{itemize}[nosep]
\item a reflection-fixed full centre keeps its radial mode;
\item a pair of full centres keeps two components;
\item a pair of half centres keeps one.
\end{itemize}
This equals the number of independent forces in Theorem~\ref{thm:finitedoubling}. In the stacking there are $E$ pairs, one
for each half bridge, and $E-1$ independent forces (Theorem~\ref{thm:finitestacking}); the remaining source is removed by
the axial Killing identity in Section~\ref{sec:nonlinear}.

\subsection{Transfer to the pregluing}

\begin{lemma}[Coefficient comparison]\label{lem:coefficient}
Let two immersion charts in coordinates $(v,\phi)$, divided by their local scale, have metric determinant bounded below and bounded
$C^{4,\alpha}$ jets, and let their jets differ by at most $\varepsilon_1$. Let their conjugating factors, which together with their inverses are bounded by powers of $s$, differ by at
most $\varepsilon_1$ through order two. Then their normalized conjugated Jacobi operators differ by $C\varepsilon_1$ as maps from
scaled $C^{2,\alpha}$ to $C^{0,\alpha}$. Their normalized conormal derivatives at a fixed smooth cut differ by $C\varepsilon_1$ as maps
to $C^{1,\alpha}$. The constants are polynomial in the jet bounds and the inverse ellipticity bounds.
\end{lemma}

\begin{proof}
The coefficients of $\Delta_gu+|A|^2_gu$ are rational functions of the metric and of the second fundamental form, hence of
the jets. The identity $g^{-1}-\tilde g^{-1}=g^{-1}(\tilde g-g)\tilde g^{-1}$ and the product estimates in $C^{0,\alpha}$ give the
bound. The conormal on a cut $\{v=\text{const}\}$ is
$\pm(\sqrt{g^{vv}}\,\partial_v+g^{v\phi}(g^{vv})^{-1/2}\partial_\phi)$, which is treated in the same way.
\end{proof}

\begin{lemma}[The auxiliary conormal and the corner traces]\label{lem:conormal}
Take derivative traces on the pregluing as $R\,\partial_{\eta_A}$, with the outgoing auxiliary conormal $\eta_A$ and $R$ constant
on each cut. Then at a half cut the value traces of Neumann heights lie in $T^{2,\alpha}_{\rm H}$, the derivative traces lie
in $T^{1,\alpha}_{\rm H}$, and these are exactly the reference trace spaces. The Euclidean conormal does not have this
property.
\end{lemma}

\begin{proof}
In an adapted half strip $(v,\phi)$ near the edge $\phi=\pm\pi/2$, the edge is a geodesic of the induced auxiliary metric,
because $\Sph$ is totally geodesic for $g_A$. Hence $g_{v\phi}=0$ and $\partial_\phi g_{vv}=0$ on the edge, and therefore
$\partial_\phi g^{vv}=0$. For a Neumann height $u_\phi=0$ and $u_{v\phi}=0$ on the edge, so the $\phi$-derivative of the conormal
trace vanishes at the endpoints. On the reference sheet the normalized conformal derivative $-R\partial_r$ differs from the
auxiliary conormal by a factor $1+\Ostar(R)$ whose transverse derivative vanishes at the edge, and on a reference neck the
derivative is $\pm\partial_\sigma$. As an example of the failure for the Euclidean conormal, take the graph $z=0$ in the chart
$\Xi$ and $u=y$. On $\{r=R\}$ the normalized Euclidean radial derivative is $R\,\mathfrak d\sin\phi$, whose derivative at
$\phi=\pi/2$ is $R^2\neq0$. The normalized auxiliary one is $R\sqrt{(1+R^2/4)^2-R^2\cos^2\phi}\,\sin\phi$, with zero endpoint
derivative.
\end{proof}

\begin{proposition}[Operator comparison]\label{prop:operator}
There are identifications of the pieces of $M$ with the reference pieces such that
\[
 \|\cA_M-\cA_0\|\lst\sqrt\tau
\]
on exactly the same broken spaces, uniformly on the force ball. The identifications preserve the Neumann edges and their
transverse normal lines, and they depend continuously on $F$.
\end{proposition}

\begin{proof}
\emph{Necks.} Recentre each bridge with Lemma~\ref{lem:recenter} and place the reference catenoid in the new frame. The
centre displacement is absorbed in the frame and not counted as an error. By Lemma~\ref{lem:bridgeerror}(1), the profile
bounds and the matching estimate $\lst\tau^2$ at $r\asymp R$, the neck immersion, divided by its scale, differs from the
reference by $\Ostar(r)$ in $C^{4,\alpha}$. The ambient chart $\Xi$ and the conjugating factor $\Omega$ contribute normalized
coefficient differences $\Ostar(r)$: the Euclidean immersions $\Xi\circ Y$ and $e_1+Y$ differ by $O(|Y|^2)$, and
$\Omega=1-x+O(|Y|^2)$ by Lemma~\ref{lem:chart}(5). Lemma~\ref{lem:coefficient} gives an operator difference
$\lst R\lst\sqrt\tau$.

\emph{Cuts.} Project the actual plateau cut to the base by the auxiliary normal projection $\Pi$ and read it in the
marked conformal coordinate $\mathcal Z$ of the base. With $e_r=(\cos\phi,\sin\phi)$, on the actual cut
\[
 \mathcal Z\circ\Pi\circ\Xi\bigl(R\cos\phi,R\sin\phi,B_e\pm v_e\bigr)=Re_r+\delta\mathcal Z(\phi),\qquad \|\delta\mathcal Z\|_{C^{4,\alpha}}\lst R^2+\tau .
\]
Here $R^2$ is the Taylor remainder of the tangent-to-conformal coordinate change and $\tau$ comes from the heights; at a
full centre $\Xi$ is omitted. At a half edge $\delta\mathcal Z_x(\pm\pi/2)=0$, since the boundary is preserved, and
$\partial_\phi\delta\mathcal Z_y(\pm\pi/2)=0$, by Lemma~\ref{lem:projection}. The map $(r,\phi)\mapsto re_r+\chi_{\rm col}(r/R)\delta\mathcal Z(\phi)$ on a
collar of width $R$ is a diffeomorphism. It has relative scaled size $\lst R+\tau/R\lst\sqrt\tau$, and it preserves the edge
and its transverse normal line exactly. The neck end is rectified in a unit collar in the same way, leaving the core
fixed.

\emph{Sheets.} Away from the cuts the outer sheets are auxiliary normal graphs with scaled height jets $\lst\tau$, so their
relative difference at scale $r\geq R$ is $\lst\tau/r\leq\tau/R$.

\emph{Ends.} On the end regions $\{t\geq t_{\rm end}\}$ no cut map acts. The height jets are $\Ostar(\tau(1+t))$ on unit
$(t,\theta)$-patches while the scale is $\lambda\asymp e^t$. The
coefficient difference, including its action on a constant end value, is $\lst\tau(1+t)^Ne^{-t}$, which is bounded after
multiplication by $e^{\delta t}$.

\emph{Data.} Value jumps are defined on the identified cuts and agree exactly. Derivative jumps use the auxiliary
conormal; by Lemma~\ref{lem:conormal} the trace spaces agree. Sources and gauges are transported in the fixed core
coordinates. Alternating sheet orientations change incidence signs by fixed signs. Moving centres are identified by
cutoff translations in fixed conformal charts, which are tangential along the edge at half centres. Neck lengths are
identified with a fixed interval by maps equal to the identity on the core. Taking the supremum over the pieces proves
the estimate.
\end{proof}

\begin{lemma}[Continuity in the force parameter]\label{lem:continuity}
The pregluing $M(\tau,F)$, its identifications with the reference pieces, and the operator $\cA_M$, regarded as an operator
between the fixed broken spaces, depend continuously on $F$ in the force ball.
\end{lemma}

\begin{proof}
The finite data depend smoothly on $F$ (Theorems~\ref{thm:finitedoubling} and~\ref{thm:finitestacking}). The pole fields and
their regular parts depend smoothly on the positions of the centres away from the poles, and the bridges, profiles,
recentering maps and interpolations are explicit smooth functions of the data on fixed charts. The cut radius $R$ is fixed
on the whole force ball; in the doubling we use $\tau$ at $F=0$ and the waist comparison of
Theorem~\ref{thm:finitedoubling}(2). The moving centres are identified with fixed ones by cutoff translations in fixed
conformal charts, and the neck lengths by maps that are the identity on the core; both depend smoothly on the data.
Hence the coefficients of $\cA_M$ in the fixed charts depend continuously on $F$ in $C^{k,\alpha}$, and so does $\cA_M$ in operator
norm.
\end{proof}

\begin{theorem}[Actual inverse]\label{thm:actualinverse}
For small $\tau$, $\cA_M$ has a two-sided inverse $\cQ_M$ on the broken spaces with $\|\cQ_M\|\lst1$, depending continuously on
$F$. If the source is global, that is, its unnormalization $q/(\Omega_M\mathfrak r^2)$ is $C^{0,\alpha}$ across the cuts, and the jumps
are zero, then the solution is a global $C^{2,\alpha}$ function on $M$ with the Neumann condition, with the same bound.
\end{theorem}

\begin{proof}
By Theorem~\ref{thm:reference} and Propositions~\ref{prop:matrixstack}--\ref{prop:matrixzero}, $\|\cQ_0\|\lst1$, so
Proposition~\ref{prop:operator} gives $\|(\cA_M-\cA_0)\cQ_0\|<\frac12$ for small $\tau$, and
\[
 \cQ_M=\cQ_0\bigl[\Id+(\cA_M-\cA_0)\cQ_0\bigr]^{-1},\qquad \|\cQ_M\|\leq2\|\cQ_0\| .
\]
Continuity in $F$ follows from Lemma~\ref{lem:continuity} and the resolvent identity. For a global source with zero jumps the solution has matching
values and conormal derivatives across each cut. Its tangential derivatives agree by differentiating the common value. In
a common collar chart, the second tangential derivative agrees by differentiating twice, the mixed derivative by
differentiating the conormal equation, and the second normal derivative by the common equation, whose principal
coefficient is positive. The piecewise $C^{2,\alpha}$ jets therefore match; splitting the distance between two points at the cut
gives the global H\"older bound. The functions whose jets match across the cuts form a closed subspace $\cX_M$ of the broken
height space, and $\cQ_M$ maps global sources with zero jumps into it.
\end{proof}

\begin{remark}\label{rem:continuity}
The subspaces $\cX_M$ depend on $F$. In the nonlinear argument we therefore work in the common broken spaces, where all
operators depend continuously on $F$, and we use that the iteration preserves $\cX_M$ for each fixed $F$.
\end{remark}

\section{The nonlinear problem and the forces}\label{sec:nonlinear}

For a height $u$ in the broken height space let $M_u$ be the auxiliary normal exponential graph of $u$ over $M$, taken piece by
piece. Its normalized mean curvature is
\[
 \mathbf H_F(u)=\Omega_M\mathfrak r^2H(M_u),
\]
with $\Omega_M$ and $\mathfrak r$ fixed on the initial surface. Then $\mathbf H_F(u)=\mathbf H_F(0)+\widehat\cL_Mu+\mathcal R(u)$, where $\widehat\cL_M$ is the
operator of $\cA_M$.

\begin{lemma}[Nonlinear remainder]\label{lem:nonlinear}
On the ball $\|u\|=o(\tau)$, auxiliary graphs of Neumann heights have exact exterior free boundary, and
\[
 \|\mathcal R(u)-\mathcal R(u')\|_{\cY}\lst\tau^{-1}\bigl(\|u\|+\|u'\|\bigr)\|u-u'\| .
\]
\end{lemma}

\begin{proof}
The boundary statement is Lemma~\ref{lem:auxgraph}(2). On a patch of scale $r$ the normalized curvature is $r$ times a
smooth function of the rescaled immersion and of the rescaled height jets $u/r$. Its second derivative in the heights is
therefore bounded by $C/r$, also in H\"older norm. The smallest scale is comparable to $\tau$. On the ends $r\asymp e^t$, and
the bound decays like $e^{-t}$, also for constant heights; since $\delta<1$ the weighted norm obeys the same estimate.
\end{proof}

\begin{proposition}[Projected solution]\label{prop:projected}
For every $F$ in the force ball there is a unique pair $(u(F),\mathbf z(F))$ in the ball $\{\|u\|+|\mathbf z|\leq Cs^{N}\tau^2\}$ such that
\[
 \mathbf H_F(u)=\Psi\mathbf z,\qquad \Lambda u=0,
\]
and $u(F)$ is a global height in $\cX_M$. The pair depends continuously on $F$ in the common broken spaces.
\end{proposition}

\begin{proof}
Write the equation as $(u,\mathbf z)=-\cQ_M(\mathbf H_F(0)+\mathcal R(u),0,0,0)$, with zero jumps and gauges. By
Proposition~\ref{prop:initialerror}, Theorem~\ref{thm:actualinverse} and Lemma~\ref{lem:nonlinear} the right-hand side maps a
ball of radius $Cs^N\tau^2$ into itself with Lipschitz constant $\lst\tau<\frac12$. For fixed $F$, start the iteration at zero.
If the iterate is global, the curvature source is global and Theorem~\ref{thm:actualinverse} returns a global height; since
$\cX_M$ is closed, the fixed point is global. All maps depend continuously on $F$ in the common broken spaces
(Proposition~\ref{prop:operator} and Remark~\ref{rem:continuity}), and the uniform contraction gives continuity of the fixed
point. No derivative of the inverse in $F$ is used.
\end{proof}

The surface $M_{u(F)}$ is minimal except for the supported sources $\Psi\mathbf z(F)$ in the waist cores. It remains to choose $F$
so that $\mathbf z(F)=0$.

\subsection{Killing fluxes}

Let $K$ be a Killing field of $\R^3$. For an oriented curve $\mathfrak c$ in a surface, with conormal $\eta$, put
$\mathrm{Flux}_K(\mathfrak c)=\int_{\mathfrak c}\langle K,\eta\rangle\,ds$. If $\mathcal D$ is a compact domain in a surface with mean curvature
$H$ and unit normal $\n$ (with the convention of Definition~\ref{def:efbms}), then $\operatorname{div}_{\mathcal D}K^\top=H\langle K,\n\rangle$, and
\begin{equation}\label{eq:killing}
 \mathrm{Flux}_K(\partial\mathcal D)=\int_{\mathcal D}H\langle K,\n\rangle\,dA .
\end{equation}
At a full centre we use the two translations of the tangent plane. At a half centre we use $K=c^{-1}e_3\times x$, the
rotation about the axis of the base annulus divided by the boundary radius. This rotation is tangent to $\Sph$, so
$\langle K,\eta\rangle=-\langle K,x\rangle=0$ on every curve in $\Sph$ along which the surface meets $\Sph$ orthogonally. In the
chart of a half centre, before the recentering of Lemma~\ref{lem:recenter}, its pullback has $K_\parallel(0)=e_2$ and $K_z(0)=0$;
after the recentering $K_\parallel(0)=e_2+\Ostar(\tau)$ and $K_z(0)=\Ostar(\tau)$. For a graph $z=f(x,y)$ in a conformal chart with
factor $\mathfrak d_0^{-1}$ (with $\mathfrak d_0=\mathfrak d$ at a half centre and $\mathfrak d_0=1$ at a full centre), the flux current is
\begin{equation}\label{eq:current}
 \mathsf J(f)=\mathfrak d_0^{-2}\Bigl[\Bigl(\mathsf W\Id-\frac{Df\otimes Df}{\mathsf W}\Bigr)K_\parallel+\frac{Df}{\mathsf W}K_z\Bigr],
 \qquad \mathsf W=\sqrt{1+|Df|^2},
\end{equation}
and the flux through a cut is $\int\mathsf J(f)\cdot n_{\rm cut}\,ds$, where $n_{\rm cut}$ is the outward unit normal of the cut
in the chart.

\begin{lemma}[Cap fluxes]\label{lem:flux}
Let $\mathrm F_K(e)$ be the sum of the fluxes through the two plateau cuts $\{r=R\}$ of the bridge $e$, oriented outwards
from the cap $\{r\leq R\}$, for the graphs $B_e\pm v_e$. At a full centre the common displacement vanishes, $m_e=0$, because full
bridges occur only in the doubling, where the two sheets are $\mp u$. Then
\begin{align*}
 \mathrm F_K(e)&=-4\pi a_e\beta_e+\Ostar\Bigl(a^2R+\frac{a^4}{R^3}\Bigr)&&\text{(full)},\\
 \mathrm F_K(e)&=-2\pi a_e\beta_e+\Ostar\Bigl(a^2R+\frac{a^4}{R^3}+\frac{\tau a^2}R+\tau^2R^2\Bigr)&&\text{(half)},
\end{align*}
with $\beta_e$ the tangential component at a half centre. If the traces change by at most $\varepsilon_2$ in scaled $C^1$ norm,
with $\varepsilon_2/R=o(1)$, then either flux changes by $\lst R\varepsilon_2+a\varepsilon_2/R+\varepsilon_2^2/R$.
\end{lemma}

\begin{proof}
\emph{Full centre.} The exact minimal base has zero flux through the boundary of a filled disk. The sum for $B+v$ and $B-v$
is a central difference of the current~\eqref{eq:current}, whose quadratic term at the zero base gradient is
$(|Dv|^2\Id-2Dv\otimes Dv)K_\parallel$. On the cut $Dv=(a/R)n_{\rm cut}+\beta+\Ostar(aR)$ by Lemma~\ref{lem:gapexpansion}. The radial part integrates
to zero and the cross term to $-4\pi a\beta$. The current is even in the gradient, so replacing the base gradient
$DB=\Ostar(R)$ by zero changes the quadratic coefficient by $\Ostar(R^2)$; together with the remainder of $Dv$ this costs
$\Ostar(a^2R)$. The fourth-order remainder of the central difference is $\lst R(a/R)^4$.

\emph{Half centre with $m=0$.} The reflection $y\mapsto-y$ preserves the rotational base and the conformal factor and
reverses $K$. Hence an even graph has zero flux, and so does the pair $B\pm v^{\rm odd}$ for an odd function $v^{\rm odd}$. Split
$v=v^{\rm ev}+v^{\rm odd}$ into its even and odd parts. On the semicircle their gradients are $(a/R)n_{\rm cut}+\Ostar(a)$ and
$\beta e_2+\Ostar(aR)$, and their values are $\Ostar(a)$ and
$\Ostar(aR)$. At the centre $K_\parallel=e_2+\Ostar(\tau)$, $K_z=\Ostar(\tau)$ and $\mathfrak d_0=1$; the $\Ostar(\tau)$ deviations change the
quadratic coefficients by $\Ostar(\tau)$ and are covered by the term $\tau a^2/R$ below. The even-odd cross term dotted with $n_{\rm cut}$ is $-2a\beta/R$, which
integrates to $-2\pi a\beta$. Coefficient variation, the trace remainders and the terms involving values cost $\Ostar(a^2R)$, and
the fourth-order remainder $\lst a^4/R^3$.

\emph{The common displacement.} Adding $m$ changes the quadratic coefficients by $\Ostar(\tau)$, which costs $\Ostar(\tau a^2/R)$.
The middle graph has flux $\Ostar(\tau^2R^2)$ by~\eqref{eq:killing} on its filled half disk, whose spherical edge has zero flux.

\emph{Perturbation.} Along a trace $|\partial_f\mathsf J|\lst1$ and $|\partial_{Df}\mathsf J|\lst R+a/R$, because the gradient
derivative vanishes at the central base jet. A height change $\varepsilon_2$ changes the gradient by $\varepsilon_2/R$, and integration
over a curve of length $O(R)$ gives the bound; the quadratic term is $\lst\varepsilon_2^2/R$.
\end{proof}

For the corrected surface $M_{u(F)}$ cut the caps again at the horizontal radius $R$. This does not change their fluxes:
between the two cuts the surface is minimal, and at a half centre the intervening spherical edge has zero flux. The
corrected traces differ from the plateau traces by $\varepsilon_2\lst\tau^2$ in scaled $C^1$ norm. Define the \emph{actual force} by
\[
 \mathcal P_e(F)=-\frac{\lambda(t_e)}{4\pi a_e^2}\,\mathrm F_K(e)\ \text{(full)},\qquad
 \mathcal P_e(F)=-\frac{c}{2\pi a_e^2}\,\mathrm F_K(e)\ \text{(half)},
\]
computed on $M_{u(F)}$, for the bridges whose forces are independent in Theorems~\ref{thm:finitedoubling}
and~\ref{thm:finitestacking}. Here $t_e$ is the $t$-coordinate of the centre of $e$.

\begin{proposition}[Force map]\label{prop:forcemap}
$\sup_F|\mathcal P(F)-F|\lst\sqrt\tau$, and $\mathcal P$ is continuous.
\end{proposition}

\begin{proof}
By~\eqref{eq:forcegradient} and~\eqref{eq:forcegradientstack} the leading terms of Lemma~\ref{lem:flux} are $-4\pi a^2F/\lambda$ and $-2\pi a^2F/c$. Divide all
errors by $a^2$, using $a\asymp\tau$, $R\asymp\sqrt\tau$ and $\varepsilon_2\lst\tau^2$:
\[
\begin{array}{c|ccccccc}
 \text{term}&a^2R&a^4/R^3&\tau a^2/R&\tau^2R^2&R\varepsilon_2&a\varepsilon_2/R&\varepsilon_2^2/R\\ \hline
 \text{power of }\tau\text{ after division}&\tfrac12&\tfrac12&\tfrac12&1&\tfrac12&\tfrac12&\tfrac32
\end{array}
\]
The factors $\lambda(t_e)$ and $c$ are bounded by powers of $s$. Continuity follows from Proposition~\ref{prop:projected}.
\end{proof}

\begin{lemma}[Forces determine the sources]\label{lem:pairing}
Write $\mathbf z_e$ for the source coefficients of the bridge $e$. On the corrected surface, $\mathrm F_K(e)=\mathsf P_e\mathbf z_e$, where
$\mathsf P_e$ is a matrix ($2\times2$ at a full centre and $1\times1$ at a half centre) with $\mathsf P_e=-\Id+\Ostar(\tau)$. In particular
$\mathbf z_e=0$ if and only if $\mathrm F_K(e)=0$.
\end{lemma}

\begin{proof}
Apply~\eqref{eq:killing} to the cap. The surface is minimal outside the waist core, and in the core $H=\Psi\mathbf z/(\Omega_M\mathfrak r^2)$.
The area element of the neck is $\mathfrak r^2d\sigma\,d\phi$, so the factor $\mathfrak r^2$ cancels. On the unit neck the normal components
of the two translations are $-\sech\sigma\cos\phi$ and $-\sech\sigma\sin\phi$, and the pairing of $p_{\rm F}\cos\phi$, $p_{\rm F}\sin\phi$ with
them is exactly $-\Id_2$, because $\int\chi\sech^2\sigma\,d\sigma=I_\chi$. At a half neck the pairing of $p_{\rm H}\sin\phi$ with
$-\sech\sigma\sin\phi$ over $[-\pi/2,\pi/2]$ is $-1$. On the corrected surface the geometry and these pairings differ by
$\Ostar(\tau+\|u\|/\tau)=\Ostar(\tau)$.
\end{proof}

\begin{lemma}[The ends]\label{lem:ends}
Let $F$ be in the force ball and $\Sigma_F=M_{u(F)}$. Each end of $\Sigma_F$ is, outside a compact set, a graph
\[
 x_3=\mu_k\log r+\zeta_k+O_{2,\alpha}(r^{-\delta})\qquad(r=|(x_1,x_2)|)
\]
over the complement of a disk in the horizontal plane, where $\mu_k$ is the growth of the corresponding end of the pregluing,
and the $\mu_k$ are pairwise distinct. If $\Sigma_F$ is minimal, its ends are regular catenoidal ends of multiplicity one, they are
parallel, and $\Sigma_F$ has finite total curvature.
\end{lemma}

\begin{proof}
On an end the pregluing is the graph of $\Omega f$ over $\mathcal C_c$, where $f$ is the sheet scalar ($f_k$ in the stacking,
$\mp u$ in the doubling). By the explicit fields of Section~\ref{sec:green}, on an end $f=A_k\Upsilon+B_kw+O_{2,\alpha}(e^{-t})$ with
constants $A_k,B_k=\Ostar(\tau s)$, and $\Upsilon=-dt+O(1)$; hence the growth of the end is $\mu_k=c\hc+\Ostar(\tau s)$, which is nonzero. The correction has the form $u_\infty+O(e^{-\delta t})$, so it changes only the constant and the decaying
terms, and the growth coefficients are those of the pregluing. They are strictly ordered by
Theorem~\ref{thm:finitestacking}(5), by~\eqref{eq:outgoing} and by Proposition~\ref{prop:genuszero}. The
sources $\Psi\mathbf z$ are supported in the waist cores, so on the ends $\Sigma_F$ is minimal. Since $r\asymp e^t$, each end is a graph as
stated, and the minimal surface equation gives
\[
 \Delta x_3=\frac{Dx_3^{\top}D^2x_3\,Dx_3}{1+|Dx_3|^2}=-\mu_k^3r^{-4}+O(r^{-4-\delta}).
\]
Subtracting $-\mu_k^3/(4r^2)$ leaves a source with finite second moments. The expansion of exterior harmonic functions then
gives $x_3=\mu_k\log r+\zeta_k+O(r^{-1})$ with differentiated bounds. Hence $|A|=O(r^{-2})$ on the ends. If $\Sigma_F$ is minimal, its
total curvature is therefore finite. Distinct growths give embedded nested ends with parallel limiting planes.
\end{proof}

\begin{theorem}[Existence of the minimal surfaces]\label{thm:existence}
For every configuration of Theorems~\ref{thm:finitestacking}, \ref{thm:finitedoubling} and Proposition~\ref{prop:genuszero},
and every sufficiently small $\tau$, there is $F_*$ in the force ball with $\mathbf z(F_*)=0$. Thus $\Sigma=M_{u(F_*)}$ is minimal, it
is an exterior free-boundary surface, and $\|u(F_*)\|\lst\tau^2$.
\end{theorem}

\begin{proof}
In the dihedral genus-zero doubling (Proposition~\ref{prop:genuszero}) there are no forces and no sources, and
Proposition~\ref{prop:projected} already gives a minimal surface. In the other cases we argue as follows.
The force ball has radius at least $s^{-2}/C$: it is fixed in the stacking and of radius $d^2\geq s^{-2}/C$ in the doubling. By
Proposition~\ref{prop:forcemap}, the continuous map $F\mapsto F-\mathcal P(F)$ sends this ball into itself for small $\tau$, and
Brouwer's theorem gives $F_*$ with $\mathcal P(F_*)=0$. By Lemma~\ref{lem:pairing}, all source coefficients of the bridges with
independent forces vanish.

In the doubling with $\gamma\geq1$ the reflection-invariant spaces contain no other sources.

In the stacking one half bridge, $e=(1,1)$, remains. Apply~\eqref{eq:killing} with the axial rotation $K=e_3\times x$ to the
whole surface truncated at a large radius $r=|(x_1,x_2)|$ on every end. The boundary on $\Sph$ contributes nothing. On an end, which is a
graph $x_3=\mu_k\log r+\zeta_k+O(r^{-\delta})$ by Lemma~\ref{lem:ends}, the axial flux is a constant multiple of
$-\int_0^{2\pi}r\,\partial_rx_3\,\partial_\theta x_3\,\mathsf W^{-1}d\theta=O(r^{-\delta})$, which tends to zero. The only remaining
contribution is the pairing of the source of $(1,1)$ with $K$, which is a nonzero multiple of $\mathbf z_{(1,1)}$ by
Lemma~\ref{lem:pairing}. Hence $\mathbf z_{(1,1)}=0$. The boundary condition is exact by Lemma~\ref{lem:nonlinear}.
\end{proof}

\section{Embeddedness and the proof of the main theorems}\label{sec:proof}

\subsection{Separation of the sheets}

\begin{lemma}[A tube around the base]\label{lem:basetube}
For $c$ in the compact ranges used here there is $\epsilon_{\rm t}>0$ such that the auxiliary normal exponential map of
$\mathcal C_c$ is injective on $\{|v|<\epsilon_{\rm t}\lambda(t)\}$. All outer sheets lie in this tube for small $\tau$. Hence pointwise
ordering of auxiliary heights implies disjointness of the sheets.
\end{lemma}

\begin{proof}
On a compact part this is the usual tubular neighbourhood theorem, uniform over a compact family of embedded annuli, with
the product collar at the boundary. Far out $g_A=g_E$ and the normal map is
\[
 (t,\theta,v)\mapsto\bigl((\lambda(t)-v\sech t)\,e_r(\theta),\ \hc+c\hc(t-t_0)+v\tanh t\bigr),
\]
where $e_r(\theta)=(\cos\theta,\sin\theta,0)$. For a fixed horizontal radius $r_h$ we have $v=(\lambda-r_h)\cosh t$ and
$\partial_tx_3|_{r_h}=\cosh t\,(2\lambda-r_h)>0$ when $|r_h-\lambda|<\epsilon_{\rm t}\lambda$, since then $2\lambda-r_h>0$. Hence the map is injective on the ends. Choose an
overlapping compact annulus and decrease $\epsilon_{\rm t}$. The sheet heights are $\Ostar(\tau(1+t))$ outside the bridge disks, and
$(1+t)/\lambda(t)$ is bounded, so the sheets lie in the tube.
\end{proof}

\begin{lemma}[Ordering]\label{lem:ordering}
On the outer sheets and the plateaus, outside the disks $\{r\leq R\}$ around the centres, adjacent sheets are separated in
auxiliary height by at least $c_1\tau s$, for a constant $c_1>0$ depending on the configuration. On every end the sheets are
strictly ordered by their growth coefficients.
\end{lemma}

\begin{proof}
\emph{Stacking.} By Lemma~\ref{lem:radial}, $\Upsilon\geq1$ is increasing. The boundary fields $P_\vartheta$ are bounded above,
have a negative logarithmic singularity at their own point, and are bounded on the ends. Indeed, in terms of $t-t_0$ their
nonconstant part is $\log|1-e^{-(t-t_0)+\mathrm i(\theta-\vartheta)}|+O(1)$, since $(n+w)/\nu_n-1/n=O(n^{-2})$. Away from the pole disks
of radius $s^{-N_0}/C$ we therefore have $J_i=O(\tau\log s)$, while $h_i\Upsilon\geq h_i\geq c_2\tau s$. Near a pole of interface $i$
the half gap is $a\log(2r/a)+\mathbf1_{\rm H}ar\mathfrak a+\beta\cdot x+\Ostar(ar^2)$ by Lemma~\ref{lem:gapexpansion}. For $R\leq r\leq s^{-N_0}/C$
with $N_0$ large this is at least $\frac12as+O(a)$. Near a pole of an adjacent interface the field enters $v_i$
in~\eqref{eq:sheetfields} with the sign that increases the gap. Same-sheet separation excludes other singularities in these
neighbourhoods. On the ends $J_i$ stays bounded while $\Upsilon$ grows, and the growth difference is $-2dh_i>0$.

\emph{Doubling, $\gamma\geq1$.} Below the full centres the mean of $u/\tau$ is
$(\mathfrak b/d)\Upsilon+b\bar\kappa w/(2d)=w(2c\mathfrak b+\frac12b\bar\kappa)/d+O(1+T)$ by~\eqref{eq:wronskian}. Here
$2c\mathfrak b+\frac12b\bar\kappa\to\Gamma>0$, so the mean is at least $c_2\tau s$. Between the full centres the additional means are bounded, and above them~\eqref{eq:outgoing} keeps the
mean positive and increasing after division by $w$. The nonconstant parts are $O(\tau\log s)$ outside the pole disks; physical
and conformal distances differ by factors that are powers of $s$, since $\lambda(T)\lst1$. Inside the pole disks use
Lemma~\ref{lem:gapexpansion} as above.

\emph{Doubling, $\gamma=0$.} Here $U(t_0)=U_0\asymp as$ and $\partial_t(U/w)=-a\mathfrak b(1+\csch^2t)>0$ by
Proposition~\ref{prop:genuszero}, and the same argument applies.

Passing from Euclidean to auxiliary heights multiplies by $\Omega$, which is bounded below. The difference between the
plateau graphs and the auxiliary normal graphs near the cuts is $\lst\tau^2$. Everything is uniform on the force ball.
\end{proof}

In the stacking, bridges at nonadjacent interfaces may have overlapping projections onto the base, and the next lemma
compares their heights. At an interface $i$ the common height of the two incident sheets is
$m_i=\frac12(\alpha_i+\alpha_{i+1})\Upsilon+\frac12(J_{i-1}-J_{i+1})$, which is smooth at the centres of interface $i$. Near a
half centre $\Upsilon=1+O(r)$, and converting to a common vertical chart at distance $O(R)$ changes heights by
$O(\tau)+\Ostar(\tau sR+\tau^2)=o(\tau s)$, also when nonincident centres coincide.

\begin{lemma}[Height bands]\label{lem:bands}
Near a centre of interface $i$, relative to the common base, the bridge of interface $i$ lies in the vertical band
\[
 \bigl[\alpha_i+\tfrac12h_i,\ \alpha_i+\tfrac32h_i\bigr]+o(\tau s),
\]
and the sheet $k$, where it is present, lies in the band $[\alpha_k-\frac12h_{k-1},\ \alpha_k+\frac12h_k]+o(\tau s)$, with $h_0=h_\ell=0$.
\end{lemma}

\begin{proof}
Same-sheet separation makes $J_{i-1}$ and $J_{i+1}$ smooth and $O(\tau)$ near a centre of interface $i$, so its middle height
is $\alpha_i+h_i+o(\tau s)$. Each half of the bridge reaches at most
$a\arcosh(4\rho/a)+O(\tau^2)=\frac12as+O(\tau)=\frac12h_i+o(\tau s)$ above or below the middle height, and the interpolations obey
the same bound. For a sheet $k$, the field $J_{k-1}$ can lower its height by at most $a|\log R|+O(\tau)=\frac12h_{k-1}+o(\tau s)$, and
$-J_k$ can raise it by at most $\frac12h_k+o(\tau s)$; all other terms are $O(\tau)$.
\end{proof}

\begin{proposition}[Separation]\label{prop:separation}
No two pieces of $M$ that are not joined by a bridge intersect. Wherever their projections to the base are within
distance $O(R)$, their vertical separation is at least a fixed positive multiple of $\tau s$.
\end{proposition}

\begin{proof}
Bridges at equal or adjacent interfaces have disjoint projected neighbourhoods by same-sheet separation, since
$R\ll s^{-N_0}$. Recall $\alpha_j-\alpha_i=2\sum_{k=i}^{j-1}h_k$. For bridges at interfaces $i<j$ with $j\geq i+2$, the lower edge of
the band of $j$ minus the upper edge of the band of $i$ is
\[
 \tfrac12(h_i+h_j)+2\sum_{k=i+1}^{j-1}h_k+o(\tau s)>0 .
\]
For a sheet $k\geq i+2$ above the bridge $i$, and for a sheet $k\leq i-1$ below it, the corresponding differences are
\[
 \tfrac12h_i+2\sum_{j=i+1}^{k-2}h_j+\tfrac32h_{k-1}+o(\tau s),\qquad \tfrac32h_k+2\sum_{j=k+1}^{i-1}h_j+\tfrac12h_i+o(\tau s),
\]
both positive. Each $h_j$ is a fixed positive multiple of $\tau s$ to leading order, so the remainders are absorbed. Two outer
sheets are separated by Lemma~\ref{lem:ordering} and Lemma~\ref{lem:basetube}; this includes sheets $k$ and $k+j$ with $j\geq3$
over a deleted disk of an intermediate interface, where the bands of Lemma~\ref{lem:bands} apply. If the projected
neighbourhoods are disjoint with a margin comparable to $R$, horizontal separation suffices. Otherwise we use a common
chart and the bands, also when nonadjacent centres coincide. The doubling has two sheets and disjoint bridge
neighbourhoods, so only Lemma~\ref{lem:ordering} and the local bridge geometry are needed.
\end{proof}

\begin{theorem}[Embedded pregluing and a quantitative tube]\label{thm:tube}
The pregluing $M$ is embedded and lies in $\Ext$. Its auxiliary normal exponential map is injective on heights
$|v|<\tau s^{-N_3}/C$ for a fixed exponent $N_3$. Every auxiliary normal graph over $M$ whose height and scaled first derivative are $\lst\tau^2$ is
embedded, diffeomorphic to $M$, and, if its height satisfies the Neumann condition, an exterior free-boundary surface.
\end{theorem}

\begin{proof}
On $\{r\leq C_0a\}$ each bridge is a perturbation, of relative size $\Ostar(a)$ in $C^{2,\alpha}$, of a compact embedded piece of a
standard catenoid or half catenoid. It therefore has a normal neighbourhood of radius $c_3a$. On the tails the horizontal
radius is monotone, and each tail is a graph of small slope for large $C_0$. By Lemma~\ref{lem:bridgeerror}(2) the two tails are
separated by $2a\arcosh(r/a)+\Ostar(ar+a^2)\geq c_3a$. Interpolation collars are graphs that preserve the order and join the
incident sheets smoothly. Proposition~\ref{prop:separation} excludes all other intersections. The surface lies in $\Ext$ by
Lemmas~\ref{lem:chart}(2), \ref{lem:halfedge} and~\ref{lem:auxgraph}.

For the normal neighbourhood we use two facts.
\begin{itemize}[nosep]
\item \emph{Local charts.} Every point $p\in M$ has a scaled immersion chart $\mathcal U_p$ of intrinsic radius $\mathsf r_p\geq c_3\tau$, with
$C^{2,\alpha}$ bounds and inverse ellipticity bounded by powers of $s$. By the inverse function theorem the normal exponential
is injective on $\{(q,v):q\in \mathcal U_p,\ |v|<\mathsf r_1\}$ with $\mathsf r_1=\tau s^{-N_3}/C$; at a half edge the product reflection of $g_A$
gives the one-sided chart.
\item \emph{Chord-arc bound.} If $p,q\in M$ and $q\notin \mathcal U_p$, then the auxiliary distance between $p$ and $q$ is at least
$4\mathsf r_1$. Indeed, if $p$ and $q$ lie on pieces that are not joined by a bridge, their distance is at least $c_1\tau s$ by
Proposition~\ref{prop:separation}, or their projections to the base are at distance at least $R$. If they lie on the same bridge, either both lie in
its compact core, which is embedded with normal radius $c_3a$, or they lie on distinct tails, separated by $c_3a$, or on one
graphical tail, where intrinsic and ambient distances are comparable. If $p$ lies on a bridge and $q$ on an incident outer sheet,
then either both lie on the union of a graphical tail, the interpolation collars, the plateau and the outer sheet, which is
one graph of uniformly small slope over the base, so intrinsic and ambient distances are comparable; or $p$ lies in the compact
core, whose projection is at distance $\geq R/2\gg a$ from the outer sheet. On one outer sheet and on the ends, the sheets are
graphs of small slope over the base tube of Lemma~\ref{lem:basetube}, and distinct sheets are separated by $c_1\tau s$. We take
the charts $\mathcal U_p$ so large that they contain these graphical pieces within intrinsic distance $8\mathsf r_1$ of $p$.
\end{itemize}
If two normal segments of length $<\mathsf r_1$ starting at $p$ and $q$ meet, then $p$ and $q$ are at distance $<2\mathsf r_1$, so $q\in \mathcal U_p$,
and injectivity in $\mathcal U_p$ gives $p=q$. Hence the normal exponential is injective on $\{|v|<\mathsf r_1\}$. Since $\tau^2s^N=o(\tau s^{-N_3})$
for every $N$, a graph of size $\lst\tau^2$ lies in this tube and is diffeomorphic to $M$. The boundary statement is Lemma~\ref{lem:auxgraph}(2).
\end{proof}

\subsection{\texorpdfstring{Proof of Theorems~\ref{thm:main}, \ref{thm:stacking} and~\ref{thm:doubling}}{Proof of Theorems A, C and D}}

Let $\ell\geq3$, $b\geq1$, $\gamma\geq0$. Take the stacking data of Theorem~\ref{thm:finitestacking}. Theorem~\ref{thm:existence}
gives an exterior free-boundary minimal surface $\Sigma$ that is a normal graph over the pregluing $M$ of size $\lst\tau^2$. By
Theorem~\ref{thm:tube} it is embedded and diffeomorphic to $M$. By Proposition~\ref{prop:topology} it is connected, orientable
and of type $(\ell,b,\gamma)$. By Lemma~\ref{lem:ends} its ends are parallel catenoidal ends of multiplicity one, and its total
curvature is finite. As $\tau\to0$, away from the pole disks the sheet heights are $\Ostar(\tau s(1+t))$ and the correction is
$\Ostar(\tau^2)$. Moreover $|F_*|=|F_*-\mathcal P(F_*)|\lst\sqrt\tau$ by Proposition~\ref{prop:forcemap}, so the bridge
angles converge to $\vartheta^0_{i\mathsf a}$. Hence the surfaces converge in $C^{2,\alpha}$ on compact subsets of $\R^3$ that avoid the limit points
$X_{4/5}(t_0,\vartheta^0_{i\mathsf a})$ to $\mathcal C_{4/5}$ with multiplicity $\ell$, and smoothly by elliptic regularity of minimal
surfaces with free boundary. This is the description in Theorem~\ref{thm:stacking}; the convergence is not uniform on the ends. Elliptic regularity at the free
boundary makes $\Sigma$ smooth.

Let $\ell=2$, $b\geq1$. For $\gamma\geq1$ take the data of Theorem~\ref{thm:finitedoubling}, and for $\gamma=0$ those of
Proposition~\ref{prop:genuszero}. The same argument gives a surface of type $(2,b,\gamma)$ with the properties of
Theorem~\ref{thm:doubling}. This proves Theorem~\ref{thm:main}. \qed

\section{Remarks and questions}\label{sec:remarks}

\paragraph{Arrangements.} Theorem~\ref{thm:main} concerns topology only. The free-boundary condition prescribes both the
position and the conormal of $\partial\Sigma$, so a boundary curve determines the surface by unique continuation. Hence the
boundary curves are not free parameters: this does not exclude particular analytic curves, but a curve cannot be prescribed
arbitrarily. A meaningful refinement is combinatorial. The $b$ disjoint boundary
curves divide $\Sph$ into $b+1$ regions, whose adjacency graph is a tree with $b$ edges. In our surfaces all boundary curves
are thin loops near one circle of $\Sph$.

\begin{question}
Which trees occur for embedded exterior free-boundary minimal surfaces of a given type? For instance, are there examples
with $b=3$ in which the three boundary curves are nested?
\end{question}

For free-boundary minimal surfaces in the ball the analogous notion of a boundary graph was introduced by Chu and
Franz~\cite{ChuFranz2025}, who proved that such surfaces with isomorphic boundary graphs are isotopic, and discussed which
boundary graphs are known to occur, among them the star configurations of~\cite{KKMS2024}.

\paragraph{Immersed surfaces.} For immersed surfaces the ends need not be parallel. It would be interesting to know which
configurations of end directions and growths occur, beyond the balancing imposed by the axial and the other rotational
Killing fields.

\paragraph{Uniformity.} The gluing scale depends on the type: the doubling needs $d\to0$, and the stacking uses a fixed
hierarchy $\varpi_i=\mathfrak q^{i-1}$ chosen for each type. We do not know whether the surfaces of Theorem~\ref{thm:main} lie in
continuous families that leave the gluing regime. A degree or continuation argument based on a priori compactness might
answer this.

\paragraph{Comparison with the ball.} In the unit ball, gluing constructions such as~\cite{FolhaPacardZolotareva2017} double
the equatorial disk with many bridges, because the mean of the separation must be produced by the bridges. Here the
growing Jacobi field $\Upsilon$ of Lemma~\ref{lem:radial} takes over this role. At the critical value $c_*$ the field becomes
bounded, and the doubling of Theorem~\ref{thm:doubling} couples the criticality parameter $d$ with the bridge
scale.

\end{document}